\documentclass[12pt, letterpaper]{amsart}
\usepackage{amssymb,latexsym,graphics,enumerate,color}
\usepackage[margin=1in]{geometry}
\usepackage{slashed}
\usepackage{accents}
\usepackage{xcolor}

\title{Asymptotics of the radiation field for the massless Dirac--Coulomb system}
\author{Dean Baskin}
\author{Robert Booth}
\author{Jesse Gell-Redman}

\newtheorem{theorem}{Theorem}
\newtheorem{lemma}[theorem]{Lemma}
\newtheorem{proposition}[theorem]{Proposition}
\newtheorem{corollary}[theorem]{Corollary}
\theoremstyle{remark}
\newtheorem{definition}[theorem]{Definition}
\newtheorem{remark}[theorem]{Remark}

\newcommand{\dirac}{\slashed{\partial}}
\newcommand{\charge}{\mathbf{Z}}

\newcommand{\Ld}{{\mathcal L}}
\newcommand{\dop}{\Ld}
\newcommand{\tdop}{\widetilde{\dop}}
\newcommand{\dops}{\widehat{\dop}_\sigma}
\newcommand{\doph}{\widehat{\dop}_h}

\newcommand{\Pd}{{ P}}
\newcommand{\Ps}{\widehat{P}_\sigma}
\newcommand{\Ph}{\widehat{P}_h}

\newcommand{\pd}[1][]{\partial_{#1}}
\newcommand{\norm}[1]{\left\| #1\right\|}
\newcommand{\abs}[1]{\left\lvert{#1}\right\rvert}
\newcommand{\ang}[1]{\left\langle {#1} \right\rangle}

\newcommand{\bB}{\mathbf{B}}
\newcommand{\bJ}{\mathbf{J}}
\newcommand{\bL}{\mathbf{L}}
\newcommand{\bp}{\mathbf{p}}
\newcommand{\bR}{\mathbf{R}}
\newcommand{\br}{\mathbf{r}}
\newcommand{\bsigma}{\boldsymbol{\sigma}}
\newcommand{\bSigma}{\mathbf{\Sigma}}
\newcommand{\cB}{\mathcal{B}}
\newcommand{\cE}{\mathcal{E}}
\newcommand{\cH}{\mathcal{H}}
\newcommand{\cM}{\mathcal{M}}
\newcommand{\cV}{\mathcal{V}}
\newcommand{\cX}{\mathcal{X}}
\newcommand{\cY}{\mathcal{Y}}
\newcommand{\tG}{\widetilde{G}}
\newcommand{\tQ}{\widetilde{Q}}
\newcommand{\tR}{\widetilde{R}}
\newcommand{\tw}{\widetilde{w}}
\newcommand{\trho}{\widetilde{\rho}}
\newcommand{\Sigmadot}{\dot{\Sigma}}
\newcommand{\us}{\widetilde{u_\sigma}}
\newcommand{\fs}{\widetilde{f_\sigma}}
\newcommand{\uxi}{\underline{\xi}}
\newcommand{\ueta}{\underline{\eta}}
\newcommand{\usigma}{\underline{\sigma}}
\newcommand{\bOmega}{\mathbf{\Omega}}

\newcommand{\bl}{\mathrm{b}}

\newcommand{\CI}{C^\infty}
\DeclareMathOperator{\Diff}{Diff}
\newcommand{\Diffh}{\Diff_h}
\DeclareMathOperator{\liptic}{ell}
\newcommand{\Hh}{H_{h}}
\newcommand{\module}{\cM}
\DeclareMathOperator{\Op}{Op}
\DeclareMathOperator{\supp}{supp}
\newcommand{\WF}{\operatorname{WF}}
\newcommand{\conormal}[1]{I^{(#1)}}
\newcommand{\phg}{\mathcal{A}_{\mathrm{phg}}}

\newcommand{\Diffb}{\Diff_{\bl}}

\newcommand{\ellbh}{\liptic_{\bl,h}}
\newcommand{\Hb}{H_{\bl}}
\newcommand{\Psib}{\Psi_\bl}
\newcommand{\Psibh}{\Psi_{\bl,h}}
\newcommand{\Sb}{{}^{\bl}S}
\newcommand{\Sbstar}{\Sb^*}
\newcommand{\specb}{\operatorname{spec}_\bl}
\newcommand{\sigmab}{\sigma_\bl}
\newcommand{\sigmabh}{\sigma_{\bl,h}}
\newcommand{\Tb}{{}^\bl T}
\newcommand{\Tbstar}{\Tb^*}
\newcommand{\WFb}{\WF_{\bl}}
\newcommand{\WFbh}{\WF_{\bl,h}}

\newcommand{\scri}{\mathcal{I}}
\newcommand{\mf}{\operatorname{mf}}
\newcommand{\cf}{\operatorname{cf}}

\newcommand{\xt}{\bar{x}}
\newcommand{\rhot}{\bar{\rho}}
\newcommand{\ubar}[1]{\underaccent{\bar}{#1}}
\newcommand{\rhott}{\ubar{\rho}}
\newcommand{\xtt}{\ubar{x}}

\newcommand{\bo}{\mathcal{O}}
\newcommand{\ds}{\displaystyle}
\newcommand{\grad}{\nabla}
\newcommand{\Lap}{\Delta}
\DeclareMathOperator{\sgn}{sgn}
\newcommand{\complexes}{\mathbb{C}}
\newcommand{\integers}{\mathbb{Z}}
\newcommand{\naturals}{\mathbb{N}}
\newcommand{\reals}{\mathbb{R}}
\newcommand{\sphere}{\mathbb{S}}
\DeclareMathOperator{\id}{Id}

\newcommand{\tow}{\mathrm{ftr}}
\newcommand{\away}{\mathrm{past}}

\renewcommand{\Im}{\operatorname{Im}}
\renewcommand{\Re}{\operatorname{Re}}

\begin{document}

\begin{abstract}
  We consider the long-time behavior of solutions to the massless Dirac equation
  coupled to a Coulomb potential.  For nice enough initial data, we
  find a joint asymptotic expansion for solutions near the null and
  future infinities and characterize explicitly the decay rates seen
  in the expansion.

  This paper can be viewed as a successor to previous work on
  asymptotic expansions for the radiation
  field~\cite{BVW15,BVW18,BM19}.  The key new elements are propagation
  estimates near the singularity of the potential, building on work of
  the first author with Wunsch~\cite{BW20} and an explicit calculation
  with hypergeometric functions to determine the rates of decay.
\end{abstract}

\maketitle

\section{Introduction}
\label{sec:introduction}

We consider the long-time asymptotics of solutions of the massless
Dirac--Coulomb equation on $\reals \times \reals^{3}$.  We show that
if the initial data (or source term for the inhomogeneous equation) is sufficiently nice, the
solution has a complete asymptotic expansion near null and timelike
infinities.  The decay rates of the terms in the asymptotic expansion
are analogous to resonances and we compute them explicitly.  One way to
capture the leading order part of these asymptotics is through the
Friedlander radiation field, which is a rescaled restriction of the
solution to null infinity $\scri^{+}$.  We find, as in previous
work in related settings~\cite{BVW15,BVW18,BM19}, that the asymptotics of the radiation
field as the lapse
  function  $t-r=s \to \infty$ are given by
the resonance poles of a Dirac-type operator (in the sense that its square is
principally the Laplacian) on hyperbolic space.

The following theorem is the main result of this paper.  Notation
involving the Dirac equation will be explained in
Section~\ref{sec:dirac-coul-equat}.  Let $\eta$ denote the (mostly
plus) Minkowski metric on $\reals\times\reals^{3}$, whose coordinates
are $t=x^{0}$, $x^{1}$, $x^{2}$, $x^{3}$ and $\gamma^{\alpha}$ denote
the Dirac matrices, which satisfy $\gamma^{\alpha}\gamma^{\beta} +
\gamma^{\beta} \gamma^{\alpha} = -2\eta^{\alpha\beta}\id_{4}$.  We let
$r=r(x) = \left( (x^{1})^{2}+(x^{2})^{2}+(x^{3})^{2}\right)^{1/2}$
denote radius in the spatial coordinates.  

\begin{theorem}
  \label{thm:radfield}
  For $\charge \in \reals$, $\abs{\charge}<1/2$, consider the
  massless Dirac--Coulomb operator in $\reals \times \reals^{3}$:
  \begin{equation*}
    \dirac_{\charge/r} =  \gamma^{0} \left( \pd[t] +
      \frac{i\charge}{r}\right) +
    \sum_{j=1}^{3}\gamma^{j}\pd[j].
  \end{equation*}
    Let $\psi$ be the forward solution of $\dirac_{\charge/r}\psi = f$
  with $f\in \CI_{c}$. Then its Friedlander radiation field, given in terms of $s=t-r$ and
  $\theta \in \sphere^{2}$ by
  \begin{equation*}
    \mathcal{R}_{+}[\psi](s,\theta) = \lim_{r\to
      \infty}r^{1+i\charge}\psi(s + r, r\theta),
  \end{equation*}
  is a smooth function on $\mathbb{R}_s \times \mathbb{S}^2_\theta$.  

  Moreover, the radiation field admits an  asymptotic expansion as $s\to +\infty$:
  \begin{equation*}
    \mathcal{R}_{+}[\psi](s,\theta) \sim
    s^{i\charge}\sum_{j,k=1}^{\infty}s^{-j - \sqrt{k^{2}-\charge^{2}}} a_{jk}(\theta),
  \end{equation*}
  where $a_{jk}$ are smooth functions on $\sphere^{2}$.  
\end{theorem}

In fact, we prove the following somewhat stronger theorem, showing
that in fact $\psi$ is polyhomogeneous on a compactification $[M;S_+]$
of $\mathbb{R} \times \mathbb{R}^3$ which we describe in detail below
(the factor of $i$ in the index set is due to our convention described
below in Section~\ref{sec:mellin-transform}):
\begin{theorem}
  \label{thm:phg-full}
  If $\psi$ is the forward solution of
  $i\dirac_{\charge/r}\psi =g$ with $g\in \CI_{c}$, then $\psi$ is
  polyhomogeneous on $[M;S_{+}]$ with index sets
  \begin{equation*}
    \begin{cases}
      \emptyset & \text{at } C_{-}\cup C_{0},\\
      \{ (i -\charge + ik, 0) : k =0,1,2,\dots\} & \text{at }\scri^{+} ,\\
     \{ (i(1 + j + \sqrt{k^{2}-\charge^{2}}), 0) : j, k=1,2,\dots\} &\text{at }C_{+}.
    \end{cases}
  \end{equation*}
\end{theorem}
In particular, the leading order behavior near the light cone is
given by
\begin{equation*}
  \frac{1}{(t+r)^{1+i\charge}(t-r)^{1+\sqrt{1-\charge^{2}}}} \text{ as
  }t,r\to +\infty.
\end{equation*}

The apparent difference between the leading order behavior of
$\mathcal{R}_+[\psi]$ as $s \to \infty$ and $\psi$ as
$t, r \to \infty$ is an consequence of our definition of the radiation
field below in equation~\eqref{eq:actual def of rad field}.  Indeed, as explained
there, $\mathcal{R}_+[\psi]$ is $\rho^{-1-i\charge}\psi$ restricted to
null infinity.  The expansion exponents of $\mathcal{R}_+[\psi]$ at
$C_+$ are therefore shifted in comparison to that of $\psi$ itself by
$1 + i \charge$.

A significant novelty of this paper is the \textit{explicit}
characterization of the decay rates of solutions off the light cone.
We compute these estimates by finding the ``resonant states''
associated to a family of operators at infinity.  The solutions are
given in terms of hypergeometric functions and we characterize the
exponents as poles of the inverse of this operator family.

Though we do not state it explicitly, one can also see that the
Friedlander radiation field in this context carries a particular
polarization; the forward radiation field always lies in the
$+1$-eigenspace of the Dirac matrix corresponding to Clifford
multiplication by $\pd[r]$.  (The backward radiation field lies in the
other eigenspace.)  This can be seen either through a modification of
the argument given in Section~\ref{sec:radiation-field} or explicitly
in terms of the hypergeometric functions of
Section~\ref{sec:char-expon}.

The proof, described in steps at the end of this section, follows the
same outline used to study the radiation field in other
settings~\cite{BVW15,BVW18,BM19}; arguments near the singularity of
the potential are modeled on the proof of the diffractive propagation
theorem for the Dirac--Coulomb system~\cite{BW20}.  The methods used
in this paper apply in higher dimensions (and, indeed, for conic Dirac
operators), but we specialize to the case of three spatial dimensions
for reasons of clarity and physical interest.  See
Section~\ref{sec:other-dimensions} for a discussion of the
higher-dimensional case.

The Dirac--Coulomb equation provides a model for spin-$\frac{1}{2}$
particles in the presence of a point charge $\charge$.  In the massive
setting, much of the literature about the system and its related
operators is focused on the characterization of the point spectrum.
In contrast, the massless case has purely continuous
spectrum\footnote{The essential self-adjointness of the Hamiltonian
  implies that the spectrum is entirely real; separation of variables
  and a simple ODE analysis shows that there are no $L^{2}$ solutions
  and hence no eigenvalues.} and so this description is insufficient
to characterize the asymptotic behavior of solutions of the
time-dependent equation.  Darwin~\cite{Da:28} used separation of
variables to characterize the generalized eigenfunctions of the
massive Hamiltonian; a similar approach applies to the massless case.
In principle, one could derive our theorem by a careful analysis of
the special functions involved but it would be delicate and our
methods apply more generally.  Indeed, our methods also treat
certain perturbations of the equation; specifically, one could add a
compactly supported first order term (compact support implying also
support away from the singularity $r = 0$) or even a compactly
supported leading order term which does not change the large scale structure
of the characteristic set and in particular does not introduce trapping.

We further remark that the restriction that $\abs{\charge} < 1/2$ owes
to our repeated use of the Hardy inequality in the propagation
arguments.  We conjecture that the theorem holds (though the proof
would require considerably more care in the construction of the commutants) for
the entire range of charges for which the Hamiltonian is essentially
self-adjoint, i.e., $\abs{\charge} < \sqrt{3}/2$.  That the results
of Section~\ref{sec:char-expon} hold for this range can be viewed as
partial evidence for this view.

Interest in the massless Dirac--Coulomb system as an evolution
equation has also arisen in the community surrounding dispersive
equations.  That work has largely focused on proving dispersive and
Strichartz estimates for solutions both via separation of variables
and by treating the components as solutions of systems of coupled wave
equations.  Notable here is the work of D'Ancona and
collaborators~\cite{DaFa07,BoDaFa11,CaDa13}, the work of
Cacciafesta--S{\'e}r{\'e}~\cite{CaSe16}, and the work of Erdo{\u
  g}an--Green--Toprak~\cite{ErGrTo19}.  In many respects this paper is
complementary to that work, as we are also concerned with the global
decay of solutions of the equation, but use very different methods.

A natural follow-up question to this work is whether similar results
hold for smooth potentials decaying like $1/r$ at infinity (i.e.,
potentials of critical decay).  Although the methods employed here do
not directly apply to treat more general potentials, we expect that
they can be combined with an adaptation of Vasy's second
microlocalization~\cite{Vasy:LA1, Vasy:LA2} to obtain related
results.  We expect to treat this setting in future work.

\subsection{Notation}
\label{sec:notation-1}

Norms without subscript decorations are always the relevant $L^{2}$
norm.

We use the notation $s-0$ to denote $s-\epsilon$ for all $\epsilon >
0$.  Our convention is that the natural numbers include zero, i.e.,
\begin{equation*}
  \naturals = \{ 0, 1, 2, \dots\}.
\end{equation*}
Some further notation is as follows:
\begin{itemize}
  \item $M$ is the radially compactified spacetime with the singularity of
    the potential blown up;  $\mf$ is the boundary at spacetime
    infinity.  Section~\ref{sec:set-up}.
\item $L^2(\mathbb{R}\times \mathbb{R}^4)$ is the standard $L^2$ space
  on $\mathbb{R}^4$ for the Euclidean measure $dtdz$, while $L^2(M)$
  is the weighted $L^2$ space based on the measure that is ``b'' at
  space time infinity.  Section~\ref{sec:homogeneous-version}.
\item The Sobolev spaces we use most on the bulk are $\Hb^{k,m,l}(M) =
  \Hb^{k,m,l}$, $k \in \{ 0, 1\}$.  Here $k$ is the standard Sobolev
  regularity order, $m$ is the $\bl$-Sobolev regularity order, and $l$
  describes a weight.  Section~\ref{sec:inter-with-diff}.  
\item For distributions $u$, $\WFb^{1,m,\ell}u$ and $\WFb^{0,m,\ell}u$
  are the notions of wavefront sets associated to $H^{1,m,l}$ and
  $H^{0,m,l}$.  Section~\ref{sec:inter-with-diff}.
  \item The functions $s = s_{\tow}, s^* = s_{\away}$ on the sphere at infinity are the
    regularity functions on the boundary used to analyze the Mellin
    transformed normal operators.  Section~\ref{sec:vari-order-sobol}.
  \item $\cX^{s}$ and $\cY^{s}$ and their semiclassical versions
    $\cX^{s}_h$ and $\cY^{s}_h$ are variable order Sobolev spaces of
    distributions on the sphere at infinity.  Section~\ref{sec:vari-order-sobol}, same for the dual spaces $s^*$.
\end{itemize}

The space $\Psib^m(M)$ is the space of b-pseudodifferential
operators of order $m$ on $M$ (Section \ref{sec:bl-calculus}).  Below
we often let
$$
\{ A_1, \dots, A_n \} \Psib^m(M)
$$
denote the (right) module of operators generated by
$A_1, \dots, A_n \in \Diff^*(\mathbb{R}\times (\mathbb{R}^3 \setminus
\{ 0 \})$ over $\Psib^m(M)$.

\subsection{Outline of proof}
\label{sec:outline-proof}

The proofs of Theorems~\ref{thm:radfield} and~\ref{thm:phg-full}
follow the same general contours of analogous results for the scalar
wave equation on asymptotically Minkowski spaces~\cite{BVW15,BVW18}
and on cones~\cite{BM19}.  In particular, the analysis is somewhat
round-about and has five major steps, which we describe now.

\subsubsection*{Set-up}
\label{sec:set-up}

We define a compactification $M$ of
$\reals \times (\reals^{3}\setminus \{0\})$ to a manifold with
corners; this has the effect of ``resolving'' the singularity of the
potential and making the use of Melrose's $\bl$-calculus more natural.
We let $\rho$ be a boundary defining function for the main face $\mf$ of the
boundary (at infinity).  In the region of greatest interest, one can
take $\rho = t^{-1}$.  We then consider the equation
\begin{equation*}
  i\dirac_{\charge/r}\psi = g,
\end{equation*}
but then rescale and conjugate to rewrite it as 
\begin{equation*}
  \dop u = f,
\end{equation*}
where
\begin{align*}
  \dop &\equiv \rho^{-2-i\charge}\gamma^{0}i\dirac_{\charge/r}\rho^{1+i\charge}, \\
  u = \rho^{-1-i\charge} \psi &\in  C^{-\infty}(M),  \quad f = \rho^{-2-i\charge}g \in \CI_{c}(M^{\circ}).
\end{align*}
The rescaling is helpful because $\dop$ is then a ``$\bl$-differential
operator'' in the parlance of Melrose~\cite{Melrose:APS}, and it enables the use
of the $\bl$-pseudodifferential calculus to obtain microlocal
estimates on $u$ near $\mf$.

\subsubsection*{Propagation of $\bl$-regularity}
\label{sec:prop-bl-regul}

We prove the propagation of $\bl$-regularity, i.e., microlocalized
conormal regularity with respect to the main face $\mf$ starting at
the backwards null cone, where by hypothesis the solution is trivial
(as the solution is zero for $t$ sufficiently negative).  The aim is to propagate this regularity until we
reach the intersection $S_{+}$ of the forward null cone with the
boundary of $M$,
where the relevant bicharacteristic flow has radial points and so one
needs subtler estimates.  Most of this step is essentially contained
in prior work~\cite{BVW15,BVW18}, though the propagation through the
singularity of the potential is new; we adapt the argument used by the
first author to prove a diffractive theorem for the Dirac--Coulomb
system~\cite{BW20} here.  As the argument is so similar, we leave many
of the proofs in this step to an appendix.

\subsubsection*{Fredholm estimates}
\label{sec:fredholm-estimates}

We then use a strategy developed by Vasy~\cite{Vasy13} to show that we
may set up a global Fredholm problem on $\mf$ for the family of
``reduced normal operators'' $\dops$.  This is the family of operators
given by freezing coefficients at $\rho=0$ and then conjugating $\dop$
by the Mellin transform in $\rho$.  To obtain a Fredholm problem, we
view $\dops$ as acting on spaces with varying degrees of regularity,
with more regularity mandated at the backward end of the flow lines
than the forward end.  The family $\dops^{-1}$ then only has finitely
many poles in any given horizontal strip in $\complexes$ and satisfies
polynomial estimates as $\abs{\Re \sigma}\to \infty$.  To propagate
the estimates near the singularity of the potential on $\mf$, we rely
on a semiclassical version of the diffractive theorem in the bulk.

\subsubsection*{Asymptotic expansions}
This portion of the argument is essentially identical to the setting
of asymptotically Minkowski spaces~\cite{BVW15,BVW18}.  To begin the
asymptotic development of $u$ (and therefore $\psi$) near $\mf$, one
cuts off near $\mf$ and takes the Mellin transform to obtain a
$\sigma$-dependent family
of equations of the form
\begin{equation*}
  \dops \us = \fs,
\end{equation*}
where $\us$ is known to be analytic in a half-plane
$\Im \sigma \geq \varsigma_{0}$ by the propagation of
$\bl$-regularity.  Inverting $\dops$, one then obtains the global
meromorphy of $\us$.  Applying the inverse Mellin transform turns the
poles $\us$ into terms in an asymptotic expansion, with poles at
$\sigma = z$ of degree $k$ becoming a term
$\rho^{iz}(\log \rho)^{k-1}$.  The coefficients in this expansion are
functions on $\mf$ that become worse in their regularity at $S_{+}$ as
$\Im z$ decreases (so as we gain more decay in $\rho$).  To obtain the
full expansion as in Theorems~\ref{thm:radfield}
and~\ref{thm:phg-full}, one then tests via shifts of the scaling
vector field at $S_{+}$.

\subsubsection*{Identification of the exponents}
\label{sec:identification-poles}

Having established the polyhomogeneity of the solution, we then
explicitly identify the poles $\sigma$ of $\dops^{-1}$.  By changing
coordinates, we find an expression for $\dops$ in which the ODEs
obtained after separating variables become a family of hypergeometric
equations, which we solve explicitly near past and future infinity.

\subsection{Other dimensions}
\label{sec:other-dimensions}

We now briefly describe how the theorem changes in $(n+1)$-dimensions
for $n > 3$.  Our approach applies nearly verbatim as long as the
Hardy inequality applies, i.e., for $\abs{\charge} < \frac{n-2}{2}$.
In this setting we write the Dirac--Coulomb operator as
\begin{equation*}
  \dirac_{\charge/r} = \gamma^{0}\left((\pd[t] + i
    \frac{\charge}{r}\right) + \gamma^{r}\left( \pd[r] +
    \frac{n-1}{2r}\right) + \frac 1r D_{S}.
\end{equation*}
Here the spinor bundle on $\reals\times\reals^{n}$ is trivial and 
$2^{k+1}$-dimensional, where $k = \lfloor (n-1)/2\rfloor$.  The spinor
bundle on $\sphere^{n-1}$ is similarly trivial and
$2^{k}$-dimensional.  In fact, over the sphere, we can identify the spinor bundle on
$\reals\times \reals^{n}$ with two copies of the spinor bundle on
$\sphere^{n-1}$.  The operator $D_{S}$ then enjoys the property that
$\gamma^{r}D_{S}$ acts as the spherical Dirac operator
$\dirac_{\sphere^{n-1}}$ on one copy and $-\dirac_{\sphere^{n-1}}$ on
the other.  The eigenvalues of $\gamma^{r}D_{S}$ (and hence $D_{S}$)
are then given by (see, e.g., B{\"a}r~\cite{Bar})
\begin{equation*}
  \pm \left( \frac{n-1}{2} + k\right), \quad k \in \naturals_{0}.
\end{equation*}
To prove the analogous theorem in higher dimensions, we then consider
\begin{equation*}
  \mathcal{L} = \rho^{-1- \frac{n-1}{2} - i \charge\gamma^{0}}i
    \dirac_{\charge/r}\rho^{\frac{n-1}{2} + i \charge},
\end{equation*}
and proceed as in the three-dimensional case.  Although the
propagation results for finite $t$ and $n >3$ are not in the
literature, the same proof as in $n=3$ applies and yield
Theorem~\ref{thm:phg-full} in $(n+1)$-dimensions, $n> 3$, with the
following index sets:
\begin{align*}
  \begin{cases}
    \emptyset & \text{at }C_{-}\cup C_{0},\\
    \left\{ \left( i \left(  \frac{n-1}{2} + i \charge +  k\right), 0 \right) : k = 0, 1, 2,
      \dots\right\} & \text{at }\scri^{+}, \\
    \left\{ \left(i \left( \frac{n+1}{2} + j + \sqrt{\left( \frac{n-1}{2} +
            k\right)^{2}- \charge^{2}}\right), 0\right) : j,k = 0, 1,
    2, \dots \right\} & \text{at }C_{+}.
  \end{cases}
\end{align*}

In $(2+1)$-dimensions, however, our basic approach fails due to the
failure of the Hardy inequality.  Indeed, the Dirac--Coulomb system
fails to be essentially self-adjoint for $\charge \neq 0$.  We
conjecture that a similar theorem could hold with
$\abs{\charge} < 1/2$ for the distinguished self-adjoint extension
found by Schmincke~\cite{Schmincke} (see also
W{\"u}st~\cite{Wust1,Wust2}, Nenciu~\cite{Nenciu} and
Klaus--W{\"u}st~\cite{Klaus-Wust}).  The failure of the Hardy
inequality in two dimensions, however, suggests that the propagation
arguments require significantly more care.  The authors believe that
they could be made to work by appealing to the ``very basic''
operators of Melrose--Vasy--Wunsch~\cite[Section 10]{MVW08} but this
discussion would take us too far afield.

\subsection{Structure of the paper}
\label{sec:structure-paper}

Section~\ref{sec:dirac-coul-equat} fixes our notation and conventions
for the massless Dirac--Coulomb system.  In
Section~\ref{sec:bl-geometry} we introduce the relevant
compactifications of our space time and recall some facts about
Melrose's $\bl$-calculus and its relationship with standard
differential operators.  Collected in Section~\ref{sec:analyt-prel}
are a number of results and definitions used in the main analysis.

The remainder of the paper is devoted to the proof of the theorems:
Section~\ref{sec:propagation-bulk} describes the needed results in the
bulk of the spacetime, while Section~\ref{sec:bdry-op} is devoted to
the proof that the operator on the boundary is Fredholm with finitely
many poles in any strip.  In Section~\ref{sec:polyhomogeneity} we show
the polyhomogeneity of the solution.  Finally, in
Section~\ref{sec:char-expon} we find the exponents explicitly.

\subsection{Acknowledgements}
\label{sec:acknowledgements}

This material is based largely upon work done during the Fall 2019
semester program ``Microlocal Analysis'' at the Mathematical Sciences
Research Institute in Berkeley, California, supported by National
Science Foundation under Grant No. DMS-1440140, where DB was in
residence.  DB was partially supported by NSF CAREER grant
DMS-1654056.  JGR was partially supported by the Australian Research
Council through the Discovery Project grants DP180100589 and DP210103242.

\section{The Dirac--Coulomb equation}
\label{sec:dirac-coul-equat}

\subsection{Notation}
\label{sec:notation}

We use coordinates $x^{\alpha}$, $\alpha = 0,\dots, 3$ on
$\reals\times \reals^{3}$.  When referring to spatial coordinates
(i.e., indices $1$, $2$, $3$) we use Latin rather than Greek
subscripts and superscripts.  When convenient we use $t=x^{0}$ and
spatial polar coordinates
$r\in (0,\infty)$, $\theta \in \sphere^{2}$.  (In
Section~\ref{sec:compactifications} we describe coordinate systems
that are better adapted to use ``near infinity''.)

The Dirac operator on $\reals \times\reals^{3}$ is given by
\begin{equation*}
  \dirac = \gamma^{\alpha}\pd[\alpha],
\end{equation*}
where $\gamma^{\alpha}$ are the $4\times 4$ Dirac matrices
\begin{equation*}
  \gamma^{0} =
  \begin{pmatrix}
    I & 0 \\ 0 & -I
  \end{pmatrix},
\end{equation*}
and
\begin{equation*}
  \gamma^{j} =
  \begin{pmatrix}
    0 & \sigma_{j} \\ -\sigma_{j} & 0 
  \end{pmatrix},
\end{equation*}
and $\sigma_{j}$ are the $2\times 2$ Pauli matrices,
\begin{equation*}
  \sigma_{1} =
  \begin{pmatrix}
    0 & 1 \\ 1 & 0
  \end{pmatrix},
  \sigma_{2} =
  \begin{pmatrix}
    0 & -i \\ i & 0 
  \end{pmatrix},
  \sigma_{3} =
  \begin{pmatrix}
    1 & 0 \\ 0 & -1
  \end{pmatrix}.
\end{equation*}

The $\gamma^\alpha$ satisfy the anticommutation relation
\begin{equation*}
  \gamma^{\alpha}\gamma^{\beta} + \gamma^{\beta}\gamma^{\alpha} = -2 \eta^{\alpha\beta}\id_{4},
\end{equation*}
where $\eta^{\alpha\beta}$ are the components of the Minkowski metric,
i.e.,
\begin{equation*}
  \eta^{\alpha\beta} =
  \begin{cases}
    -1 & \alpha =\beta=0 \\
    1 & \alpha = \beta \in \{1,2,3\} \\
    0 & \alpha \neq \beta
  \end{cases}.
\end{equation*}

Given a charge $\charge \in \mathbb{R}$, we couple the Dirac operator to a Coulomb
electric potential of charge $\charge$ via the ``minimal coupling''
convention:
\begin{equation*}
  \dirac_{\charge/r} = \gamma^{0}\left(\pd[t] + i
    \frac{\charge}{r}\right) + \gamma^{j}\pd[j].
\end{equation*}

The operator $\dirac_{\charge/r}$ is not symmetric with respect to the
standard flat inner product on functions $\phi, \psi \colon
\mathbb{R} \times \mathbb{R}^3 \longrightarrow  \mathbb{C}^4$ given by
$\langle \phi, \psi \rangle = \int \langle \phi, \psi \rangle_{\mathbb{C}^4}
dtdx^1dx^2dx^3$.  (Here $\langle \cdot, \cdot'
\rangle_{\mathbb{C}^4}$ is the pointwise hermitian inner product on
$\mathbb{C}^4$.)  Indeed, $\gamma_0 \partial_0$ is antisymmetric while the
$\gamma_j \partial_j$ are symmetric.  On the other hand, by the
anticommutation properties of the Dirac matrices, the operators $i
\gamma_0 \dirac_{\charge/r}$ and $i\dirac_{\charge/r} \gamma_0$ are
easily checked to be symmetric with respect to this inner product.

We employ several other notational conventions.  In keeping with
physics notation (see, e.g., Rose~\cite{Rose}), we write
\begin{equation*}
  \beta = \gamma^{0},
\end{equation*}
and let $\alpha_{j}$ be defined by
\begin{equation*}
  \gamma^{j}= \beta \alpha_{j},
\end{equation*}
i.e.,
\begin{equation*}
  \alpha_{j} =
  \begin{pmatrix}
    0 & \sigma_{j} \\ \sigma_{j} & 0
  \end{pmatrix}.
\end{equation*}

In spherical coordinates, we require the radial versions of the
various matrices and so we set
\begin{equation}
  \label{eq:r-versions-of-matrices}
  \sigma_{r} = \sum_{j=1}^{3}\frac{x_{j}}{r}\sigma_{j} , \quad
  \alpha_{r}= \sum _{j=1}^{3}\frac{x_{j}}{r}\alpha_{j}, \quad
  \gamma^{r} = \sum_{j=1}^{3}\frac{x_{j}}{r}\gamma^{j}.
\end{equation}

\subsection{Separation of variables}
\label{sec:separation-variables}

In this section we use the convention that a boldface letter
(such as $\bsigma$ or $\br$) denotes the associated $3$-vector of
matrices or operators (such as $(\sigma_{1},\sigma_{2},\sigma_{3})$ or
$\frac{1}{r}(x^{1},x^{2},x^{3})$).  We also set
\begin{equation*}
  \bSigma =
  \begin{pmatrix}
    \bsigma & 0 \\ 0 & \bsigma
  \end{pmatrix}.
\end{equation*}

We let
\begin{equation*}
  \bL = \br \times \bp,
\end{equation*}
denote the orbital angular momentum operators, where, as is standard,
\begin{equation*}
  \bp =
  \begin{pmatrix}
    \frac{1}{i}\pd[x^{1}] \\
    \frac{1}{i}\pd[x^{2}] \\
    \frac{1}{i} \pd[x^{3}]
  \end{pmatrix}.
\end{equation*}
We then let
\begin{equation*}
  \bJ = \bL + \frac{1}{2}\bSigma
\end{equation*}
denote the total angular momentum operators (so orbital angular
momentum and spin together).  We now introduce Dirac's $K$ operator
and set
\begin{equation*}
  K = \beta (1 + \bSigma \cdot \bL).
\end{equation*}

The remarkable property of $K$ is the following lemma found in many
physics texts (e.g., Rose~\cite[Section 12]{Rose}).
\begin{lemma}
  The following operators are mutually commuting:
  \begin{equation*}
    \dirac_{\charge/r}, J^{2} = \bJ \cdot \bJ, J_{3}, K.
  \end{equation*}
Moreover,
  \begin{equation*}
    [\beta , K] = 0.
  \end{equation*}
\end{lemma}

In Section~\ref{sec:char-expon} below, we consider the action of a
rescaling of $\dirac_{\charge/r}$ on the common eigenfunctions of the
remaining operators in the lemma.  These eigenfunctions are described blockwise by
two component spinor spherical harmonics.  Following, e.g.,
Szmytkowski~\cite{Szmytkowski}, for $\theta \in \sphere^{2}$ we set
\begin{equation*}
  \Omega_{\kappa\mu} (\theta)=
  \begin{pmatrix}
    \sgn (-\kappa) \left( \frac{\kappa + 1/2-\mu}{2\kappa +
        1}\right)^{1/2}Y_{l, \mu - 1/2}(\theta) \\
    \left( \frac{\kappa + 1/2 + \mu}{2\kappa + 1}\right)^{1/2}Y_{l,
      \mu + 1/2}(\theta)
  \end{pmatrix},
\end{equation*}
where
\begin{align*}
  \kappa &\in \integers \setminus \{ 0 \}, \\
  \mu &\in \left\{ - \abs{\kappa} + 1/2, \dots,
        \abs{\kappa}-1/2\right\} , \\
  l &= \abs{\kappa + \frac{1}{2}} - \frac{1}{2},
\end{align*}
and $Y_{l,m}$ are the standard spherical harmonics.  The eigenvectors of $K$ are given by the span of 
\begin{equation*}
  \begin{pmatrix}
    \Omega_{\kappa\mu} \\ 0 
  \end{pmatrix},
  \quad
  \begin{pmatrix}
    0 \\ \Omega_{-\kappa\mu'}
  \end{pmatrix},
  \quad \mu, \mu' \in \{ - \abs{\kappa} + 1/2, \dots , \abs{\kappa}- 1/2\},
\end{equation*}
where both of these are understood to be $4$-vectors,
and the eigenvalue of $K$ on this eigenspace is $-\kappa$.

We further observe that
\begin{equation*}
  \alpha_{r}
  \begin{pmatrix}
    a \Omega_{\kappa\mu} \\ b \Omega_{-\kappa \mu'}
  \end{pmatrix}
  =
  \begin{pmatrix}
    -b \Omega_{\kappa \mu'} \\ -a \Omega_{-\kappa\mu}
  \end{pmatrix}.
\end{equation*}

The spherical Laplacian $\Lap_{\theta}$ and $K$ are related by
\begin{equation*}
I_{4}   \Lap_{\theta} = K^{2} - \beta K,
\end{equation*}
so that $\Lap_{\theta}$ commutes with $K$ and $\beta$, where $I_4$ is
the $4 \times 4$ identity matrix.

Writing
\begin{equation*}
  i \gamma^{0}\dirac_{\charge/r} = i \pd[t] - \cB, 
\end{equation*}
i.e.,
\begin{equation}\label{eq:stationary operator}
  \cB = \sum_{j=1}^{3}\frac{1}{i}\alpha_{j} \pd[j] + \frac{\charge}{r}.
\end{equation}
In polar coordinates, we then have
\begin{equation}
  \label{eq:polar-coords-def-of-B}
  \cB = -i \alpha_{r}\left( \pd[r] + \frac{1}{r} - \frac{1}{r}\beta
    K\right) + \frac{\charge}{r}.
\end{equation}

We recall for $\abs{\charge}< \sqrt{3}/2$, the operator $\cB$ is essentially
self-adjoint on $\reals^{3}$ with domain $H^{1}$.  Kato in his
book~\cite{Kato:book} established this result for $\abs{\charge}<1/2$
and Weidmann~\cite{Weidmann} later extended it to
$\abs{\charge}<\sqrt{3}/2$.  Beyond this range it is no longer
essentially self-adjoint.  Previous work~\cite{BW20} provided
another proof of this fact based on the structure of the indicial
operator of $\cB$.

\subsection{The radiation field}
\label{sec:radiation-field}

As our definition of the radiation field differs slightly from
Friedlander's~\cite{Friedlander}, we briefly recall its definition and construction.

Given a solution $\psi$ of $\dirac_{\charge/r}\psi = f$ with $f$ smooth and
compactly supported, we define the function
\begin{equation*}
  \varphi (s, \theta, \rho) = \rho^{-(1+i\charge)}\psi \left(s +
    \frac{1}{\rho}, \frac{1}{\rho}\theta\right),
\end{equation*}
i.e., a rescaling of $\psi$ written in terms of the coordinates $\rho
= 1/r$, $s = t-r$, and $\theta \in \sphere^{2}$, with $\sphere^{2}$
identified with the unit sphere in $\mathbb{R}^3$ so that $r\theta =
(x^1, x^2, x^3)$.

Because $\psi$ is a solution of $\dirac_{\charge/r}^{2}\psi = 0$ near
infinity, $\varphi$ is a solution of
\begin{equation*}
  \rho^{-(1+i\charge)}\left( 2(1 + i \charge)\rho \pd[s] -
    2\rho^{2}\pd[\rho]\pd[s] + \rho^{2}\Lap_{\theta} - \rho^{2}(\rho
    \pd[\rho])^{2} + \rho^{3}\pd[\rho] - \charge^{2}\rho^{2} + i
    \alpha_{r}\charge\rho^2\right)\rho^{1+i\charge}\varphi = 0
\end{equation*}
near $\rho = 0$.  Rewriting this equation yields
\begin{equation*}
  \rho^{2}\left( -2\pd[\rho]\pd[s] + \Lap_{\theta} - (\rho \pd[\rho] +
    1 + i\charge)^{2} + (\rho \pd[\rho] + 1 + i \charge) - \charge^{2}
    + i \alpha_{r}\charge\right)\varphi = 0.
\end{equation*}
In other words, $\varphi$ is the solution of a hyperbolic equation
that is non-degenerate near $\rho = 0$.  If $\psi$ vanishes
identically for $s \leq s_{0}$ (as is the case for the forward
solution,) 
then the argument of Friedlander~\cite[Section
1]{Friedlander} shows that $\varphi$ may be smoothly extended across
$\rho=0$.  In particular, $\varphi$ and its derivatives may be
restricted to $\rho=0$.

If $\psi$ is the forward solution of $i\dirac_{\charge/r}\psi = g$,
with $g$ smooth and compactly supported, we may therefore define the
(forward) \textit{radiation field} of $\psi$ by
\begin{equation*}
  \mathcal{R}_{+}[\psi](s,\theta) = \varphi (s,\theta, 0).
\end{equation*}

Note that our definition differs from Friedlander's original
construction in two important ways.  First, in our construction we
have conjugated by $\rho^{1+i\charge}$ rather than $\rho$ to account
for the additional oscillations introduced by the potential; this
modification is required to ensure that the initial-value formulation
of the radiation field is a translation representation of the
evolution semigroup.  Indeed, if $U(t)\psi_{0}$ is the solution
operator associated to the problem
\begin{equation*}
  \dirac_{\charge/r}\psi = 0, \quad \psi(0, x) = \psi_{0}(x),
\end{equation*}
and $\mathcal{R}_{+}(\psi_{0})(s,\theta) =
\mathcal{R}_{+}[\psi](s,\theta)$, then
\begin{equation*}
  \mathcal{R}_{+}(U(T)\psi_{0})(s,\theta) = \mathcal{R}_{+}(\psi_{0})(s+T,\theta),
\end{equation*}
i.e., the radiation field intertwines wave evolution and translation.

The second important difference is in the normalization of the
radiation field; Friedlander's construction includes a derivative to
ensure that the $L^{2}$ norm of the radiation field of a solution of
the wave equation is bounded by the energy of the initial data (indeed,
in that setting it is an isometry).  As the Dirac--Coulomb system is
first order, no derivative is warranted; it is straightforward to see
that
\begin{equation*}
  \norm{\mathcal{R}_{+}(\psi_{0})}_{L^{2}(\reals \times \sphere^{2})}
  \leq \norm{\psi_{0}}_{L^{2}(\reals^{3}\setminus \{0\})}.
\end{equation*}
The question of whether the map taking initial data to the radiation
field is an isometry is essentially a question of local energy decay
and is left to future work.

In Section~\ref{sec:compactifications} below, we realize the radiation
field as a rescaled restriction of $\psi$ to one face in a
compactification of our spacetime.

\section{$\bl$-geometry and the $\bl$-calculus}
\label{sec:bl-geometry}

\subsection{Compactifications}
\label{sec:compactifications}

As the operator $\dirac_{\charge/r}$ is singular at the spatial
origin, most
of the analysis to follow takes place on a compactification $M$ of $\reals
\times (\reals^{3}\setminus \{0\})$.  In particular, we treat
$\reals^{3}\setminus\{0\}$ as a conic manifold and compactify as in
previous work~\cite[Section 3]{BM19}.  Roughly speaking, we resolve
the singularity at the origin and consider the radial compactification
at infinity.

For clarity, we first discuss the setting where the underlying spatial
manifold is a half-line.  We compactify $\reals_{t}\times
(0,\infty)_{r}$ by stereographic projection to a (closed) quarter-sphere
$\sphere^{2}_{++}$ as depicted in Figure~\ref{fig:1-d-compact}.  The
map $\reals_{t}\times (0,\infty)_{r}\to \sphere^{2}\subset \reals^{3}$
given by
\begin{equation*}
  (t,r) \mapsto \frac{(t,r,1)}{\sqrt{1+t^{2}+r^{2}}}
\end{equation*}
sends $\reals \times (0,\infty)$ to the interior of the quarter-sphere
given by
\begin{equation*}
  \sphere^{2}_{++}= \{ (z_{1},z_{2},z_{3})\in \sphere^{2}\subset
  \reals^{3} \mid z_{2}\geq 0, z_{3}\geq 0\}.
\end{equation*}
Here $\sphere^{2}_{++}$ is a manifold with corners and
has two boundary hypersurfaces defined by $z_{2} = 0$ and
$z_{3} = 0$, respectively.  We let $\cf$ (or ``conic face'') denote the
hypersurface defined by
\begin{equation*}
  z_{2}= \frac{r}{\sqrt{1+r^{2}+t^{2}}} = 0,
\end{equation*}
while we use $\mf$ (or ``main face'') to denote the boundary
hypersurface defined by
\begin{equation*}
  z_{3} = \frac{1}{\sqrt{1+r^{2}+t^{2}}} = 0.
\end{equation*}

The above construction defines a smooth structure on the compactification
of $\reals \times (0,\infty)$.  We thereby obtain a compactification $M$ of $\reals
\times (\reals^{3}\setminus \{0\})$ is then given by
\begin{equation*}
  M = \sphere^{2}_{++}\times \sphere^{2}_{\theta},
\end{equation*}
where we use polar coordinates
$(r, \theta) \in (0,\infty)\times \sphere^{2}$ on
$\reals^{3}\setminus\{0\}$ and identify the interior of first factor
$\sphere^{2}_{++}$ with $t\in \reals$ and $r\in (0,\infty)$ via the
above construction.

\begin{figure}
  \centering
  \includegraphics{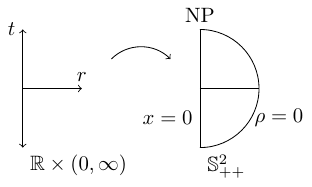}
  \caption{The compactification of $\reals \times (0,\infty)$ to
    $\sphere^2_{++}$}
  \label{fig:1-d-compact}
\end{figure}

Away from $\mf$, we use the coordinates $(t,r,\theta)$.  Near the
north pole 
(given by $(z_{1},z_{2},z_{3}) = (1,0,0)$), in lieu of the
boundary defining functions $z_{2}$ and $z_{3}$, it is convenient to
use the (homogeneous) functions
\begin{equation*}
  \rho = \frac{1}{t}, \quad x = \frac{r}{t}.
\end{equation*}
Similarly, near the south pole (given by $t < 0$, $z_{2}=0$,
$z_{3}=0$), we use
\begin{equation*}
  \rho = \frac{1}{\abs{t}}, \quad x = \frac{r}{\abs{t}}.
\end{equation*}
In the region between the two poles, (i.e., where $\abs{x} > 1$), we
choose our boundary-defining function $\rho$ so that it agrees with
the above definition in both polar regions and is homogeneous of degree
$-1$ in the scaling $(t, z) \mapsto (ct, cz)$ near $\mf$.  We extend $x$ to be smooth and strictly greater than
$1$ in this region.  

In our discussion of the radiation field, the smooth submanifolds
\begin{equation*}
  S_{\pm} = \{ \rho = 0, x = 1,\pm t > 0\} = \{ (z_1 = \pm 1/\sqrt{2}, z_2
  = 1/\sqrt{2} , 0) \} \times \sphere^2_\theta \subset \mf
\end{equation*}
play a crucial role.  These are defined by the functions $\rho$ and $v =
1-x$.  Lightlike geodesics on the interior all have limits at
$S_{\pm}$ in the future/past time directions. 

The complement of $S_{\pm}$ in $\mf$ consists of three open
components.  We denote by $C_{0}$ the region in $\mf$ where $x > 1$,
while the region where $x < 1$ has two components which we denote by
$C_{\pm}$ according to whether $\pm t > 0$ nearby.

The submanifold $S_{+}$ plays an additional role; in order to identify
the forward radiation field $\mathcal{R}_{+}$, we \textit{blow up}
$S_{+}$ in $M$ by replacing it with its inward pointing spherical
normal bundle, and in doing so we introduce a boundary hypersurface
$\scri^+$ which can be thought of as future null infinity and which
will serve as the domain of the radiation field.  Though we elide
detailed background on radial blow-ups, we briefly describe the
upshot here.\footnote{For more information about the blow-up
  construction, we refer the reader to Melrose's
  book~\cite{Melrose:APS}.}
This construction introduces a manifold with corners $[M; S_+]$ and a
``blow-down'' map
$$
\beta_{bd} \colon [M; S_+] \longrightarrow M
$$
such that $\scri^+ = \beta_{bd}^{-1}(S_+)$ is a boundary hypersurface
of $[M ; S_+]$ and ``cylindrical coordinates''
$$
\rho_{\scri^+} = (\rho^2 + (1-x)^2)^{1/2},\ \phi_{\scri^+} = (\rho /
\rho_{\scri^+}, 1-x / \rho_{\scri^+}), \ \theta \in \sphere^2_\theta
$$
give a smooth parametrization of a neighborhood of $\scri^+ =
\{\rho_{\scri^+} = 0\} \cap \{ t  > 0 \}.$  The map $\beta_{bd}$ is a diffeomorphism
from $[M; S_+] \setminus \scri^+$ to $M \setminus S_+$, i.e.\ the
construction only ``modifies'' $M$ near $S_+$.  The structure
of this manifold with corners depends only
on the submanifold $S_{+}$ and not on the particular choice of
defining functions $\rho$ and $v$, and
 in our setting, this is equivalent to
blowing up the point $(\frac{1}{\sqrt{2}}, \frac{1}{\sqrt{2}}, 0) \in
\sphere^{2}_{++}$ and then taking the product with $\sphere^{2}$.
See Figure~\ref{fig:blow-up}.  

The new space $[M;S_{+}]$ has four boundary hypersurfaces: the closure
of the lifts of the interiors of $C_{+}$ and $C_{0}\cup C_{-}$ by the
blow-down map, the lift of $\cf$, and a new boundary hypersurface
$\scri^+$ introduced by blow up.  By construction,
$\scri^{+}$ is naturally a fiber bundle over $S_{+}$ with fibers
diffeomorphic to intervals.  Indeed, given $v=1-x$ and $\rho$, the
fibers of the interior of $\scri^{+}$ in $[M;S_{+}]$ can be identified
with the cylinder $\reals_s \times \sphere^{2}$ by the coordinate $s =
v/\rho$.  

A simple computation (and the observation that, for fixed $s$,
$(1+s\rho)^{1+i\charge}\to 1$ as $\rho \to 0$) shows that for
solutions of $\dirac_{\charge/r}\psi = f$ with smooth, compactly
supported $f$, the Friedlander radiation field defined
above agrees with the restriction
\begin{equation}\label{eq:actual def of rad field}
  \mathcal{R}_{+}[\psi](s,\theta) = \rho^{-1-i\charge}\psi\vert_{\scri^{+}}.
\end{equation}

\begin{figure}
  \centering
  \includegraphics{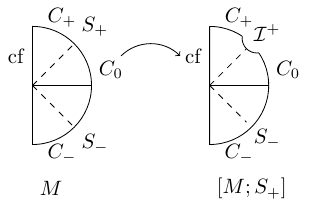}
  \caption{A schematic view of the forward radiation field blow-up.
    The lapse function $s$ increases along $\scri^{+}$ toward $C_{+}$.}
  \label{fig:blow-up}
\end{figure}

\subsection{The pseudodifferential $\bl$-calculus}
\label{sec:bl-calculus}

We describe now the homogeneous version of the $\bl$-calculus, which
we use in two separate cases.  We primarily use it in
our discussion of propagation on the compactification $M$ of bulk
spacetime to a manifold with corners.  We later use it briefly in our
discussion of the non-semiclassical aspects of the operator $\dops$ on
$\mf$.  In the discussion below, we therefore describe the
$\bl$-calculus on a manifold with corners $X$, though the explicit
examples are given only for $M$.

(Recall, briefly that a smooth
manifold with corners $X$ of dimension $m$ is locally diffeomorphic to
$\mathbb{R}_+^{k} \times \mathbb{R}^{m - k}$ and that $\partial X$ is
the union of the boundary hypersurfaces $\{ H_1, \dots, H_l \}$, which are
themselves manifolds with corners.  Given a particular boundary
hypersurface $H$ there exists a boundary defining function $\rho_H
\colon X \longrightarrow \mathbb{R}_+$, meaning $\rho_H$ is smooth,
non-negative, that $\{ \rho_H = 0 \} = H$, and $d\rho_H$ is
non-vanishing on $H$.)

To begin, recall the space of \textit{$\bl$-vector fields}, 
$\cV_{\bl}(X)$, that is, vector fields on $X$ which are tangent to the
boundary.  These are exactly those vector fields $V \in C^\infty(X;
TX)$, defined and smooth on the whole of
$X$, which over $\partial X$ point along the boundary, i.e.\ which
satisfy $V \rvert_{\partial X} \in C^\infty(\partial X, T(\partial X))$.  On $M$
these can be described easily enough; in a neighborhood of the
codimension $2$ corner $\mf \cap \cf$ they are generated over
$C^\infty(M)$ by the vector fields $\rho \pd[\rho], x \pd[x],$ and
$\pd[\theta]$, where here and below we abuse notation slightly by
allowing $\theta \in \sphere^2$ to denote local coordinates on
$\sphere^2$ (which can be accomplished locally on the sphere by
dropping on of the three components of $\theta$).   Thus still near
the corner,
$$
V = a(\rho, x, \theta) \rho \pd[\rho] + b(\rho, x, \theta)  x
\pd[x] + \sum_{k \in \{ 1, 2, 3\}} c_k (\rho, x, \theta)\pd[\theta_k].
$$
where the $\theta_k$ are whichever components of $\theta$ define local
coordinates on $\sphere^2$, and where smoothness on $M$ in this
neighborhood simply means the $a, b, c_k$ are smooth functions on $[0, 1)_\rho
\times [0, 1)_x \times \sphere^2$. Near $\mf$ but away
from $\mf \cap \cf$, they can be written in terms of by $\rho\pd[\rho]$
and the remaining coordinate vector fields (such as $\pd[x]$,
$\pd[\theta]$), while near $\cf$ away from $\mf$, it suffices to use
$r\pd[r]$, $\pd[t]$, and $\pd[\theta]$, always with coefficients which
are smooth up to the boundary.  It is straightforward to
check that $\cV_{\bl}$ is a Lie algebra, meaning for $V, W \in
\cV_{\bl}$, the commutator $[V, W] \in \cV_{\bl}$; its universal enveloping
algebra over $C^\infty(M)$ is (by definition) the algebra of
$\bl$-differential operators and is denoted
$\Diffb^{*}(X)$.  Near the codimension two corner $\mf \cap \cf$, an
operator $A\in \Diffb^{m}(M)$ has the form
\begin{equation}
  \label{eq:sample-diff-b-bulk}
  A = \sum_{|\alpha| + j + k \leq m}a_{jk\alpha}(\rho, x, \theta) (\rho
  D_{\rho})^{j}(x D_{x})^{k} D_{\theta}^{\alpha},
\end{equation}
where the coefficients $a_{jk\alpha} \in \CI(M)$.

The $\bl$-pseudodifferential operators $\Psib^{*}(X)$ are the
microlocalization of this algebra and formally consist of properly supported
operators of the form (for $\Psib^{*}(M)$ near $\mf \cap \cf$)
\begin{equation*}
  a(\rho, x, \theta, \rho D_{\rho}, x D_{x}, D_{\theta}),
\end{equation*}
where $a(\rho, x, \theta, \sigma, \xi, \eta)$ is a Kohn--Nirenberg
symbol.  

The space $\cV_{\bl}(X)$ is additionally the space of sections of the
$\bl$-tangent bundle $\Tb X$, which is a smooth vector bundle over
$X$ with the feature that for a given boundary hypersurface $H$ of
$X$ with boundary defining function $\rho_H$, the $\bl$-vector field $\rho_H
\pd[\rho_H]$ defines a \textit{non-vanishing} section of $\Tb X$ at $H$.
Its dual bundle is denoted $\Tbstar X$.  On $M$ near $\mf
\cap \cf$ it is locally spanned over $\CI (M)$ by $d\rho / \rho$, $dx /x$, and
$d\theta$ and we can write parametrize points in $\Tbstar M$ by writing

We may thus regard the symbols of operators in $\Psib^{*}(X)$ as
symbols on $\Tbstar X$, and the principal symbol map, denoted
$\sigmab$, maps the classical subalgebra of $\Psib^{m}$ to homogeneous
functions of order $m$ on $\Tbstar X$.\footnote{Recall that we can identify homogeneous functions on $\Tbstar X$ of
  a given order with smooth functions on $\Sbstar X$.  In an abuse of
  notation, we often view $\sigmab(A)$ as a smooth function on
  $\Sbstar X$.}  In the particular case of
$\bl$-differential operators on $M$, if $A$ is given as
above~\eqref{eq:sample-diff-b-bulk}, we have
\begin{equation*}
  \sigmab (A) = \sum_{\abs{\alpha}+j+k=m}a_{jk\alpha}(\rho, x, \theta) \sigma^{j}\xi^{k}\eta^{\alpha},
\end{equation*}
where $\sigma$, $\xi$, and $\eta$ are the canonical fiber coordinates
on $\Tbstar M$ defined by specifying that the canonical one-form is
given by
\begin{equation*}
  \sigma \frac{d\rho}{\rho} + \xi \frac{dx}{x} + \eta \cdot d\theta .
\end{equation*}

Up to this point all the operators we have discussed are
\textit{scalar}, meaning they act on functions.  To include operators on
vector valued functions one simply assumes that the coefficients
$a_{jk\alpha}$ above lie in $C^\infty(M; Mat_N)$ meaning they are
smooth functions values in $N \times N$ matrices, (for us typically $N
= 4)$.  It will be clear from context below whether we are considering
scalar or non-scalar operators.

As in previous work~\cite{BW20}, it is also convenient to identify a
subalgebra of $\Psib^{*}(X)$ essential for a commutator argument in
Section~\ref{sec:propagation-bulk}.  To do this, recall first that
there is an action of $SO(3)$ on the spatial variables; given $R \in
SO(3)$, then $R \cdot (t, x^{1}, x^{2}, x^{3}) = (t, R(x^{1}, x^{2},
x^{3}))$.  This induces a left action on functions, $R f = f \circ
R^{-1}$ and thus allows us to make the following definition.
\begin{definition}
  We say that a scalar pseudodifferential operator $A\in \Psib^{m}$ is
  \textit{invariant} if it is invariant with respect to the
  action of $SO(3)$ on functions, i.e.\ if for all $R \in SO(3)$ and
  all $f \colon C^\infty(M)$, $A (Rf) = R Af$.
\end{definition}
Any scalar symbol invariant under the lifted action of $SO(3)$ on
$\Tbstar M$ can be quantized to an invariant operator.  Invariant
operators commute with the angular operators $\Lap_{\theta}$ and $K$
\cite[Lemma 4]{BW20}.  

Accompanying the principal symbol map (which describes the leading
order behavior of elements of $\Psib^{*}(X)$ in terms of the
filtration), there is another collection of maps measuring the leading
order behavior of the operators at each boundary hypersurface.  In our
setting, we need only the map in each section: for $X=M$, we use the
map associated to $\mf$, while $X=\mf$ has only one boundary hypersurface.
Together with the principal symbol, these maps measure the obstruction
to compactness of $\bl$-operators.  We need this notion below only in
the case of $\bl$-differential operators, where it is simple to
describe.  This extra operator-valued symbol is the operator given by
freezing the coefficients of powers of $\bl$-vector fields at the
relevant boundary hypersurface.  In the case of $X= M$, if $A$ is
given by
\begin{equation*}
  \sum_{\abs{\alpha} + j + k \leq m} a_{jk\alpha}(\rho, x,
  \theta)(\rho D_{\rho})^{j}(x D_{x})^{k}D_{\theta}^{\alpha},
\end{equation*}
we then define the normal (or indicial) operator
\begin{equation*}
  N(A) = \sum _{\abs{\alpha} + j + k \leq m} a_{jk\alpha}(0, x,
  \theta) (\rho D_{\rho})^{j}(x D_{x})^{k} D_{\theta}^{\alpha}.
\end{equation*}
Recall that $N$ is a homomorphism and its conjugation by the Mellin
transform (described below in Section~\ref{sec:mellin-transform})
yields the reduced normal operator (also called the indicial family)
\begin{equation*}
  \widehat{N}(A) = \sum_{\abs{\alpha} + j + k \leq
    m}a_{jk\alpha}(0,x,\theta) \sigma^{j}(x D_{x})^{k}D_{\theta}^{\alpha}.
\end{equation*}
We then define the \textit{boundary spectrum of $A$} (in the case of
$X=M$; for $X=\mf$ we must change $\CI(\mf)$ to $\CI(\sphere^{2})$ below):
\begin{equation*}
  \specb (A) = \{ \sigma \in \complexes \mid \widehat{N}(A) \text{ is
    not invertible on }\CI(\mf)\}.
\end{equation*}
This set plays two important roles in our context: it is a key
ingredient in the identification of the domain of the essentially
self-adjoint Hamiltonian $\cB$ and, more centrally, it determines,
through its  relationship
with polyhomogeneity (described below), the exponents seen
in the asymptotic expansions of Theorem~\ref{thm:phg-full}.

Further associated to an operator $A \in \Psib^{m}(X)$ is its
microsupport
\begin{equation*}
  \WFb'(A) \subset \Sbstar M.
\end{equation*}
The microsupport is a closed subset and is the essential support of
the total symbol, just as in the usual pseudodifferential calculus.
It obeys the usual microlocality property
\begin{equation*}
  \WFb' (AB) \subset \WFb' (A) \cap \WFb'(B).
\end{equation*}
We also use the notion of $\bl$-ellipticity at a point, which is
equivalent (in the classical subalgebra) to the invertibility of the
principal symbol.  We postpone our discussion of the $\bl$-wavefront
set of distributions to later as we require a variant of it in our
estimates.

We also require a semiclassical version of the $\bl$-calculus on the
boundary hypersurface $\mf$.  We use $\Psibh^{*}(\mf)$ to denote this
space and refer the reader to previous work~\cite{BM19} and especially
to the excellent paper of Gannot--Wunsch~\cite[Section 3]{GW} for more
details.  Analogues of the constructions above exist for the
semiclassical calculus as well and are typically decorated with an $h$.

\subsection{Interaction with differential operators}
\label{sec:inter-with-diff}

The proofs of the propagation estimates near the singularity of the potential
rely on the understanding of the interaction between differential
operators and the $\bl$-calculus.  

\subsubsection{The homogeneous version}
\label{sec:homogeneous-version}

We let $L^{2}(M)$ denote the space of square integrable functions with
respect to a \textit{mixed $\bl$-metric density}, that is, 
a density with the metric-induced behavior near $\cf$ with the
$\bl$-induced behavior near $\mf$.  Concretely, if $\mu_{euc} =
dx^0dx^1dx^2dx^2$ denotes the Euclidean density, then setting 
$$
\mu = \rho^4 \mu_{euc},
$$
we set
\begin{equation}
L^2(M) = L^2(M, \mu) = L^2(\mathbb{R} \times \mathbb{R}^3 ; \rho^4
\mu_{euc}) = \rho^{-2} L^2(\mathbb{R} \times
\mathbb{R}^3),\label{eq:regular L2 space}
\end{equation}
where $L^2(\mathbb{R} \times \mathbb{R}^3)$ denotes the standard $L^2$
space.  Near $\mf\cap \cf$,
the density $\mu$ is given by
\begin{equation*}
  \frac{d\rho}{\rho}\, x^{2}\,dx\,d\theta .
\end{equation*}
As the (standard) Sobolev space $H^{1}$ is the domain of the various operators we
consider below, we use this space as the basis for the Sobolev spaces
on $M$.  We will use the space $\Hb^{1,0,0}(M)$
of distributions which are $H^1$ near the pole and $\Hb^1$ at $\mf$.
More precisely, $\Hb^{1,0,0}(M)$ is (by definition) equal to the
standard $\bl$-Sobolev space $\Hb^1(\overline{\mathbb{R}^4})$ where
$\overline{\mathbb{B}^4} = \overline{\mathbb{R}^4}$ is the radial
compactification and the $\bl$ refers to the behavior at $\partial
\overline{\mathbb{B}^4}$.  The identification of
$\Hb^1(\overline{\mathbb{R}^4})$ with a space of distributions on $M$
is realized by pullback via the map $M \longrightarrow
\overline{\mathbb{R}^4}$ which collapses the $\mathbb{S}^2_\theta$
factors over $\{x = 0\}$.

For
$\psi$ supported near $\mf \cap \cf$, the $\Hb^{1,0,0}$ norm can be
taken to be
\begin{equation*}
  \norm{\psi}_{\Hb^{1,0,0}}^{2} = \int \left( \abs{\rho \pd[\rho]
      \psi}^{2} + \abs{\pd[x]\psi}^{2} +
    \abs{\frac{1}{x}\grad_{\theta}\psi}^{2} + \abs{\psi}^2 \right)\, \frac{d\rho}{\rho}\,x^{2}\,dx\,d\theta.
\end{equation*}

For $m\geq 0$, we then let $\Hb^{1,m,0}(M)$ denote the Sobolev space of order $m$
relative to $\Hb^{1,0,0}(M)$, i.e., fixing $A\in \Psib^{m}(M)$
elliptic and invertible, we have $w\in \Hb^{1,m,0}(M)$ if $w\in \Hb^{1,0,0}(M)$ and
$Aw \in \Hb^{1,0,0}(M)$.  (This is independent of the choice of $A$.)
In particular, the $\Hb^{1,m,0}$ norm near $\mf\cap \cf$ is given by
\begin{equation*}
  \norm{\psi}_{\Hb^{1,m,0}}^{2}= \int \left( \abs{\rho \pd[\rho]
      A\psi}^{2} + \abs{\pd[x]A\psi}^{2} +
    \abs{\frac{1}{x}\grad_{\theta}A\psi}^{2} + \abs{A\psi}^{2}\right)
  \frac{d\rho}{\rho} x^{2}\,dx\,d\theta.
\end{equation*}
In the present manuscript we need only $m\geq 0$.  Finally, let
$\Hb^{1,m,l}(M) = \rho^{\ell}\Hb^{1,m,0}(M)$ denote the corresponding
weighted spaces.  A brief calculation in the region near the north
pole ($x < c < 1$) with $x=r/t$ and $\rho = 1/t$ shows that membership
in $H^{1}(\reals \times \reals^{3})$ is equivalent to membership in
$\Hb^{1,0,1}(M)\cap \rho^{2}L^{2}(M)$.

Note that away from the singularity, $\Hb^{1, m, l}$ regularity is equivalent to
$\bl$-regularity, meaning for $\chi(x)$ supported near $x = 0$,
$$
(1 - \chi(x)) \Hb^{1, m, l} \subset \Hb^{m + 1, l}(M).
$$
where $\Hb^{m', l}(M)$ is the ``standard'' $\bl-$Sobolev space.  To avoid
excessive notation below we avoid the notation $\Hb^{m', l}(M)$.

We now describe our microlocal characterization of regularity, the
wavefront set.  We define the notion only in the bulk $M$.  Although
it would be natural to define the analogous notions (both homogeneous
and semiclassical) on $\mf$, our propagation estimates are stated
explicitly in terms of operators in that region and so we omit those
definitions here.

On the bulk $M$, we use $\WFb^{1,m,\ell}$ to describe a failure to lie
in the space $\Hb^{1,m,\ell}$, while we use $\WFb^{0,m,\ell}$ to
describe the ``standard'' $\bl$-wavefront set with respect to the
underlying space $L^{2}(M)$ (with the metric volume form
$\mu$ above.)  In this direction we let $\Hb^{0, m, \ell}$ denote the
$\bl$-spaces with respect to $L^2(M)$, meaning $u \in \Hb^{0, m,
  \ell}$ if an only if for all $A \in \Psib^m(M)$, $\rho^{-\ell} Au
\in L^2(M)$.

\begin{definition}
  Suppose $u \in \Hb^{1,s,r}(M)$ for some $s,r$, and suppose $m,\ell
  \in \reals$.  We say $q\in \Sbstar M$ is \textit{not} in
  $\WFb^{1,m,\ell}(u)$ if there is some $A \in \rho^{-\ell}\Psib^{m}(M)$
  that is elliptic at $q$ and so that $Au \in \Hb^{1,m,\ell}$.

  Similarly, for $u \in \Hb^{0,s,r}$ and $m,\ell\in \reals$, we say
  that $q\in \Sbstar M$ is \textit{not} in $\WFb^{0,m,\ell}(u)$ if there is
  some $A\in \rho^{-\ell}\Psib^{m}(M)$ elliptic at $q$ with $Au \in
  L^{2}(M, \frac{d\rho}{\rho}\,x^{2}dx\,d\theta)$.
\end{definition}

Although the definition of $\WFb^{0,m,\ell}(u)$ is nearly the same as
that of the ``standard'' $\bl$-wavefront set, we keep the $0$ in the
notation as a reminder that the $L^{2}$ space is equipped with the
rescaled metric density.

Throughout the arguments in Section~\ref{sec:propagation-bulk} we rely
on the Hardy inequality, which allows us to estimate the $0$-th order
terms near the singularity of the potential.
\begin{lemma}
  \label{lem:bulk-hardy}
  If $u \in H^{1}(\reals^{n})$ with $n \geq 3$, then
  \begin{equation*}
    \frac{(n-2)^{2}}{4}\int \frac{\abs{u}^{2}}{r^{2}}\,dz \leq \int
    \abs{\grad u}^{2}\,dz.
  \end{equation*}
\end{lemma}
We use this inequality repeatedly in $\reals^{3}$ and its analogues
for distributions defined on $\mf$ near $x=0$ and for distributions
defined on $M$ near $\mf$, where it reads in both cases
\begin{equation*}
  \norm{x^{-1}u}\leq 2 \norm{\pd[x]u}.
\end{equation*}

The following lemma is essentially from~\cite[Lemma 8.6]{MVW08} (and similar
to~\cite[Lemma 2.8]{Vasy08}) will be used in the commutator
computations below.
\begin{lemma}
  \label{lem:comms-with-b-ops}
  If $A \in \Psib^{m}$ with principal symbol $a$, then
  \begin{equation*}
    \left[\frac{1}{x}, A\right] = C_{L}\frac{1}{x} = \frac{1}{x}C_{R},
  \end{equation*}
  where $C_{\bullet}\in \Psib^{m-1}$ with
  \begin{equation*}
    \sigmab(C_{\bullet}) = \frac{1}{i} \pd[\xi]a.
  \end{equation*}
  Moreover,
  \begin{equation*}
    \left[ D_{x}, A\right] = B + C_{L}D_{x},
  \end{equation*}
  with $C_{L}$ as above and 
  \begin{equation*}
    B \in \Psib^{m}, \quad \sigmab(B) = \frac{1}{i} \pd[\xi]a.
  \end{equation*}
\end{lemma}
\begin{proof}
  The proof uses standard tools from the $\bl$-calculus, and therefore
  we sketch only the main steps.  We discuss only the proof of the
  first statement, i.e.\ for the commutators with $1/x$, as the
  statement for $D_x$ follows exactly as in the references given.  

  Writing $[x^{-1}, A] = x^{-1} [A, x] x^{-1}$ and using the fact that
  $x \in \Psib^0(M)$ we obtain $[A, x] \in \Psib^{m-1}(M)$.  Using
  that $\sigma_{\bl, m-1}(i [A, x])= \{ A, x \}$, the Poisson bracket,
  and that in the coordinates $(\rho, x, \theta, \sigma,
  \xi, \eta)$ above the Hamilton vector field of a symbol $a$ is
  $$
(\pd[\sigma] a) \rho \pd[\rho] - (\rho \pd[\rho] a) \pd[\sigma]  +
(\pd[\xi] a) x \pd[x] - (x \pd[x] a) \pd[\xi] + \sum_k
(\pd[\eta_k] a) \pd[\theta_k] - (\pd[\theta_k] a ) \pd[\eta_k]
$$
we see that $\sigma_{\bl, m-1}(i [A, x]) = x \pd[\xi]a$.  It is now a
standard fact from the $\bl$-calculus that both $C_L = x^{-1}[A, x]$
and $C_R = [A, x] x^{-1}$ lie in $\Psib^{m-1}$ and have
$\sigma_{\bl, m-1}(C_\bullet) = x^{-1} \sigma_{\bl, m-1} ([A, x])$.  \end{proof}

It is also convenient to know we can microlocalize our estimates.  The
following lemma is essentially in previous work of the first author
with Wunsch~\cite[Lemma 9, Lemma 12]{BW20}.
\begin{lemma}
  \label{lem:elliptic-reg-bulk}
  If $A,G\in \Psib^{s}$ with $\WFb'(A)\subset \liptic G$, then for all
  $u$ with
  \begin{equation*}
    \WFb^{1,s,\ell}u \cap \WFb' G=\emptyset,
  \end{equation*}
  we may bound
  \begin{equation}\label{eq:elliptic bound bulk}
    \norm{Au}_{\Hb^{1,0,\ell}} \leq C\left( \norm{Gu}_{\Hb^{1,0,\ell}}
      + \norm{u}_{\Hb^{1,0,\ell}}\right).
  \end{equation}

      In particular, if $A \in \Psib^0$ then
      \begin{equation}
    \norm{Au}_{\Hb^{1,0,\ell}} \leq
    C\norm{u}_{\Hb^{1,0,\ell}}.\label{eq:b ops bounded}
      \end{equation}
\end{lemma}
The proof is identical to that in the referenced paper.  The
boundedness statement in \eqref{eq:b ops bounded} follows from the
commutator formulas in Lemma
\ref{lem:comms-with-b-ops}.  Once boundedness is established, the small calculus elliptic parametrix
used to prove \eqref{eq:elliptic bound bulk} is also valid in on $M$.
(Note that there is no improvement in the weight $\ell$.)

\subsubsection{The semiclassical version}
\label{sec:semicl-vers}

As $\mf$ blows down to a sphere $\sphere^{3}$, we can appeal to the
standard notion of differential operators on $\mf$.   We denote by
$\Hh^{1}(\mf)$ the lift of the semiclassical Sobolev space
$\Hh^{1}(\sphere^{3})$ to $\mf$ via the blow-down map.  For $u \in
\Hh^1(\mf)$, in particular $u \in L^2(\mathbb{S}^3)$, and the
$\Hh^1(\mf)$ norm controls the $L^2$ norms of $h\pd[x] u$, and
$\frac{h}{x}\pd[\theta] u$.

The classical analogues of the following lemmas can be found in prior work~\cite{BW20}.
\begin{lemma}
  \label{lem:bdedness-semicl}
  If $A \in \Psibh^{0}$ then
  \begin{equation*}
    \norm{Av}_{\Hh^{\pm 1}} \leq C\norm{v}_{\Hh^{\pm 1}}.
  \end{equation*}
\end{lemma}

\begin{lemma}
  \label{lem:semiclassical-elliptic-reg-sortof}
  Suppose $A, G\in \Psibh^{0}$ are supported near $x=0$ and $\WFbh'(A)
  \subset \ellbh (G)$.  For any $k,N$, there is a constant $C$ so that
  \begin{equation*}
    \norm{Au}_{\Hh^{\pm 1}} \leq C \norm{Gu}_{\Hh^{\pm 1}} +
    C h^{k}\norm{\chi u}_{\Hh^{\pm 1}} + Ch^{k}\norm{(1-\chi)u}_{\Hh^{-N}},
  \end{equation*}
  where $\chi \in C^{\infty}$ is identically $1$ on the support of $G$.
\end{lemma}

\begin{proof}
  The result follows from a standard elliptic parametrix construction:
  we can find $B\in \Psibh^{0}$ and $R \in \Psibh^{-\infty}$ so that
  on the microsupport of $A$, we have
  \begin{equation*}
     \operatorname{Id} = BG + h^{k}R.
  \end{equation*}
  and so
  \begin{equation*}
    \norm{Au}_{\Hh^{\pm 1}} \leq \norm{ABGu}_{\Hh^{\pm 1}} +
    h^{k}\norm{ARu}_{\Hh^{\pm 1}}.
  \end{equation*}
  Lemma~\ref{lem:bdedness-semicl} bounds the first term; to bound the
  second term we insert cutoff functions.  Near $x=0$ the term is
  bounded by $\norm{\chi u}_{\Hh^{\pm 1}}$ while away from $x=0$, we
  exploit that $R$ is an operator of order $-\infty$.
\end{proof}

As in the bulk setting, we repeatedly use the Hardy inequality on
$\mf$, where it reads 
\begin{equation*}
  \norm{h x^{-1}u} \leq 2\norm{h\pd[x]u}.
\end{equation*}
Clearly the $h$ factors out from both sides, but this phrasing of the
inequality emphasizes that $h/x$ will be estimated as a semiclassical
operator of order $1$.

We also need to understand the commutators of $1/x$ and $hD_{x}$ with
semiclassical $\bl$-pseudodifferential operators.
\begin{lemma}
  \label{lem:sc-easy-commutants}
  If $A \in \Psibh^{0}$, then
  \begin{equation*}
    \left[ \frac{1}{x}, A\right] = \frac{h}{x}C_{R} = C_{L}\frac{h}{x},
  \end{equation*}
  with
  \begin{equation*}
    C_{L}, C_{R} \in \Psibh^{-1}, \quad \sigmabh (C_{L}) = \sigmabh
    (C_{R}) = \frac{1}{i}\pd[\xi]\sigmabh(A) .
  \end{equation*}

  Moreover,
  \begin{equation*}
    \frac{1}{h}\left[ hD_{x}, A\right] = B + C hD_{x},
  \end{equation*}
  with
  \begin{align*}
    &B\in \Psibh^{0},   &C\in \Psibh^{-1}, \\
    &\sigmabh(B) = \frac{1}{i}\pd[x]a, & \sigmabh(C) = \frac{1}{i}\pd[\xi]a.
  \end{align*}
\end{lemma}

\subsection{The Mellin transform and polyhomogeneity}
\label{sec:mellin-transform}

Just as the Fourier transform is a key element in the study of
translation-invariant operators, we consider here the Mellin
transform, its analogue for dilation-invariant operators.  For our
purposes, we need only the Mellin transform associated to the single
boundary hypersurface $\mf$.  Suppose $u$ is a distribution on $M$
localized near $\mf$ (which is defined by the function $\rho$).  The
Mellin transform of $u$ associated to $\mf$ is defined by
\begin{equation*}
  \us\equiv \mathcal{M}_{\mf}u (\sigma, y) =
  \int_{0}^{\infty}\chi (\rho) u(\rho, y)\rho^{-i\sigma}\frac{d\rho}{\rho},
\end{equation*}
where $y$ denote the remaining coordinates near $\mf$ and $\chi$ is a
smooth compactly supported function localizing near $\rho = 0$.  The
Mellin transform has many rich properties analogous to those enjoyed by the
Fourier transform and many of its mapping properties can be deduced
from those of the Fourier transform by a change of variables.

The Mellin transform is particularly helpful in the study of
asymptotic expansions in powers of $\rho$ (the boundary defining
function for the hypersurface $\mf$) and $\log \rho$.  For simplicity,
we first discuss the case where our manifold has only a single boundary
hypersurface, i.e., when we have a  manifold with boundary $X$.  In
particular, we recall from Melrose~\cite[Section 5.10]{Melrose:APS}
the definition of a polyhomogeneous conormal distribution in this
setting.  If $u$ is a distribution on a manifold with boundary $X$, we write
\begin{equation*}
  u \in \phg^{E} (M) \quad (u\text{ is polyhomogeneous with index set }E)
\end{equation*}
if $u$ is conormal to $\pd X$ and
\begin{equation*}
  u \sim \sum_{(z,k)\in E}\rho^{iz}(\log \rho)^{k}a_{z,k}, 
\end{equation*}
where the $a_{z,k}$ are smooth functions on $\pd X$.  Here the
expansion should be interpreted as an asymptotic series as $\rho\to 0$
and $E$ is an \textit{index set} and must satisfy\footnote{As in the
  first author's previous works~\cite{BVW15,BVW18, BM19}, we adopt the
  convention of Melrose's unpublished book~\cite{Melrose:MWC} rather
  than the other reference~\cite{Melrose:APS}.}
\begin{itemize}
\item $\ds E \subset \complexes \times \{ 0, 1, 2, \dots \}$,
\item $E$ is discrete,
\item if $(z_{j},k_{j}) \in E$ with $\abs{(z_{j},k_{j})}\to \infty$,
  then $\Im z_{j} \to -\infty$,
\item if $(z,k) \in E$, then $(z,l)\in E$ for all $l = 0 , 1, \dots,
  k-1$, and 
\item if $(z,k)\in E$, then $(z-ij, k) \in E$ for all $j=1, 2, \dots$.
\end{itemize}
With these conventions, the functions that are smooth up to $\pd X$
are polyhomogeneous with index set
\begin{equation*}
  \cE_{0} = \{ (-ik, 0) \mid k = 0 , 1, 2, \dots\}.
\end{equation*}
The distributions in $\phg^{E}(X)$ can be characterized by the Mellin
transform, in which case the Mellin transform is meromorphic with
appropriate decay estimates in $\sigma$ and $(z,k) \in E$ if the
Mellin transform has a pole of order $k+1$ at $z$.  Polyhomogeneous
distributions can also be characterized by testing with radial vector
fields.  In the case of a manifold with boundary $X$, let $R$ denote
the radial vector field $\rho D_{\rho}$.  Then $u \in \phg^{E}(X)$ if
for all $A$, there is some $\gamma_{A}$ with $\gamma_{A}\to +\infty$
as $A\to +\infty$ so that
\begin{equation*}
  \left( \prod _{(z,k)\in E, \Im z > -A}(R - z)\right) u \in \rho^{\gamma_{A}}\Hb^{\infty}(X).
\end{equation*}
Here $\Hb^{\infty}(X)$ is the standard $\bl$-Sobolev space of order
$\infty$ and indicates iterated regularity under the application of
arbitrarily many $\bl$-vector fields.

Our main theorem concerns joint polyhomogeneity jointly at $\scri^{+}$
and $C_{+}$ in $[M; S_{+}]$, which is a manifold with codimension $2$
corners.  In this case one wants a polyhomogeneous distribution to
have compatible expansions at the two faces.  The index sets seen in
the expansions at the two faces are typically different and so we use
the notation $\cE = (E_{1}, E_{2})$ to denote the pair of index sets
and $\phg^{\cE}([M;S_{+}])$ to denote this space of distributions.

To test for polyhomogeneity at multiple boundary hypersurfaces, it
suffices to test \textit{individually} at each one with uniform
estimates at the other.  The following lemma is due independently to
Melrose~\cite[Chapter 4]{Melrose:MWC} and Mazzeo~\cite[Appendix]{Economakis}.
\begin{lemma}[Mazzeo, Melrose]
  \label{lem:double-phg}
  Suppose $H_{\ell}$, $\ell$ are boundary hypersurfaces of a manifold
  with corners $X$ and suppose $\rho_{\ell}$ defines $H_{\ell}$.  Let
  $R_{\ell}$ denote $\rho_{\ell}D_{\rho_{\ell}}$, the radial vector
  field at the $\ell$-th boundary hypersurface.  Suppose that for each
  $\ell$, there exists a $\gamma'$, and for all $A$ there is a
  $\gamma_{A}$ with $\lim_{A\to \infty}\gamma_{A}=+\infty$ such that
  \begin{equation}
    \left( \prod_{(z,k)\in E_{\ell}, \Im z > -A} (R_{\ell}-z)\right)u
    \in \rho_{\ell}^{\gamma_{A}}\rho^{\gamma'}\Hb^{\infty}(X),
  \end{equation}
  where $\rho^{\gamma'}$ denotes the multiproduct of the defining
  functions for $H_{j}$, $j\neq \ell$. Then $u \in \phg^{\cE}(M)$, where $\cE =
  (E_{1},E_{2}, \dots)$.  
\end{lemma}

In other words, one can test for polyhomogeneity with radial vector
fields as in the setting of a manifold with boundary provided that the
remainder improves the decay at the hypersurface in question at no
cost to the growth or decay at the other boundary hypersurfaces.

\section{Analytic preliminaries}
\label{sec:analyt-prel}

\subsection{Related wave equations and domains}
\label{sec:relat-wave-equat}
We introduce two first order operators related to the Dirac operator via a conjugation
\begin{align*}
  \dop &= i \rho^{-2-i\charge}\gamma^{0}\dirac_{\frac{\charge}{r}}\rho^{1+i\charge} ,\\
  \tdop &= i \rho^{-2-i\charge}\dirac_{\frac{\charge}{r}}\gamma^{0}\rho^{1+i\charge},  
  \end{align*}
  as well as the conjugated second order operator
  \begin{equation*}
  \Pd = -\rho^{-3-i\charge}\dirac_{\frac{\charge}{r}}^{2}\rho^{1+i\charge}.
  \end{equation*}
To uncover the relationship between these three operators, we introduce three additional  ``$l$-based'' families of operators
\begin{align*}
  \dop _{\ell} = \rho^{-\ell}\dop \rho^{\ell}, \quad
  \tdop_{\ell} = \rho^{-\ell}\tdop \rho^{\ell}, \quad
  \Pd_{\ell}  = \rho^{-\ell}\Pd \rho ^{\ell},
\end{align*}
where $\ell \in \reals$.  The notational advantages will be evident
momentarily.  

Observe that, with respect to the volume form
$\frac{d\rho}{\rho}x^{2}\,dx\,d\theta$, as long as $\ell\in \reals$,
\begin{equation*}
  \dop_{\ell}^{*} = \dop_{1-\ell}, \quad \tdop_{\ell}^{*} = \tdop_{1-\ell}.
\end{equation*}
so that
\begin{align*}
  \Pd &= \tdop_{1}\dop = \tdop^{*}\dop, \\
  \gamma^{0}\Pd_{\ell} \gamma^{0} &= \dop_{\ell+1}\tdop_{\ell}.
\end{align*}
The latter relationship plays a role in the boundary section, as
\begin{equation}\label{eq:Pd dual}
  \Pd^{*} = \gamma^{0}\Pd \gamma^{0}.
\end{equation}

In the sequel, we will require explicit expressions for our operators near $C_{+}$.  For convenience, we record this below, with $x=r/t$ and $\rho = 1/t$:
\begin{align*}
  \dop &= i\rho^{-2-i\charge}\left( \pd[t] + \frac{i\charge}{r} +
         \alpha_{r}\left( \pd[r] + \frac{1}{r} - \frac{1}{r}\beta
         K\right)\right)\rho^{1+i\charge} \\
  &= -i \rho\pd[\rho] - ix\pd[x] -i +\charge - \frac{\charge}{x} +
    \alpha_{r}\left( i \pd[x] + \frac{i}{x} - \frac{i}{x}\beta
    K\right), \\
  \Pd &= \rho^{-3-i\charge}\left[- \left( \pd[t] + \frac{i\charge}{r}\right)^{2} + \sum
        \pd[j]^{2} - \frac{i\charge}{r^{2}}\alpha_{r}\right]\rho^{1+i\charge} \\
  &= \left( \rho\pd[\rho] + x \pd[x] + 1 + i \charge -
    \frac{i\charge}{x}\right)^{*}\left( \rho \pd[\rho] + x\pd[x] + 1 +
    i \charge-
    \frac{i\charge}{x}\right) - D_{x}^{*}D_{x} -
    \frac{1}{x^{2}}\Lap_{\theta} - \frac{i\charge}{x^{2}}\alpha_{r}.
\end{align*}

The corresponding Mellin transformed and semiclassical $\dop$ operator takes the forms

\begin{align*}
  \dops &= xD_{x} + \sigma - i + \charge - \frac{\charge}{x} - \alpha_{r} \left(
          D_{x}  - \frac{i}{x} + \frac{i}{x}\beta K\right), \\
  \doph &= h \dops = h xD_{x} + z - i h + h \charge - \frac{h\charge}{x} -
          \alpha_{r}\left( hD_{x} - \frac{ih}{x} + \frac{ih}{x}\beta K\right),
\end{align*}
where $h = \abs{\sigma}^{-1}$ and $z = \frac{\sigma}{\abs{\sigma}}$.

Similarly, the second-order operators are given by 
\begin{align*}
  \Ps &= \left( x D_{x} - \frac{\charge}{x} + \sigma + \charge - i
        \right)^{*}\left( xD_{x} - \frac{\charge}{x} + \sigma + \charge- i
        \right) + 2i (\Im \sigma)\left( x D_{x} - \frac{\charge}{x} +
        \sigma + \charge - i \right) \\
      &\quad \quad \quad- D_{x}^{*}D_{x} -
        \frac{1}{x^{2}}\Lap_{\sphere^{2}} -
        \alpha_{r}\frac{i\charge}{x^{2}}, \\
  \Ph &= h^{2}\Ps = \left( h xD_{x} - \frac{h \charge}{x} + z + h
        \charge - i
        h\right)^{*} \left( h xD_{x} - \frac{h \charge}{x} + z + h \charge- i
        h\right) \\
    &\quad\quad\quad + 2i (\Im z) \left( h xD_{x} - \frac{h\charge}{x}
      + z + h\charge
        - i h\right) - (hD_{x})^{*}(hD_{x}) -
    \frac{h^{2}}{x^{2}}\Lap_{\sphere^{2}} - \alpha_{r} \frac{ih^{2}\charge}{x^{2}}.
\end{align*}

We do not need the precise forms of the operators near $C_{-}$ until
Section~\ref{sec:char-expon}; as we use a different set of coordinates
in that section, we do not record the forms here.

We observe now that as the indicial operator of $\dops$ agrees with that of $\cB$ near
$x=0$, the domain of $\dops$ also consists of functions that are
$H^{1}$ near the singularity.  We provide a sketch of the proof in
Section~\ref{sec:fredholm-statement} below.

\subsection{The radial sets}
\label{sec:radial-sets-module}

Classical propagation of singularities arguments show that wavefront
set is propagated along integral curves of the Hamilton vector field
within the characteristic set; the arguments become more complicated
when the vector field is singular (as near the singularity of the
potential) or proportional to the radial vector field.  In the latter
case, we call the subset the radial set.

Our treatment of the estimates near the radial set is in terms of
the second-order operators $\Pd$ and $\Ps$.  In both settings it is
convenient to replace the coordinate $x$ with the coordinate $v =
1-x$.  

We consider first the homogeneous (bulk) version (i.e., for $\Pd$).
As the $\bl$-principal symbol of $\Pd$ on $\mf$ agrees with that of
the wave operator on Minkowski space, the set of radial points must
agree as well.  In particular, as in previous
work~\cite[Section~3.6]{BVW15}, the radial set $\mathcal{R}$ is
exactly
\begin{equation*}
  \mathcal{R} = \{ (\rho, v, \theta, \sigma, \gamma, \eta) : \rho = v
  = 0, \eta = 0, \sigma = 0\}.
\end{equation*}
As it projects to $S_{+}\cup S_{-}$, the set $\mathcal{R}$ naturally splits into two components
$\mathcal{R}_{\pm}$ according to whether the component lies over
$S_{+}$ or $S_{-}$.  The radial set propagation arguments are structured along
\textit{thresholds}: near the past radial set $S_{-}$, we can propagate
regularity \textit{out} of the radial set provided we have enough a
priori regularity of our solution.  (This will be easily achieved as
we consider the forward solution, which necessarily vanishes near
$S_{-}$.)  At the future radial set $S_{+}$, we can propagate
regularity \textit{in}, but only up to the threshold.

We now consider the operator on the boundary $\Ps$.  The analysis of $\Ps$ near the radial sets follows from previous work
on the wave operator on Minkowski space; specifically analysis of the
conjugated, rescaled, Mellin-transformed normal operator $\widehat N(\rho^{-3} \Box
\rho)(\sigma)$.  Indeed, the operator $\Ps$ has the same semiclassical
principal and subprincipal symbol as $\widehat N(\rho^{-3} \Box
\rho) Id$.  We can therefore use relevant results in
\cite{Vasy13,BVW15}.  In particular, the characteristic set of the
semiclassical differential operator $\widehat N(\rho^{-3} \Box
\rho)(\sigma)$ is not homogeneous, and is understood as a submanifold
of the fiber-compactified cotangent bundle
$\overline{T}^*(\mf)$.  (Note that away from the poles this is
canonically identified with $T^*(\mathbb{S}^3)$.) It is shown in \cite{BVW15} that the
characteristic set of $\widehat N(\rho^{-3} \Box
\rho)$ (and therefore of $\Ps$), admits smooth families of radial sets
given by
\begin{equation}
\Lambda^\pm := \partial
\overline{N}^{*}S_{\pm},\label{eq:boundary radial}
\end{equation}
Concretely, near the fiber boundary of $\overline{N}^{*}S_{\pm}$, one
has coordinates $(v, \theta, \nu, \tilde \eta)$ where $(v, \theta)$
are spatial coordinates on $\mathbb{S}^3$ near $S_\pm$ with $\theta$
coordinates on $S_\pm$ and $v$ defining $S_{\pm}$, $\nu =
1/|\gamma|$ and $\tilde \eta = \eta / |\gamma|$ where $\gamma$ is dual to
$v$ and $\eta$ dual to $\theta$.
Then
\begin{equation*}
  \Lambda^\pm= \{ (v, \theta,
  \nu, \eta) \mid \nu = 0, v = 0, \tilde
\eta =0\}.
\end{equation*}
The boundaries of $N^{*}S_{\pm}$, denoted $\partial N^{*}S_{\pm}$, in
the radial compactification of the fibers of $T^{*}\mf$ act as sources
or sinks for a rescaling of the Hamilton vector
field.\footnote{Whether they are sources or sinks depends on the sign
  of $\gamma$ nearby and so proving the estimates requires treating
  the two components of $N^{*}S_{\bullet}\setminus 0$ separately.  As
  this argument is identical to the one in Minkowski space, we will
  omit it and so do not provide the components with unique names.}

The global structure of the bicharacteristic flow of $\Ps$, which is
explained in detail in \cite{BVW15}, has the following general
structure.  The (classical) characteristic set $\tilde \Sigma$ of $\Ps$ lies in
$\overline{T}^*(\mf)$ and lies entirely above $C_0$.  It consists of
two components $\tilde \Sigma^{\pm}$; along $\tilde \Sigma^{\pm}$ the
Hamilton flow limits to $\Lambda^{\pm}$ as the flow parameter goes to
$\infty$, and to $\Lambda^{\mp}$ as it goes to $- \infty$.  

\subsection{Variable-order Sobolev spaces}
\label{sec:vari-order-sobol}

As in prior work~\cite{BVW15, BVW18, BM19}, we aim to show that the
operator $\dops$ on $\mf$ is Fredholm on appropriate spaces.  In order
to do this, we aim to propagate regularity from $S_{-}$ to $S_{+}$ and
so the spaces on which it is Fredholm should include functions that
are more regular than some threshold at the past radial set and less
regular than the threshold at the future radial set.  As the two
thresholds agree, we employ variable-order Sobolev spaces.
Complicating the definition is our desire to guarantee enough
regularity near the singularity so that we can carry out our
propagation estimates there.

We therefore define a smooth regularity function
$s_{\tow}: \mf \to \reals$.
\begin{enumerate}
\item The function $s_{\tow}$ is constant near $S_\pm$ and $s_{\tow}\equiv
  0$ in $\{ x < 1/4\}$;
\item Within $C_0$, $s_{\tow}$ is monotonically decreasing as a function
  of $t/r$, and constant in neighborhoods of $S_+$ and $S_-$  
\item $s_{\tow} \rvert_{S_+} < 1/2 + \Im \sigma$ and
  $s_{\tow} \rvert_{S_-} > 1/2 + \Im \sigma$, the threshold exponents
  at $\Lambda^{\pm}$.
\end{enumerate}
The monotonicity in (2) ensures that the lift of $s_{\tow}$ to
$T^*(\mf)$ is monotonically decreasing along the flow from
  $\Lambda^-$ to $\Lambda^+$ within the characteristic set of
  $\dops$.  (This means that it is \textit{increasing} as a function of
  the flow parameter on the component of $\tilde \Sigma$ for which
  $\Lambda^-$ is a source.)
As the (classical) characteristic set of $\dops$ lives above the closure of
$C_{0}$, the condition at $x=0$ is easy to satisfy.

We further define
$$s_{\away} = -s_{\tow},
$$
which satisfies related properties; it is monotonically
\textit{increasing} (hence decreasing along the flow in the opposite direction) and satisfies $\pm s_{\away} \rvert_{S_\pm} > \mp (1/2 + \Im \sigma)$.

Motivated by previous work~\cite[Appendix A]{BVW15}, we define the
variable order Sobolev spaces $H^{s_{\tow}}(\sphere^{3})$ and
$H^{s_{\away}}(\sphere^{3})$ on the sphere.  As
$L^{2}(\mf) = L^{2}(\sphere^{3})$ and $s_{\tow}=s_{\away}=0$ near the
poles, there is a canonical identification of the spaces
$H^{s_{\tow}}$ (or $H^{s_{\away}}$) and $L^{2}(\mf)$ locally near the
poles.  We may therefore identify the variable order Sobolev spaces on
the sphere as consisting of distributions on $\mf$.

We now define the $\cX$ and $\cY$ spaces based on the variable order
spaces.  In an abuse of notation, we use $\cX^{s}$ and $\cY^{s}$ to
denote spaces based on $H^{s_{\tow}}$ and $\cX^{s^{*}}$ and
$\cY^{s^{*}}$ to denote those based on $H^{s_{\away}}$.  We first
define the $\cY$ spaces.
\begin{equation*}
  \cY^{s} = H^{s_{\tow}}, \quad \cY^{s^{*}} = H^{s_{\away}},
\end{equation*}
equipped with the inherited norms.  The $\cX$ spaces are then defined by
\begin{align*}
  \cX^{s} &= \{ u \in \cY^{s} \mid \dops u \in \cY^{s}\}, \\
  \cX^{s^{*}} &= \{ u \in \cY^{s^{*}} \mid \dops^{*} u \in \cY^{s^{*}}\}.
\end{align*}
The norms on the $\cX$ spaces are the graph norms, i.e.,
\begin{equation*}
  \norm{u}_{\cX^{s}}^{2} = \norm{u}_{\cY^{s}}^{2} + \norm{\dops u}_{\cY^{s}}^{2}.
\end{equation*}

Similarly, we define semiclassical spaces
\begin{equation*}\begin{gathered}
    \cY^{s}_h = \Hh^{s_{\tow}}, \qquad
      \cX^{s}_h = \{ u \in \cY_h^{s} \mid \doph u \in \cY_h^{s}\},
    \end{gathered}\end{equation*}
  and similarly for $s^* = s_{\away}$.

\subsection{Compressed characteristic set}
\label{sec:characteristic-set}

As in previous work~\cite{BW20}, our treatment of the propagation of
$\bl$-regularity is strongly influenced by the work of Vasy on
manifolds with corners~\cite{Vasy08}.

In the bulk, the main propagation results near the singularity take
place inside the \textit{compressed characteristic set}, which is the
appropriate extension of the ordinary characteristic set to the
boundary setting.  Near $\cf$ but away from $\mf$, we refer the reader
to prior work~\cite{BW20} for a discussion of this set.  We limit our
discussion here to a neighborhood of $C_{+}\cap \cf$.

In coordinates associated to the canonical one-form
\begin{equation*}
  \usigma\, d\rho + \uxi \,dx + \ueta \cdot d\theta
\end{equation*}
on $T^{*}M$, the characteristic set $\Sigma$ is given by
\begin{equation*}
  \Sigma = \left\{ (\rho, x, \theta, \usigma, \uxi, \ueta) \mid \left(
    \rho \usigma + x\uxi\right)^{2} - \uxi^{2} -
  \frac{1}{x^{2}}\abs{\ueta}^{2} = 0 \right\}.
\end{equation*}
The \textit{compressed characteristic set} $\Sigmadot$, originally due
to Melrose--Sj{\"o}strand~\cite{Melrose-Sjostrand1,
  Melrose-Sjostrand2}, is the closure of the image of the
characteristic set under the natural map $T^{*}M \to \Tbstar M$.  In
the coordinates associated to the canonical one-form
\begin{equation*}
  \sigma \frac{d\rho}{\rho} + \xi \frac{dx}{x} + \eta \cdot d\theta
\end{equation*}
on $\Tbstar M$, $\Sigmadot$ has the following form over $x=0$:
\begin{equation*}
  \Sigmadot |_{x=0} = \left\{ (\rho, x=0, \theta, \sigma, \xi = 0,
    \eta = 0 ) \mid \theta \in \sigma^{2}, \sigma \neq 0\right\}.
\end{equation*}

On the boundary $\mf$, the characteristic sets of $\dops$ and $\Ps$
are bounded away from $\cf$ (as the operators are elliptic there).  On
the other hand, the operators $\doph$ and $\Ph$ are not
\textit{semiclassically} elliptic there.  We therefore consider the
analogous construction in the semiclassical setting on $\mf$.  In
terms of the coordinates associated to the canonical one-form
\begin{equation*}
  \uxi \, dx + \ueta \cdot d\theta
\end{equation*}
on $T^{*}\mf$, the semiclassical characteristic set near $x=0$ is
given by
\begin{equation*}
  \Sigma_{h} = \left\{ ( x, \theta, \uxi , \ueta) \mid (x \uxi +
    z)^{2} - \uxi^{2} - \frac{1}{x^{2}}\abs{\ueta}^{2} = 0 \right\}.
\end{equation*}
The semiclassical compressed characteristic set is again the closure
of the image of $\Sigma_{h}$ under the map $T^{*}\mf \to \Tbstar
\mf$.  In terms of the coordinates given by
\begin{equation*}
  \xi \frac{dx}{x} + \eta \cdot d\theta,
\end{equation*}
on $\Tbstar \mf$, $\Sigmadot_{h} = \{ x^2 (\xi + z)^2 - (\xi^2 + |\eta|^2)
=0  \}$, so over $x=0$ we get the simple expression
\begin{equation}\label{eq:char semicl x is zero}
  \Sigmadot_{h}|_{x=0} = \left\{ (x=0, \theta, \xi = 0 , \eta = 0)
    \mid \theta \in \sphere^{2}\right\}.
\end{equation}

\section{Propagation in the bulk}
\label{sec:propagation-bulk}

The aim of this section is to prove that the forward solutions lie in
an appropriate weighted Sobolev space and possess additional
regularity; this regularity is expressed in terms of iterated
application of ``module derivatives'', which we define now.
\begin{definition}\label{def:module}
  Let $\module\subset \Psib^{1}(M)$ denote the $\Psib^{0}(M)$-module
  of pseudodifferential operators with principal symbol vanishing on
  the radial set $\mathcal{R}$.
\end{definition}
The module $\module$ is closed under commutators and is generated over
$\Psib^{0}$ by $\rho \pd[\rho]$, $\rho \pd[v]$, $v\pd[v]$,
$\pd[\theta]$, and $\id$.

\begin{theorem}
  \label{thm:bulk-regularity}
  If $\psi$ is the forward solution of $i\dirac_{\charge/r} \psi = g$,
  where $g \in \CI_{c}(M)$, then there are $m, \gamma\in \reals$ with
  $1 + m+ \gamma < 1/2$ so that $\psi \in \Hb^{1,m,\gamma}$ and for each
  $N\in \naturals$ and $A\in \mathcal{M}^{N}$, we have $A\psi\in
  \Hb^{1,m,\gamma}$.
\end{theorem}

Before we proceed to the proof we recall some facts about the domain
of the Hamiltonian $\cB$ which are discussed in detail in
\cite{BW20}.  The operator $\cB$ is essentially self-adjoint (with
core domain $C^\infty_c(\mathbb{R}^3 \setminus \{ 0 \})$) and its unique self-adjoint
domain is
$$
\mathcal{D} = \mathcal{D}(\cB) = r^{-1/2} \Hb^1(\mathbb{R}^3),
$$
where $\Hb^1$ is the Sobolev space based on $L^2(\frac{dr}{r}
d\theta)$ which is ``b'' at $r = 0$ and ``scattering'' at infinity,
meaning $u \in \Hb^1$ if, for $\chi = \chi(r) \in
C^\infty_0(\mathbb{R}_+)$  is identically one near $r = 0$ then $\{1,
r \partial_r, \partial_{\theta_i} \} \chi u \in L^2$ while
$\{ 1, \partial_r, \frac{1}{r} \partial_{\theta_i} \} (1 - \chi) u \in L^2$.
The powers $\mathcal{D}^m = \text{Dom}(\id + \cB^2)^{m/2}$ are
preserved by the forward propagator.
Following \cite{BW20}, we call a solution $i\dirac_{\charge/r} \psi =
g$ \textbf{admissible} provided $\psi \in C(\mathbb{R};
\mathcal{D}^m)$ for some $m$.  (This statement is local in time.  For
our analysis near the poles we assume that $\psi$ is the forward solution.)

The first part of the statement follows from the following proposition:
\begin{proposition}
  If $\psi$ is the forward solution of $i\dirac_{\charge/r}\psi =
  g$, where $g \in \CI_{c}(M)$, then $\psi$ is admissible and $\psi \in
  \Hb^{1,m,\gamma}$ for some $m,\gamma$.  
\end{proposition}

\begin{proof}
  As the Hamiltonian is self-adjoint and the inhomogeneous term is
  compactly supported, the spatial $L^{2}$ norm of $\psi$ is bounded
  for all time, i.e., there is some constant (depending on $g$) so
  that
  \begin{equation*}
    \int_{\reals^{3}} \abs{\psi}^{2} r^{2}\,dr\,d\theta < C(g).
  \end{equation*}
  In particular, we have that
  \begin{equation*}
    \int_{\reals \times \reals^{3}} t^{-1-2\epsilon} \abs{\psi}^{2}
    r^{2}\,dr\,d\theta \,dt < \widetilde{C}(g, \epsilon).
  \end{equation*}
  In terms of $\rho = 1/t$ and $x = r/t$ (i.e., near the northern
  cap), we therefore have that
  \begin{equation*}
    \int \rho^{-3+2\epsilon} \abs{\psi}^{2}
    \frac{d\rho}{\rho}\,dx\,d\theta < \widetilde{C}(g,\epsilon),
  \end{equation*}
  i.e., $\psi \in \rho^{\frac{3}{2}-\epsilon}L^{2}$ (recall that we
  use $\frac{d\rho}{\rho}\,dx\,d\theta$ as the volume form for this
  space).  
  
  Note also that the characterization of the domain (described in
  Section~\ref{sec:relat-wave-equat}) of the operator shows that
  $\psi$ has bounded energy for all time, i.e., there is some bound
  (depending on $g$) so that
  \begin{equation*}
    \int_{\reals^{3}} \left( \abs{\pd[t]\psi}^{2} + \abs{\grad
        \psi}^{2}\right) \,dx < C(g)
  \end{equation*}
  for all time.  Integrating this bound (with a weight in $t$) shows
  that in fact $\psi \in t^{1/2+\epsilon}H^{1}(\reals\times
  \reals^{3})$ and therefore lies in $\Hb^{1,0,1/2-\epsilon}$ near the
  pole ($x=0, \rho = 0$).

  As $\psi$ is tempered and lies in $\Hb^{1,0,1/2-\epsilon}$ near the
  pole, we can conclude that $\psi \in \Hb^{1,m,\gamma}$ for some $m$
  and $\gamma$.
\end{proof}

As a corollary to the proof of the above proposition, we record a
result useful in the estimation of error terms below.
\begin{corollary}
  \label{prop:local-H1}
  Suppose $\psi$ the forward solution of $i\dirac_{\charge/r}\psi = g$ and
  $g \in \CI_{c}$.  If $\chi \in \CI(M)$ is supported in
  $\{x < 1\}$, then $\chi \psi \in \Hb^{1,0,1/2-\epsilon}$ for all
  $\epsilon > 0$.  
\end{corollary}

Once we know that $\psi$ lies in some $\Hb^{1,m,\gamma}$, we can
decrease the weight to guarantee that $m+\gamma < 1/2$.  The remainder
of this section is devoted to the proof of module regularity, which
proceeds by propagation of singularities arguments.  In regions where
the operator $\dop$ is microlocally elliptic, classical microlocal
estimates suffice.  Similarly, where it is microlocally hyperbolic we
use the $\bl$-version of H{\"o}rmander's theorem to propagate the
regularity from one region to another.  We therefore focus our
attention on the radial points $N^{*}S_{\pm}$ (where the operator is
neither elliptic nor hyperbolic) and on the regions near the
singularity of the Coulomb potential where the operator is singular.

\subsection{The radial set}
\label{sec:radial-set}

At $N^{*}S_{+}$, the Hamilton vector field of the operator $P$ is
radial and so we appeal to radial points estimates of Vasy~\cite{Vasy13}.
The following proposition has the same proof as in previous work of
the first author~\cite{BVW18}.  To match prior work, we phrase the
result first in terms of $\dop$ rather than $\dirac_{\charge/r}$.
\begin{proposition}[{c.f.~\cite[Proposition~5.4]{BVW18}}]
  \label{prop:bulk-radial-pts}
  Suppose $u \in \Hb^{1,-\infty, \ell}$ for some $\ell$ and
  $\dop u \in \Hb^{0,m,\ell}$.  If we further assume that
  $u \in \Hb^{1,m - 1,\ell}$ on a punctured neighborhood $U\setminus
  N^{*}S_{+}$ of $N^{*}S_{+}$ in $S^{*} M$, then for $m' \leq m$
  with $m' + \ell < 1/2$, $u \in \Hb^{1,m' - 1,\ell}$ at
  $N^{*}S_{+}$.  Further, for $N \in \naturals$ with $m' + N \leq m$
  and for $A\in \module^{N}$, $Au \in \Hb^{1,m' - 1,\ell}$ at
  $N^{*}S_{+}$ as well.

  In particular if $\dop u \in \Hb^{0, \infty, \ell}$ and $u \in
  \Hb^{1, \infty, \ell}$ on the punctured neighborhood, then we
  conclude that $u$ has infinite order module regularity, i.e.\ that
  $N$ in $A \in \module^{N}$ is arbitrary.
\end{proposition}
\begin{proof}
From the assumptions, we have that $P u \in \Hb^{0, m - 1, \ell}$, and
the proposition follows from the cited source.
\end{proof}
\begin{remark}
  Since the conclusion of the theorem is drawn away from the
  singularity of the potential, the regularity at the poles is
  irrelevant, and one could equivalently assume that
  $u \in \Hb^{0 , m, \ell}$ on the punctured neighborhood and
  conclude that $u \in \Hb^{0 , m, \ell}$ at $N^{*}S_{+}$.
\end{remark}

\subsection{Near the singularity of the potential}
\label{sec:near-sing-potent}

Previous work of the first author~\cite{BW20} establishes the
following propagation of singularities estimate.  Note that this is a
statement at finite times.

\begin{theorem}[{\cite[Theorem~22]{BW20}}]
  \label{thm:dirac-coulomb-prop}
  If $\psi$ is an admissible solution of $i\dirac_{\charge/r}\psi = g
  \in \CI_{c}(\reals \times (\reals^{3}\setminus 0))$
  and $\abs{\charge}<1/2$.  For each $m$, $\WFb^{1,m}\psi\subset
  \Sigmadot$.  Away from $r=0$, $\WFb^{1,m}\psi$ is invariant under bicharacteristic flow.

  For $q_{0} = \{ (t_{0}, r=0, \theta \in \sphere^{2}, \tau_{0},
  \xi_{r}=0, \eta = 0)\}\subset \Sigmadot$ and let $U$ denote a
  neighborhood of $q_{0} \in \Sigmadot$.  If
  \begin{equation*}
    U \cap \{ \xi_{r} / \tau > 0\} \cap \WFb^{1,m}\psi =\emptyset,
  \end{equation*}
  then
  \begin{equation*}
    q_{0} \cap \WFb^{1,m}\psi = \emptyset.  
  \end{equation*}
\end{theorem}

Note that, as $\WFb^{1,m}\psi$ is closed, the second part of the
theorem yields regularity at the outgoing points ($\xi_{r} /\tau < 0$)
sufficiently near $\rho_{0}$.

The analogous statement near the poles ($\rho =0$, $x=0$) has the form:
\begin{theorem}
  \label{thm:bulk-prop-near-np}
  If $\psi$ is the forward solution of $i \dirac_{\charge/r}\psi
  = g$, $g\in \CI_{c}(\reals\times(\reals^{3}\setminus 0))$, and
  $\abs{\charge}<1/2$, then $\WFb^{1,m,\ell}\psi \subset \Sigmadot$.  For
  \begin{equation*}
    q_{0} = \{ (\rho=0, x=0, \theta \in \sphere^{2}, \sigma_{0}, \xi =
    0, \eta = 0)\} \subset \Sigmadot \cap \{\rho = 0\}
  \end{equation*}
  and let $U$ denote a neighborhood of $q_{0} \in \Sigmadot$.  If
  \begin{equation*}
    U \cap \{ \xi / \sigma > 0\} \cap \WFb^{1,m,\ell}\psi = \emptyset,
  \end{equation*}
  then
  \begin{equation*}
    q_{0} \cap \WFb^{1,m}\psi = \emptyset.
  \end{equation*}
\end{theorem}
As $g$ is smooth and compactly supported, $\WFb^{0,m}g = \emptyset$; accounting for the presence of
non-compactly supported inhomogeneous terms would require adding
statements analogous to those in
Theorem~\ref{thm:dirac-coulomb-prop}.  

Recall that, as in previous work~\cite{BW20},
Theorem~\ref{thm:bulk-prop-near-np} suffices to show that (in the
notation of the theorem)
\begin{equation*}
  U \cap \{ \xi / \sigma < 0\} \cap
  \WFb^{1,m}\psi = \emptyset.
\end{equation*}

As the proof of Theorem \ref{thm:bulk-prop-near-np} follows
\cite{BW20} closely, we defer this proof to the appendix.

\section{The boundary operator}
\label{sec:bdry-op}

The aim of this section is to show that the normal operator $\dops$ is
Fredholm, a key step in the iterative argument below.  \textbf{For the
remainder of this section, we let $s =  s_{\tow}, s^* = s_{\away} \in C^\infty(\mf)$ be
as in Section \ref{sec:vari-order-sobol}.}  We will prove the following theorem:
\begin{theorem}
  \label{thm:bdry-fredholm}
  The family $\dops$ enjoys the following mapping properties:
  \begin{enumerate}
  \item $\ds \dops : \cX^s \to \cY^s$ and $\dops^* : \cX^{s^*} \to \cY^{s^*}$
    are Fredholm.
  \item The operators $\dops$ form a holomorphic Fredholm family on
    these spaces in the strips
    \begin{equation*}
      \complexes_{s_{+}, s_{-}} = \left\{ \sigma \in \complexes \mid
        s_{+} < \frac{1}{2} + \Im \sigma < s_{-} \right\},
    \end{equation*}
    with $s_{\tow} |_{S_{\pm}} = s_{\pm}$.  The formal adjoint
    $\dops^{*}$ is antiholomorphic in the same region.  
  \item $\ds \dops^{-1}$ has only finitely many poles in each strip $a
    < \Im \sigma < b$.
  \item For all $a$ and $b$, there is a constant $C$ so that
    \begin{equation*}
      \norm{\dops^{-1}}_{\cY_{|\sigma|^{-1}}^{s}\to
        \cX_{|\sigma|^{-1}}^{s}} \leq C \langle \Re \sigma\rangle^{-1}
    \end{equation*}
    for $a < \Im\sigma < b$, $\ang{\Re\sigma} > C$, with a similar
    estimate holding for $(\dops^{*})^{-1}$.
  \end{enumerate}
\end{theorem}

The proof of the first two parts of Theorem~\ref{thm:bdry-fredholm}
follows from estimates of the form
\begin{align*}
  \norm{u}_{\cX^{s}} &\leq C\norm{\dops u}_{\cY^{s}} + C\norm{u}_{H^{-N}},
  \\
  \norm{u}_{\cX^{s^{*}}} &\leq C \norm{\dops^{*}u}_{\cY^{s^{*}}} + C \norm{u}_{H^{-N}}, 
\end{align*}
where $N$ is sufficiently large.  At elliptic and hyperbolic points
away from the singularity of the potential, microlocalizations of
these estimates follow from standard techniques.
Section~\ref{sec:fredholm-statement} is therefore devoted to the proof
of these estimates near the singularity and near the radial points
$\Lambda^\pm$.

The last two parts of Theorem~\ref{thm:bdry-fredholm} follow from the
semiclassical versions of these estimates, namely
\begin{align*}
  \norm{u}_{\cX^{s}_{h}} &\leq C \left( \frac{1}{h}\norm{\doph u}_{\cY^{s}_{h}}
                           +  h \norm{u}_{\Hh^{-N}}\right), \\
  \norm{u}_{\cX^{s^{*}}_{h}} & \leq  C \left( \frac{1}{h} \norm{\doph
                               ^{*}u}_{\cY^{s^{*}}_{h}} +  h \norm{u}_{\Hh^{-N}}\right).
\end{align*}
These estimates again follow from standard techniques at semiclassical
elliptic and hyperbolic points away from the singularity of the
potential.  Section~\ref{sec:semicl-stat} is devoted to these
estimates near the radial set and the singularity of the potential.
Though the analysis near the radial points is essentially identical to
the non-semiclassical version, the estimates near the singularity are
a bit more complicated due to the failure of semiclassical ellipticity
near the singularity.

\subsection{The Fredholm statement}
\label{sec:fredholm-statement}

We first show that $\dops$ is Fredholm on the desired spaces.
Throughout this section we use the notation $H^{1}$ to denote the
standard Sobolev space $H^{1}$ on the sphere obtained by blowing down
the lift of $\cf$ within $\mf$.  

Away from the radial sets and the singularity at the poles, standard
elliptic parametrix arguments and hyperbolic propagation arguments
apply.  Near the singularity at the poles,
the following lemma yields the desired estimate.
\begin{lemma}
  \label{lem:elliptic-property-etc}
  Fix $\chi \in C^{\infty}$ supported in $\{ x < 1/4\}$.  For any $N$,
  there is a $C$ so that
  \begin{equation*}
    \norm{\chi u}_{H^{1}} \leq C\left( \norm{\dops (\chi
        u)}_{L^{2}} + \norm{\chi u}_{H^{-N}_{g}}\right).
  \end{equation*}
  In particular, as distributions in $\cX^{s}$ lie in $L^2$ near the
  pole, for any $N$, we may estimate
  \begin{equation*}
    \norm{\chi u}_{\cX^{s}} \leq C\left( \norm{\dops (\chi
        u)}_{\cY^{s}} + \norm{\chi u}_{H^{-N}}\right).
  \end{equation*}
\end{lemma}

In the above lemma it is convenient to use the blow-down map to
identify classical Sobolev spaces on the boundary as we can then use
the compactness of the inclusion of standard Sobolev spaces.

\begin{proof}
  In this proof we use $L^{2}_{\bl}(x^{-1}dxd\theta)$ to denote the $L^{2}$ space
  associated to the $\bl$-density $\frac{dx}{x}d\theta$, while $L^{2}$
  continues to denote $L^{2}(x^2dxd\theta)$.  

  The operator $\dops$ is elliptic on $\{ x < 1/4\}$ and so we may use
  the $\bl$-calculus to construct a good (large calculus) parametrix
  $G$ for $\dops$.  (Indeed, following \cite[Sect.\ 4.1]{BW20}, the
  reduced normal operator of $\dops$ agrees with that of the
  stationary elliptic operator $\cB$ from \eqref{eq:stationary
    operator} (acting in $x$).  One can check that $[1/2, 3/2]$ does
  not intersect the boundary spectrum of $x\cB$, and thus there is a
  large calculus parametrix $Gx^{-1}$ for $x \dops$.)  From \cite{Melrose:APS}, we
  conclude that there is an operator $G$ such that
  \begin{equation*}
    I = G\dops + R,
  \end{equation*}
  where $G,R: r^{-3/2}L^{2}_{\bl} \to r^{-1/2}\Hb^{1}(x^{-1}dxd\theta)$.  As
  $r^{-3/2}L^{2}_{\bl} = L^{2}$ and $r^{-1/2}\Hb^{1}(x^{-1}dxd\theta)\cap
  L^{2} = H^{1}$, we may estimate
  \begin{equation*}
    \norm{\chi u}_{H^{1}} \leq C\left( \norm{\dops (\chi
        u)}_{L^{2}} + \norm{u}_{L^{2}}\right).  
  \end{equation*}
  Interpolation inequalities on the sphere then allow us to bound
  \begin{equation*}
    \norm{\chi u}_{L^{2}} \leq \epsilon \norm{\chi u}_{H^{1}}
    + C\epsilon^{-N}\norm{\chi u}_{H^{-N}},
  \end{equation*}
  finishing the proof.
\end{proof}

Near the radial points $\Lambda^\pm$ (see Section
\ref{sec:radial-sets-module}), we use the following estimates.
\begin{lemma}[{\cite[Proposition 2.3 and 2.4]{Vasy13}}]
  \label{lem:radial-point-estimates-bdry-vasy}
  For all $N$ and for $s_{0}\geq m > \frac{1}{2}+\Im \sigma$, and for
  all $A, B, G \in \Psib^{0}(\mf)$ supported near $\Lambda^-$ with
  $A,G$ elliptic at $\Lambda^-$ and bicharacteristics from the
  microsupport of $B$ tend to $\Lambda^-$ in one direction with closure in the
  elliptic set of $G$, we have
  \begin{equation*}
    \text{If }Au \in H^{m} \text{ then }\norm{Bu}_{H^{s_{0}}} \leq C
    \left( \norm{G\Ps u}_{H^{s_{0}-1}} + \norm{u}_{H^{-N}}\right).
  \end{equation*}

  For $s_{0} < \frac{1}{2}+\Im \sigma$ and all $A, B, G \in
  \Psib^{0}(\mf)$ supported near $\Lambda^+$ with $B, G$ elliptic at
  $\Lambda^+$ and bicharacteristics from $\WF'(B)\setminus
  \Lambda^+$ reach the microsupport of $A$ in one direction while
  remaining in the elliptic set of $G$, we have
  \begin{equation*}
    \norm{Bu}_{H^{s_{0}}} \leq C \left( \norm{G\Ps u}_{H^{s_{0}-1}} +
      \norm{Au}_{H^{s_{0}}} + \norm{u}_{H^{-N}}\right).
  \end{equation*}
\end{lemma}

Because $\Ps =
\widehat{N}\left(\rho^{-2}\tdop\rho
\right) \dops$, estimates for $\Ps$ in $H^{s^{0}-1}$ immediately yield
estimates for $\dops$ in $H^{s_{0}}$, i.e.\ we conclude that,
assumptions as in the lemma, that
\begin{equation}\label{eq:nonsemi above thresh L}
  \text{If }Au \in H^{m} \text{ then }\norm{Bu}_{H^{s_{0}}} \leq C
  \left( \norm{G\dops u}_{H^{s_{0}}} + \norm{u}_{H^{-N}}\right).
\end{equation}
and
\begin{equation}\label{eq:nonsemi below thresh L}
  \norm{Bu}_{H^{s_{0}}} \leq C \left( \norm{G\dops u}_{H^{s_{0}}} +
    \norm{Au}_{H^{s_{0}}} + \norm{u}_{H^{-N}}\right).
\end{equation}
Given $A, B, G$ as in the lemma, this follows directly from the lemma
by choosing a $G' \in \Psib^0(\mf)$ for which $B$ is still
microsupported in the elliptic set of $G'$ and with
$\WF' G' \subset \ell G$ and applying the lemma to $A, B, G'$.  Then
the fact that $G' \dops \in \Psib^1$ and elliptic regularity give
\eqref{eq:nonsemi above thresh L} and \eqref{eq:nonsemi below thresh
  L}.

Although similar estimates hold for $\Ps^{*}$, these do not
immediately give the desired estimates for $\dops^{*}$.  To obtain
these, we recall the formulas for the ``$\ell$-based'' operators
(i.e., those obtained via conjugation by $\rho^{\ell}$) in Section
\ref{sec:relat-wave-equat}.  Specifically, recall that
\begin{equation*}
(\Pd_1)^* =   \gamma^{0}\Pd_{-1} \gamma^0= \dop \tilde \dop_{-1}.
\end{equation*}
From this we can obtain estimates for $\dops^{*}$ by taking the
adjoint of the Mellin-transformed normal operator:
\begin{equation*}
\widehat{N}_\sigma(\Pd_1^*) =  \dops \widehat{N}_\sigma(\dop_{-1}),
\end{equation*}
so using that in general for $\bl$-operators $\widehat{N}_\sigma(B^*)=
\widehat{N}_{\overline{\sigma}}(B)$, we conclude that
\begin{equation*}
  \widehat{N}_{\overline{\sigma}-i}(P) =   \widehat{N}_{\sigma +
    i}(P^*)^{*}  = \widehat{N}_{\sigma}(P_1^*)^{*}= \left(\widehat{\tdop}_{-1,\sigma}\right)^{*}\dops^{*}.
\end{equation*}
The radial estimates for $\widehat{P}_{\overline{\sigma}-i}$ are read
off in the obvious way (i.e.\ by exchanging $\overline{\sigma} - i$
for $\sigma$) in the two Lemmas above. This immediately yields estimates for
$\dops^{*}$.  We may therefore conclude that the adjoint $\dops^{*}$
satisfies analogous radial point estimates with thresholds
$\frac{1}{2}-1-\Im \sigma = - \left( \frac{1}{2} + \Im \sigma
\right)$, i.e., in the dual spaces.  Concretely, for $A, B, G$ as in
the theorem, if $s_0^* > -(1/2 + \Im \sigma)$ then
  \begin{equation}\label{eq:nonsemi above thresh L dual}
    \text{If }Au \in H^{m} \text{ then }\norm{Bu}_{H^{s_{0}^*}} \leq C
    \left( \norm{G\dops^* u}_{H^{s_0^*}} + \norm{u}_{H^{-N}}\right).
  \end{equation}
  and is $s_0^* < -(1/2 + \Im \sigma)$ then
    \begin{equation}\label{eq:nonsemi below thresh L dual}
    \norm{Bu}_{H^{s_0^*}} \leq C \left( \norm{G\dops^* u}_{H^{s_0^*}} +
      \norm{Au}_{H^{s_0^*}} + \norm{u}_{H^{-N}}\right).
  \end{equation}

Taking microlocal partitions of unity as appropriate, we thus have the
two estimates
\begin{align*}
  \norm{u}_{\cX^{s}} &\leq C\left( \norm{\dops u}_{\cY^{s}} +
                       \norm{u}_{H^{-N}}\right), \\
  \norm{u}_{\cX^{s^{*}}} &\leq C\left(
                                  \norm{\dops^{*}u}_{\cY^{s^{*}}}
                                  + \norm{u}_{H^{-N}}\right).
\end{align*}
As the inclusions $\cX^{s}, \cX^{s^{*}} \hookrightarrow
H^{-N}$ are compact for sufficiently large $N$, the operators $\dops$
and $\dops^{*}$ are Fredholm on the stated spaces, proving the first
part of the theorem.  The second part of the theorem follows by
inspection of the coefficients together with the observation that we
may choose $N$ uniformly on these strips.

\subsection{The semiclassical statements}
\label{sec:semicl-stat}

The other two statements of the theorem follow from a semiclassical
estimate of the form
\begin{equation}
  \label{eq:semiclass-est}
  \norm{u}_{\cX^{s}_{h}} \leq \frac{C}{h}\norm{\doph u}_{\cY^{s}_{h}}
  + \bo (h) \norm{u}_{H^{-N}_{h}},
\end{equation}
with a corresponding estimate for $\doph^{*}$, where $h =
\abs{\sigma}^{-1}$.  This estimate immediately implies that $\doph$ is
invertible for small $h$ and provides a bound on the norm, proving the
last two statements of the theorem.

For the rest of this section, recall that $\Psibh$ (without a
superscript) denotes the space of semiclassical
$\bl$-pseudodifferential operators with \textit{compactly supported}
symbols.

\subsubsection{The radial set}
\label{sec:radial-set-bdry}

The estimates near the radial sets $\Lambda^\pm$
follow as in earlier work~\cite{BM19, BVW15, BVW18, Vasy13}:
\begin{proposition}[{c.f.~\cite[Propositions~2.8 and~2.9]{Vasy13}}]
  \label{prop:sc-radial-est}
  For $s|_{S_-} > m> \frac{1}{2}+\Im \sigma$ and $A, B, G
  \in \Psibh^{0}$ supported near $\Lambda^-$ with $A,G$ elliptic at
  $\Lambda^-$ and so that semiclassical
  bicharacteristics from the microsupport of $B$ tend to $\Lambda^-$ in one direction with closure in the elliptic
  set of $G$, we have
  \begin{equation*}
    \text{If }Au\in H^{m}\text{ then }\norm{Bu}_{\cX^{s}_{h}} \leq
    \frac{C}{h}\norm{G\doph u}_{\cY^{s}_{h}} + Ch\norm{u}_{\Hh^{-N}}.
  \end{equation*}

  For $s|_{S_+}< \frac{1}{2}+\Im \sigma$, and for all $A,B,G\in
  \Psibh^{0}$ supported near $\Lambda^+$ with
  $B,G$ elliptic at $\Lambda^+$ and so that
  semiclassical bicharacteristics from $\WFbh'(B)\setminus \Lambda^+$ reach the microsupport of $A$ in one
  direction while remaining in the elliptic set of $G$, we have
  \begin{equation*}
    \norm{Bu}_{\cX^{s}_{h}} \leq \frac{C}{h}\norm{G\doph
      u}_{\cY^{s}_{h}} + C\norm{Au}_{\cX^{s}_{h}} + Ch\norm{u}_{\Hh^{-N}}.
  \end{equation*}
\end{proposition}
Analogous estimates hold for $\dops^{*}$ on the dual spaces as well;
these follow as in Section~\ref{sec:fredholm-statement} above.

\subsubsection{Commutators with semiclassical $\bl$-pseudodifferential
  operators}
\label{sec:comm-with-semicl}

\begin{lemma}
  \label{lem:ellip-semic-comm}
  If $C \in \Psibh^{0}$ is invariant with real-valued scalar principal
  symbol, then
  \begin{equation*}
    \frac{1}{h}\left[ \Ph , C\right] \in \left\{ \frac{h^2}{x^2}
      \Delta_\theta ,h^2 D_x^2, \frac{h^2}{x}
      D_x,
      \frac{h^2}{x^2} \right\} \Psibh^{-1} + \left\{hD_x, \frac hx
    \right\} \Psibh^{0} +  \Psibh^{1},
  \end{equation*}
  The possibility that the commutator of $C$ and $K$ is non-vanishing
  leads to a different result for the first-order operator:
  \begin{align*}
    \frac{1}{h} \left[ \doph , C\right]\in \left\{ \frac hx \beta K,
    \frac hx, h D_x \right\}\Psibh^{0}.
  \end{align*}
\end{lemma}

\begin{lemma}
  \label{lem:general-semic-comm-L}
  If $C \in \Psibh^{0}(X)$ is invariant with real-valued scalar
  principal symbol $c$, then
  \begin{equation*}
    \frac{1}{ih}\left[\doph, C\right] = A_{0}\left( \alpha_{r} \left(
        ih \pd[x] + \frac{ih}{x} -\frac{ih}{x}\beta K\right)  - \frac{h\charge}{x} \right)
    + B_{0} + \alpha_{r}B_{1} +
    \bB_{2}\frac{1}{x} +\bB_{3}hD_{x} + h\bB_{4}
  \end{equation*}
  where
  \begin{itemize}
  \item $\ds A_{0}\in \Psibh^{-1}$, with $\sigmabh(A_{0}) = -
    \pd[\xi](c)$,
  \item $\ds B_{0} \in \Psibh^{0}$, with $\sigmabh(B_{0}) = - x\pd[x](c)$,
  \item $\ds B_{1} \in \Psibh^{0}$, with $\sigmabh(B_{0}) = - \pd[x](c)$
  \item $\ds \bB_{2} \in \Psibh^{0}$, with $\supp \sigmabh(\bB_{2})
    \subset \supp \pd[\eta](c)$,
  \item $\ds \bB_{3} \in \Psibh^{-1}$, with $\supp\sigmabh (\bB_{3})
    \subset \supp\pd[\eta](c)$,
  \item $\ds \bB_{4} \in \left\{ \frac hx, hD_x \right\} \Psibh^{-2}$.
  \end{itemize}
\end{lemma}

\begin{proof}
  Recall that we may write
  \begin{equation*}
    \doph = h xD_{x} + z + h\charge  -ih - \frac{h\charge}{x} - \alpha_{r} \left(
      h D_{x} - \frac{ih}{x} + \frac{ih}{x}\beta K\right).
  \end{equation*}

  We first consider the angular term and write
  \begin{equation*}
    \frac{1}{ih}\left[ -\frac{ih}{x} \alpha_{r} \beta K, C \right] = -
    \frac{i}{x}\alpha_{r}\frac{1}{ih}\left[ h \beta K, C\right] -
    \frac{i}{x}\frac{1}{ih}\left[\alpha_{r}, C\right] -
    \frac{1}{ih}\left[ \frac{i}{x}, C\right]\alpha_{r}h\beta K.
  \end{equation*}
  As $C$ is invariant and $\alpha_{r}$ and $K$ have only angular
  dependence, the first and second terms are microsupported in the
  support of $\pd[\eta](c)$ and contribute to $\bB_{2}$, while the
  third term yields the angular part of the term involving $A_{0}$.
  (Indeed, we take $A_{0} = \frac{1}{ih}[1/x, C]$ as the \textit{definition} of $A_{0}$.)
  Lemma~\ref{lem:sc-easy-commutants} then shows that $\sigmabh(A_{0}) =
  - \pd[\xi](c)$.

  We now consider the term involving $i\alpha_{r}\left( h\pd[x] +
    \frac{h}{x}\right)$.  Because $\alpha_{r}$ depends only on the
  angular variables, its commutator with $C$ is again microsupported
  in $\supp \pd[\eta](c)$ and therefore contributes to the $\bB_{2}$
  and $\bB_{3}$ terms.  Now, by Lemma~\ref{lem:sc-easy-commutants} the
  commutator of $-hD_{x}$ with $C$ yields a term of the form
  $A_{1}\alpha_{r}\pd[x]$ as well as a term of the form
  $\alpha_{r}B_{1}$, where $\sigmabh(A_{1}) = \sigmabh(A_{0})$ and
  $\sigmabh (B_{1})=\pd[x](c)$.  At the cost of a contribution to the
  $\bB_{4}$ term, we may replace $A_{1}$ by $A_{0}$.  The commutator
  of $ih/x$ with $C$ yields the corresponding part of the $A_{0}$
  term.

  Finally, the commutator of $-h\charge/x$ with $C$ yields the
  corresponding piece of the $A_0$ term, while the commutator of
  $hxD_{x}$ with $C$ yields the $B_{0}$ term.  The $h\charge$ part of
  the operator commutes with $A$ and so does not contribute.
\end{proof}

\subsubsection{Elliptic estimates near the singularity}
\label{sec:ellipt-estim-near-bdry}

We now establish the semiclassical elliptic estimates near the conic
singularity.  \textit{Throughout this section, we assume that all
  pseudodifferential operators and distributions are supported in $\{
  x < 1/4\}$.}  We further present arguments only near the ``north
pole'' (i.e., in $C_{+}$) as the proofs are nearly identical near the
``south pole'' (i.e., in $C_{-}$).

Recall that $s$ continues to satisfy the assumptions of Section
\ref{sec:vari-order-sobol}.  The core of the semiclassical elliptic
estimate near the singularity is the
following proposition: 
\begin{proposition}
  \label{prop:elliptic-est-semic}
  Suppose $\abs{\charge}<1/2$, $A \in \Psibh^{0}$ is invariant and
  satisfies $\WFbh '(A) \cap \Sigmadot_{h} = \emptyset$.  For any
  $G \in \Psibh^{0}$ with $\WFbh'(A) \subset \ellbh (G)$, there is a
  constant $C$ so that
  \begin{equation*}
    \norm{Au}_{\cX^{s}_{h}} \leq C \norm{G\doph u}_{\cY^{s}_{h}} + C
    h^{1/2}\norm{Gu}_{\cX^{s}_{h}} + \bo (h^{\infty})
    \norm{u}_{\cX^{s}_{h}} .
  \end{equation*}
\end{proposition}

In fact, by enlarging the microsupport of $G$, we can improve the
factor of $h$ in the estimate:
\begin{corollary}
  \label{prop:better-h-factor-ellip}
  If $\abs{\charge}<1/2$, $A \in \Psibh^{0}$ is invariant and
  satisfies $\WFbh'(A) \cap \Sigmadot_{h}= \emptyset$, then for any
  $N$ and $G\in \Psibh^{0}$ with $\WFbh'(A)\subset \ellbh(G)$, there
  is a constant $C$ so that
  \begin{equation*}
    \norm{Au}_{\cX^{s}_{h}} \leq C\norm{G\doph u}_{\cY^{s}_{h}} +
    Ch^{N} \norm{Gu}_{\cX^{s}_{h}} + \bo(h^{\infty})\norm{u}_{\cX^{s}_{h}}.
  \end{equation*}
\end{corollary}

The following proposition provides a convenient way to estimate error
terms of the form $\norm{u}_{\Hh^{1}}$:

\begin{proposition}
  \label{prop:semi-class-go-up-reg}
  If $\chi$ is a smooth radial cut-off function supported in $\{
  x < \delta\}$ for $\delta > 0$ sufficiently small, then there is a
  $C$ so that for all $u \in \cX^{s}_h$, $\chi u \in \Hh^{1}$ and
  \begin{equation*}
    \norm{\chi u }_{\Hh^{1}} \leq C \norm{\chi u}_{\cX^{s}_{h}}.
  \end{equation*}
\end{proposition}

\begin{proof}
  The proposition is a reflection of the observation that $\dops$ is
  elliptic near $x=0$.  Recall that the $\cX^{s}_{h}$ norm is given by
  \begin{equation*}
    \norm{v}_{\cX^{s}_{h}}^{2} = \norm{v}^{2}_{\Hh^{s}} + \norm{\doph
      v}_{\Hh^{s}}^{2}. 
  \end{equation*}
  and thus by our assumptions on $s$, if $v$ is supported in the region $\{ x < \delta\}$
  with $\delta \leq 1/4$, it has the form
  \begin{equation*}
    \norm{v}_{\cX^{s}_{h}} ^{2} = \norm{v}^{2} + \norm{\doph v}^{2},
  \end{equation*}
  where both norms are the $L^{2}$ norm (taken with respect to the
  volume form $x^{2}\,dx\,d\theta$).

  Consider now the real part of the pairing $\langle \Ph v, v \rangle$.  Using
  the form of $\Ph$ in Section \ref{sec:relat-wave-equat}, 
  integrating by parts shows that this is equal to
  \begin{align*}
    \norm{h D_{x} v}^{2} + \norm{\frac{h}{x}\grad_{\theta}v}^{2} -
    \norm{\left( h xD_{x} - \frac{h\charge}{x} + z + h\charge - ih\right)v}^{2} +
    2(\Im z)(h + 2\Im z)\norm{v}^{2}. 
  \end{align*}
  In particular, as
  \begin{equation*}
    \abs{\Re \ang{\Ph v,v}} \leq \epsilon' \norm{v}_{\Hh^{1}}^{2}+ \frac{1}{\epsilon'}\norm{\Ph}_{\Hh^{-1}}^{2},
  \end{equation*}
  we may bound  
  \begin{equation*}
    \norm{v}_{\Hh^{1}}^{2}\leq \frac{C}{\epsilon}
    \norm{h xD_{x}v}^{2} + \frac{C}{\epsilon} \norm{v}^{2} + (1+\epsilon)
    \norm{\frac{h\charge}{x}v}^{2} + \frac{C}{\epsilon'}\norm{\Ph
      v}_{\Hh^{-1}}^{2} + \epsilon' \norm{v}_{\Hh^{1}}^{2}.
  \end{equation*}
  For $\abs{\charge}<1/2$, the Hardy inequality yields
  \begin{equation*}
    \norm{\frac{h\charge}{x}v}^{2} \leq 4\charge^{2}\norm{h D_{x}v}^{2},
  \end{equation*}
  so, for $\epsilon > 0$ sufficiently small and fixed, the third term
  on the right can be absorbed into the left hand side.
  Similarly, if $v$ is supported in $\{ x < \delta\}$, the first term
  on the right is bounded by $(C/\epsilon)\delta^{2} \norm{hD_{x}
    v}^{2}$, so for $\delta$ sufficiently small this can also be
  absorbed on the left, as can the last term for $\epsilon'$ small, leaving
  \begin{equation*}
    \norm{v}_{\Hh^{1}}^{2} \leq C \left( \norm{v}^{2} + \norm{\Ph v}_{\Hh^{-1}}^{2}\right).
  \end{equation*}
  Recall that
  \begin{equation*}
    \Ph = \left(\frac{1}{\abs{\sigma}}\widehat{N}(\tdop_{1})\right)\doph,
  \end{equation*}
   and so we obtain the bound
   \begin{equation*}
     \norm{v}_{\Hh^{1}}^{2} \leq C\left( \norm{v}^{2} +
       \norm{\doph v}^{2}\right).
   \end{equation*}
   As $v$ is supported in $\{ x < \delta\}$, the right side is a
   multiple of $\norm{v}_{\cX_{h}^{s}}^{2}$.
\end{proof}

A more careful look at the real part of the pairing yields the 
following: 
\begin{lemma}
  \label{lem:semic-pair-arg}
  Suppose $A, G\in \Psibh^{0}$ with
  $A$ invariant, $\sigmab (A)$ scalar and real-valued,
  $\WFbh'(A) \subset \ellbh (G)$ and $\delta$ as in
  Proposition~\ref{prop:semi-class-go-up-reg}.  There is a constant
  $C$ so that
  \begin{align*}
    &\abs{\norm{hD_{x}Au}^{2} + \norm{\frac{h}{x}\grad_{\theta}Au}^{2}
    - \norm{\left( h xD_{x} + z - \frac{h\charge}{x} +h\charge -
      ih\right)Au}^{2}} \\
    &\quad\quad \leq \epsilon \norm{Au}_{\Hh^{1}}^{2} +
      \frac{C}{\epsilon}\norm{G\doph u}_{L^{2}_{g}}^{2} + Ch
      \norm{Gu}_{\Hh^{1}}^{2} + \bo (h^{\infty})\norm{\chi
      u}_{\Hh^{1}}^{2} + \bo (h^{\infty}) \norm{(1-\chi)u}_{\Hh^{-N}},
  \end{align*}
  where $\chi$ is a cut-off function supported in $x < \delta$.
\end{lemma}

\begin{proof}
  Considering the real part of the pairing $\ang{Au, \Ph Au}$ shows
  that we may bound 
  \begin{align*}
    &\abs{\norm{h D_{x} Au}^{2} +
      \norm{\frac{h}{x}\grad_{\theta}Au}^{2} - \norm{\left( h xD_{x} -
          \frac{h \charge}{x} + z +h\charge - ih\right)Au}^{2} + 2(\Im z)(h +2 (\Im
    z))\norm{Au}^{2}} \\
    &\quad\leq \abs{\ang{Au, \Ph Au}}.
  \end{align*}
  Using the factorization $\Ph = \widetilde{L}\doph$ with
  $\widetilde{L} \in \Diffh^{1}$, the term on the right is bounded by
  \begin{equation*}
    \epsilon \norm{Au}_{\Hh^{1}}^{2} + \frac{1}{\epsilon}\norm{A\widetilde{L}\doph
      u}_{\Hh^{-1}} + \frac{1}{\epsilon}\norm{[\Ph , A] u}_{\Hh^{-1}}^{2},
  \end{equation*}
  and so, by Lemma~\ref{lem:ellip-semic-comm} and
  Lemma~\ref{lem:semiclassical-elliptic-reg-sortof}, can be estimated by
  \begin{equation*}
    \epsilon \norm{Au}_{\Hh^{1}}^{2} + \frac{C}{\epsilon}\norm{A\doph 
      u} + \frac{C}{\epsilon}h\norm{Gu}_{\Hh^{1}}^{2} +
    \bo(h^{\infty})\norm{\chi u}_{\Hh^{1}} + \bo
    (h^{\infty})\norm{(1-\chi)u}_{\Hh^{-N}},
  \end{equation*}
  where $\chi$ is a radial cut-off function supported near $x = 0$. 
\end{proof}

\begin{proof}[Proof of Proposition~\ref{prop:elliptic-est-semic}]
  For $x > 0$, the proposition follows from standard proofs of
  semiclassical elliptic regularity.  We may therefore assume that
  the support and microsupport of $A$ are contained in $\{ x <
  \delta\}$.  

  We begin by establishing a version of the estimate with
  $\cX^{s}_{h}$ replaced by $\Hh^{1}$ and $\cY^{s}_{h}$ replaced by
  $L^{2}$.  Because $\WFbh'(A) \cap \Sigmadot_h = \emptyset$, we may
  assume there is a constant $C > 0$ so that
  \begin{equation*}
    2 \left( (\xi + z)^{2} + 1 \right) < C \left( \xi^{2} + \abs{\eta}^{2}\right)
  \end{equation*}
  on the microsupport of $A$ (see \eqref{eq:char semicl x is zero}).  We may therefore bound
  \begin{align}
    \label{eq:bounding-hdx-term}
    \norm{\left( h xD_{x} + z\right)Au}^{2} + \norm{Au}^{2} &\leq \ang{C \Op_h
                                                              (\xi^{2} +
                                                              \abs{\eta}^{2})Au,
                                                              Au} + Ch
                                                              \norm{Gu}^{2}_{\Hh^{1}}
    \\
                                                            &= C \left( \norm{h x D_{x}Au}^{2} +
                                                              \norm{x (\frac{h}{x}\grad_{\theta}Au)}^{2}\right) + C h
                                                              \norm{Gu}_{\Hh^{1}}^{2} \notag \\
                                                            &\leq C \delta^{2} \norm{Au}_{\Hh^{1}}^{2} +C h \norm{Gu}_{\Hh^{1}}^{2},\notag
  \end{align}
where in the last line we have used that $A$ is supported in $\{ x <
  \delta\}$.

Adding
  $\norm{\left( h xD_{x} + z - \frac{h\charge}{x} - ih\right)Au}^{2} +
  \norm{Au}^{2}$ to both sides of the estimate in
  Lemma~\ref{lem:semic-pair-arg} then yields
  \begin{equation*}\begin{gathered}
    \norm{Au}_{\Hh^{1}}^{2} \leq \epsilon \norm{Au}_{\Hh^{1}}^{2} +
    \frac{C}{\epsilon}\norm{G\doph u}^{2} 
    + Ch \norm{Gu}_{\Hh^{1}} +
    \norm{Au}^{2} + \norm{(hxD_{x} + z - \frac{h\charge}{x} +h\charge -
    ih)Au}^{2} \\ \qquad \qquad +  
    \bo (h^{\infty})\norm{\chi u}_{\Hh^{1}}^{2} + \bo
    (h^{\infty})\norm{u}_{\Hh^{-N}}^{2}
\end{gathered}  \end{equation*}

  Using the triangle and Hardy inequalities together with the estimate
  \eqref{eq:bounding-hdx-term} yields
  \begin{equation*}
    \norm{Au}_{\Hh^{1}} \leq \left( 4\charge^{2} + \epsilon +
      C\delta^{2}\right)\norm{Au}_{\Hh^{1}} + C\norm{G\doph u}^{2} +
    Ch\norm{Gu}_{\Hh^{1}}^{2} + \bo(h^{\infty})\left( \norm{\chi
        u}_{\Hh^{1}}^{2} + \norm{u}_{\Hh^{-N}}^{2}\right).
  \end{equation*}
  Provided that $\abs{\charge}<1/2$ and $\epsilon, \delta$ are
  sufficiently small, the first term can be absorbed into the left,
  finishing the proof of the estimate.  
  
  To finish the proof of the proposition, we note that on the support
  of $A$ and $G$, the $\cX^{s}_{h}$ norm is equivalent to the
  $\Hh^{1}$ norm while the $\cY^{s}_{h}$ norm is equivalent to the
  $L^{2}_{g}$ norm.  Finally, the error terms are both controlled by
  the $\cX^{s}_{h}$ norm provided that $N$ is sufficiently large.
\end{proof}

\subsubsection{Hyperbolic estimates near the singularity}
\label{sec:hyperb-estim-near-bdry}

In this section we establish the semiclassical propagation result for
$\doph$ in the hyperbolic regime near the singularity.  The result for
the adjoint problem is essentially identical.  In particular, we prove
the following:
\begin{proposition}
  \label{prop:main-hyp-semic-statement}
  If $G\in \Psibh$ is elliptic at $\Sigmadot_h \cap \{ x=0\}$ then
  there are $Q, Q_{1} \in \Psibh$ with $Q$ also elliptic at $\Sigmadot_h
  \cap \{ x=0\}$ and
  \begin{align*}
    \WFbh' (Q) &\subset \ellbh(G), \\
    \WFbh'(Q_{1}) &\subset \ellbh(G) \cap \{ - \xi > 0\},
  \end{align*}
  so that for all $u \in \cX^{s}_{h}$,
  \begin{equation*}
    \norm{Qu}_{\cX^{s}_{h}} \leq \frac{C}{h}\norm{G\doph
      u}_{\cY^{s}_{h}} + C_{1}\norm{Q_{1}u}_{\cX^{s}_{h}} +
    Ch^{1/2}\norm{Gu}_{\cX^{s}_{h}} + \bo(h^{\infty})\norm{u}_{\cX^{s}_{h}}.
  \end{equation*}
\end{proposition}
As the characteristic set near the singularity is in a bounded region of
phase space, we need only pseudodifferential operators with
\textit{compactly supported} symbols in this section.  Recall that we
use the notation $\Psibh$ (without a subscript) to denote this space.

In an approach similar to the one taken in
Section~\ref{sec:hyperb-prop}, we introduce an operator
$A \in \Psibh$ with compactly supported symbol given by
\begin{equation*}
  a = \chi_{0} (2-\phi / \delta) \chi_{1}(2 - \xi/\delta)
  \chi_{2}(\xi^{2} + \abs{\eta}^{2}),
\end{equation*}
where $\chi_{\bullet}$ are the same functions as in that section and
here $\phi = - \xi + \frac{1}{\beta^{2}\delta}x^{2}$.  For convenience
we recall the relevant properties of the $\chi_{\bullet}$:
\begin{itemize}
\item $\ds \chi_{\bullet}$ and $\ds \chi_{\bullet}'$ have smooth square
  roots,
\item $\ds \chi_{0}$ is supported in $[0,\infty)$ with $\chi_{0}(s) =
  \exp (-1/s)$ for $s > 0$,
\item $\ds \chi_{1}$ is supported in $[0,\infty)$ with $\chi_{1}(s)=1$ for
  $s \geq 1$ and $\chi_{1}' \geq 0$, and 
\item $\ds \chi_{2}$ is supported in $[-2c_{1},2c_{1}]$ and is equal
  to $1$ on $[-c_{1},c_{1}]$.
\end{itemize}
We think of the symbol $a$ as being determined by the three localizing
parameters $c_{1}$, $\beta$, and $\delta$.

We further choose an invariant operator $Q\in \Psibh$ with compactly
supported symbol given by
\begin{equation}
  \label{eq:definition-of-qh}
  q = \frac{\sqrt{2}}{\sqrt{\delta}}\left( \chi_{0}\chi_{0}'\right)^{1/2}\chi_{1}\chi_{2},
\end{equation}
where the arguments of $\chi_{\bullet}$ are the same as those in the
definition of $a$.  The symbol $q$ will appear when derivatives land
on the $\chi_{0}$ term in $a$.

\begin{lemma}
  \label{lem:semic-comm-comp}
  For $A$ and $Q$ defined as above,
  \begin{equation*}
    \frac{1}{ih}\left[ \doph , A^{*}A\right] = \tR \doph +Q \left(
      hxD_{x} + z +h\charge - ih\right) Q + B_{0} + \alpha_{r}B_{1} + E' + E'' +
    h\bR,
  \end{equation*}
  where
  \begin{itemize}
  \item All listed pseudodifferential operators have compact support,
  \item $Q\in \Psibh$ is invariant, self-adjoint, and has principal
    symbol $q$ defined in equation~\eqref{eq:definition-of-qh},
  \item $\ds \tR \in \Psibh$, with $\sigmabh (\tR) = -
    \pd[\xi](a^{2})$,
  \item $B_{0}, B_{1}\in \Psibh$ have principal symbols bounded by $C\beta^{-1}q^{2}$,
  \item $\ds E' \in \Psibh$, with $\WFbh'(E')\subset \{ \delta \leq
    \xi \leq 2\delta, x^{2}< 4\delta^{2}\beta^{2}\}$, 
  \item $\ds E'' \in \frac{1}{x}\Psibh$, with $\WFbh'(E'')\cap
    \Sigmadot_{h}=\emptyset$, and
  \item $\ds \bR \in \frac{1}{x}\Psibh$.
  \end{itemize}
\end{lemma}

\begin{proof}
  We carefully apply Lemma~\ref{lem:general-semic-comm-L} with $C=
  A^{*}A$.  The main term in Lemma~\ref{lem:general-semic-comm-L} (the
  one involving $A_{0}$) is due to the near-homogeneity of the
  operator $\doph$ in $x$.  The principal symbol of $A_{0}$ there is
  $-\pd[\xi](a^{2})$; we use that its coefficient includes all of the
  homogeneous terms of degree $-1$ in $\doph$ and trade it for one of
  the form
  \begin{equation*}
    \tR\doph + A_{1} \left( h xD_{x} + z +h\charge - ih\right) A_{1},
  \end{equation*}
  where $\sigmabh(A_{1}^{2})= \pd[\xi](a^{2})$.  We now split $A_{1}$
  into three terms according to where the $\xi$-derivative lands.
  Those terms where the derivative falls on $\chi_{0}$ we write as
  $Q^{2}$ modulo a lower order error (which we absorb into $h\bR$).
  Those where the derivative falls on $\chi_{1}$ are absorbed into
  $E'$ and those for which it falls on $\chi_{2}$ are absorbed into
  $E''$.  

  The $B_{1}$ term arising in Lemma~\ref{lem:general-semic-comm-L}
  has principal symbol
  \begin{equation*}
    -\pd[x](a^{2}) = -2\chi_{0}\chi_{0}' \chi_{1}^{2}\chi_{2}^{2}\cdot
    \left( \frac{-1}{\delta} \right)\left(\frac{2}{\beta^{2}\delta}x\right).
  \end{equation*}
  As $0\leq x \leq 2\beta\delta$ on the support of $a$, this term can
  be estimated by a multiple of
  $\beta^{-1}\delta^{-1}\chi_{0}'\chi_{0}\chi_{1}^{2}\chi_{2}^{2}$ and
  hence by a multiple of $\beta^{-1}q^{2}$.  The $B_{0}$ term there is
  estimated similarly (and in fact is even smaller owing to the
  additional factor of $x$.)

  The $\bB_{2}$ and $\bB_{3}$ terms in
  Lemma~\ref{lem:general-semic-comm-L} have symbols proportional to
  $\pd[\eta](a^{2})$ and so are absorbed into $E''$.  The remaining
  terms constitute the $\bR$ term.
\end{proof}

We now record a few consequences of the symbol calculus:
\begin{lemma}
  \label{lem:sc-bound-remainder}
  With $A$, $Q$, $B_{0}$, and $B_{1}$ as above, there are positive
  constants $C$ and $c$ so that the following estimates hold:
  \begin{align*}
    \norm{Au} &\leq C\sqrt{\delta}\norm{Qu} +
                \bo(h^{\infty})\norm{u}_{\cX^{s}_{h}}, \\
   \abs{\ang{B_{0}u,u}} &\leq C\beta^{-1}\norm{Qu}^{2}+
                \bo(h^{\infty})\norm{u}_{\cX^{s}_{h}}^{2} ,\\
    \abs{\ang{B_{1}u,u}} &\leq C \beta^{-1}\norm{Qu}^{2} +
                    \bo(h^{\infty})\norm{u}_{\cX^{s}_{h}}^{2} , \\
    \abs{\ang{(hxD_{x}+ z + h\charge - ih)Qu, Qu}} &\geq \abs{z}\norm{Qu}^{2} -
    c\delta\norm{Qu}^{2} - \bo(h)\norm{Qu}^{2} -
                \bo(h^{\infty})\norm{u}_{\cX^{s}_{h}} .
  \end{align*}
\end{lemma}

We now finish the proof of
Proposition~\ref{prop:main-hyp-semic-statement}.  

\begin{proof}
  We first observe that $\doph^{*} = \doph -2\Im z + 2ih = \doph + \bo
  (h)$.

  Given $u \in \cX^{s}$, we apply Lemma~\ref{lem:semic-comm-comp} and write
  \begin{align*}
    -\frac{2}{h}\Im\ang{A\doph u, Au} &= \frac{1}{ih}\ang{\left[\doph,
                                        A^{*}A\right]u, u} + \bo (1)
                                        \norm{Au}^{2} \\
    &= \ang{\tR \doph u, u} + \ang{(h xD_{x} + z - ih)Qu, Qu} +
      \ang{(B_{0}+\alpha_{r}B_{1})u, u} \\
    &\quad + \ang{E'u,u} + \ang{E''u,u} + h\ang{\bR
      u,u} + \bo(1)\norm{Au}^{2}.
  \end{align*}
  Because $z = \pm 1 + \bo (h)$, applying
  Lemma~\ref{lem:sc-bound-remainder} shows that
  \begin{align*}
    \norm{Qu}^{2} &\leq \frac{C}{h}\abs{\ang{A\doph u, Au}} + \left(
                    C\delta+ C\beta^{-1} + \bo
                    (h)\right)\norm{Qu}^{2} \\
    &\quad + \abs{\ang{\tR\doph u,u}} + \abs{\ang{E'u,u}} +
      \abs{\ang{E''u,u}} + h\abs{\ang{\bR u,u}} + \bo(h^{\infty})\norm{u}_{\cX_{h}^{s}}.
  \end{align*}

  We turn our attention to the second line.  The first term can be
  estimated by
  \begin{equation*}
    \frac{C}{h^{2}}\norm{G\doph u}^{2} + Ch \norm{Gu}^{2} +
    \bo(h^{\infty})\norm{u}_{\cX^{s}_{h}}^{2},
  \end{equation*}
  while the second term is bounded by
  \begin{equation*}
    \norm{Q_{1}}^{2} + Ch\norm{Gu}^{2} +\bo(h^{\infty})\norm{u}_{\cX^{s}_{h}}^{2}.
  \end{equation*}
  The last term can be similarly estimated by $Ch\norm{Gu}^{2} +
  \bo(h^{\infty})\norm{u}_{\cX^{s}_{h}}^{2}$.  We finally consider the
  term involving $E''$.  Writing
  \begin{equation*}
    E'' = \frac{1}{x}E_{0},
  \end{equation*}
  with $E_{0} \in \Psibh$ microsupported away from $\Sigmadot_h$, the
  Hardy inequality yields
  \begin{align*}
    \abs{\ang{\frac{1}{x}E_{0}u,u}} &\leq C\norm{\tG u}^{2} +
    C\norm{\frac{1}{x}\tG u}^{2} +
                                      \bo(h^{\infty})\norm{u}_{\cX^{s}_{h}}^{2} \\
    &= \frac{C}{h^{2}}\norm{\frac{h}{x}\tG u}^{2} + C\norm{\tG u}^{2} +
      \bo(h^{\infty})\norm{u}_{\cX^{s}_{h}}^{2} \\
    &\leq \frac{C}{h^{2}}\norm{\tG u}_{\cX^{s}_{h}}^{2} + \bo(h^{\infty}\norm{u}_{\cX^{s}_{h}}^{2},
  \end{align*}
  where $\tG$ is elliptic on the microsupport of $E_{0}$.  As $E_{0}$
  is microsupported in the elliptic set, we can guarantee that $\tG$
  is also microsupported in the elliptic set.
  Corollary~\ref{prop:better-h-factor-ellip} then applies to yield
  \begin{equation*}
    \abs{\ang{E'' u,u}} \leq \frac{C}{h^{2}}\norm{G\doph
      u}_{\cY^{s}_{h}}^{2} + Ch\norm{Gu}_{\cX^{s}_{h}}^{2} + \bo(h^{\infty})\norm{u}_{\cX^{s}_{h}}^{2}.
  \end{equation*}
  Taking $\delta$ small and $\beta$ large then provides the desired
  bound for $\norm{Qu}$.  
\end{proof}

\section{Polyhomogeneity}
\label{sec:polyhomogeneity}

In this section we show that the forward solution $\psi$ of
$i\dirac_{\charge/r}\psi = g$ with $g \in \CI_{c}$ is polyhomogeneous
at $\scri^{+}$ and $\mf$.

In particular we prove the following theorem:
\begin{theorem}
  \label{thm:phg-first-pass}
  Let $\cE$ denote the index set arising from the poles (with multiplicity) of
  $\dops^{-1}$ from Theorem~\ref{thm:bdry-fredholm}, i.e.,
  \begin{equation*}
    \cE = \{ (\sigma, m) : \sigma \text{ is a pole of
    }\dops^{-1}\text{ of order m}\},
  \end{equation*}
  set $\cE (\varsigma)$ to consist of those elements $(\sigma,m)\in
  \cE$ with $\Im \sigma < -\varsigma$, 
  and let
  $\cE_{0} = \{ (k,0) \in \complexes\times \naturals : k \in
  \naturals\}$ denote the index set describing smooth functions.
  If $\psi$ is the forward solution of
  $i\dirac_{\charge/r}\psi =g$ with $g\in \CI_{c}$, then there is a
  $\varsigma$ so that $\psi$ is
  polyhomogeneous on $[M;S_{+}]$ with index sets
  \begin{equation*}
    \begin{cases}
      \emptyset & \text{at } C_{-}\cup C_{0}\\
      -i(1 + i\charge) +\cE_{0} & \text{at }\scri^{+} \\
      -i(1+i\charge) + \cE(\varsigma) &\text{at }C_{+}
    \end{cases}.
  \end{equation*}
\end{theorem}

Conjugating and rescaling, we work instead
with $\dop u = f$, where $u = \rho^{-1-i\charge}\psi$ and
$f = \rho^{-2-i\charge}\gamma^{0}g \in \CI_{c}(M)$.  The argument in this
section closely mirrors the ``short-range'' case for the radiation
field in previous work~\cite{BVW18} and so we provide only an
abbreviated version.  In fact, the argument here is simplified in
comparison with~\cite{BVW15, BVW18}, owing to
the scale-invariance of the operator $\dop$ in $\rho$.  Indeed, because
$N(\dop) = \dop$, we are able to avoid the complications of the
remainder terms arising in the contour deformation argument.

We apply the Mellin transform (in $\rho$) to $\dop u = f$ to obtain
$\dops \us = \fs$, where $\dops = N(\dop)$ is the indicial
operator of $\dop$.  As we are considering the forward problem, $\psi$
is supported in the forward light cone from $g$, and we may therefore
assume that $\us$ is supported in $\overline{C_{+}}$.

The conormal spaces we consider are adapted to our variable-order
spaces $\cY^{s}$: 
\begin{definition}
  For a distribution $u$ on $\mf$, we say that $u \in
  \conormal{s}(S_{+})$ if $u \in \cY^{s}$ and $A_{1}\dots A_{k}u \in
  \cY^{s}$ for all $A_{j}\in \module$ with $\module$ as in Definition
  \ref{def:module} and $k \in \naturals$.
\end{definition}
In particular, elements of
$\conormal{s}(S_{+})$ lie in $\Hb^{1,\infty}$ away from $S_{+}$ and
are therefore smooth away from $S_{+}$ and $x=0$.  The order $s$
influences the regularity microlocally at $N^{*}S_{+}$, in the sense
that, if $u \in \conormal{s}(S_{+})$ and $s' > s$ then $\WFb^{1,s'}(u)
\subset N^{*}S_{+}$.

To assist in bookkeeping, we recall a bit of notation from previous
work~\cite{BVW18}:
\begin{definition}
  For $\varsigma, s \in \reals$, we let $\complexes_{\varsigma}$
  denote the upper half-plane $\Im \sigma > - \varsigma$ and define
  \begin{equation*}
    \cB(\varsigma, s) = \cH (\complexes_{\varsigma}) \cap
    \ang{\sigma}^{-\infty}L^{\infty}_{\Im \sigma}L^{2}(\reals, \conormal{s}(S_{+})).
  \end{equation*}
\end{definition}

As in \cite[Def.\ 2.1]{BVW15}, for any Fr\'echet
  space $\mathcal{F}$ and $k \in \mathbb{R}$, $\cH (\complexes_{\varsigma}) \cap
    \ang{\sigma}^{k}L^{\infty}_{\Im \sigma}L^{2}(\reals,
    \mathcal{F})$ is the space of holomorphic functions $g$
    of $\sigma$ 
    with values in $\mathcal{F}$ such that for each fixed line $\Im
    \sigma = C > - \varsigma$, with $\mu = \Re \sigma$, for each seminorm $| \cdot
    |_\bullet$ of $\mathcal{F}$, we have $\langle \mu \rangle^k |g(\mu + i
    C)|_{\bullet} \in L^2(d\mu)$ with $L^2$-norm bounded uniformly in
    $C$.

Given $f \in \CI_{c}(M)$, $\fs \in \cB(\varsigma, s)$ for all $\varsigma$
and $s$.  Our main use of Section~\ref{sec:propagation-bulk} is
summarized in the following lemma:
\begin{lemma}[{c.f.~\cite{BVW15,BVW18,BM19}}]
  \label{lem:output-of-bulk-sec}
  There are $\varsigma_{0}$ and $s_{0}$ so that
  \begin{equation*}
    \us \in \cB (\varsigma_{0}, s_{0}-0).
  \end{equation*}
\end{lemma}

The proof is essentially identical to one provided in previous
work~\cite[Section~9]{BVW15} and follows from the interpolation of the
propagation and module regularity estimates in
Section~\ref{sec:propagation-bulk} together with the mapping
properties of the Mellin transform.

We know already that $\dops \us =\fs$ and $\fs$ is entire.  As
$\fs \in \cY^{s}$ and $\us \in \cX^{s}$, we may apply
Theorem~\ref{thm:bdry-fredholm} to conclude that $\us$ is a
meromorphic function on the strip $s_{+}< \Im\sigma< s_{-}$ taking
values in $\cX^{s}$ provided that $s_{\tow}|_{S_{+}} \leq s_{+}$
and $s_{\tow}|_{S_{-}} > s_{-}$.  Shifting the contour, applying
the residue theorem, and interpolating then yields
\begin{equation*}
  \us \in \cB \left(\varsigma_0 + N, \min (s_{0}-0, \frac{1}{2}+
    \varsigma_{0}-N-0)\right) + \sum _{ (\sigma_{j}, m_{j}) \in \cE,
    \Im \sigma_{j}> -\varsigma_{0}-N} (\sigma - \sigma_{j})^{-m_{j}}a_{j},
\end{equation*}
for any $N$, where
\begin{equation*}
  a_{j} \in \cB(\varsigma_{0}+N, \frac{1}{2} + \Im \sigma_{j}-0).
\end{equation*}
Here the regularity of the $a_{j}$ follows from Cauchy's integral
formula.  Here $\cE$ denotes the set of poles of $\dops^{-1}$ lying
below our initial contour $\Im \sigma = -\varsigma_{0}$.  Inverting
the Mellin transform then yields an asymptotic expansion:
\begin{proposition}
  \label{prop:first-asymptotic}
  With $\cE$ the resonance index set and $\varsigma_{0}$ chosen to
  ignore those resonances where $\us$ is known to be holomorphic, we
  have
  \begin{equation*}
    u = \sum_{(\sigma_{j}, m)\in \cE, \Im \sigma_{j}> -\ell}
    \rho^{i\sigma_{j}}(\log \rho)^{m}a_{jm} + u',
  \end{equation*}
  where for $C = s_{0} + \varsigma_{0}$,
  \begin{equation*}
    u' \in \rho^{\ell} \Hb^{1,\min (C- \ell - 1 - 0, -1/2 -
      \varsigma_{0}-\ell-0), \gamma}.
  \end{equation*}
  The coefficients $a_{jm}$ are smooth functions of $\rho$ taking
  values in $\conormal{1/2-\Re(i\sigma_{j})-0}$.
\end{proposition}

\begin{remark}
  Because $\us$ is supported in $\overline{C_{+}}$, the $a_{jm}$ are
  also supported in $\overline{C_{+}}$ for $\rho$ near $0$.  We see
  in Section~\ref{sec:char-expon} below that in fact all $m=0$.
\end{remark}

One consequence of Proposition~\ref{prop:first-asymptotic} is that
\begin{equation*}
  u' = \left( \prod_{(\sigma_{j},m) \in \cE(\varsigma_{0}),
     \Im\sigma_{j} > -\ell} (\rho D_{\rho} - \sigma_{j})\right) u \in
 \rho^{\ell}\Hb^{1,\min (C- \ell-1-0, -1/2-\varsigma_{0}-\ell-0),\gamma}.
\end{equation*}
By Theorem~\ref{thm:bulk-regularity}, $u$ enjoys module regularity
with respect to $\rho^{\ell}\Hb^{1,s',\gamma}$ for some $s'$ and so we
can interpolate to show that in fact $u'$ enjoys module regularity
with respect to the space 
$\rho^{\ell}\Hb^{1,\min(C-\ell-1-0,-1/2-\varsigma_{0}-\ell-0),\gamma}$.

The module $\module$ includes a basis for $\cV_{\bl}$ with the
exception of $\pd[v]$, though $v\pd[v]$ lies in $\module$, leading to
\begin{lemma}
  \label{lem:conversion-of-module}
  For all $N$ and $\ell$, if $\module^{N+\ell}w\in \Hb^{1,p,\gamma}$, then $\module^{N}v^{\ell}w\in
  \Hb^{1,p+\ell,\gamma}$.  
\end{lemma}

We may then commute $\rho^{-\ell}$ through the module factors to see
that
\begin{equation*}
  \module^{N}\rho^{-\ell}v^{\ell}\left( \prod_{(\sigma_{j},m)\in \cE
      (\varsigma_{0}), \Im \sigma_{j} > -\ell} (\rho D_{\rho} -
    \sigma_{j})\right)u \in \Hb^{1,\min (C-1-0,-1/2-\varsigma_{0}-0), \gamma}.
\end{equation*}
In other words, if we set $\varpi = \abs{\rho / v}$ to be the defining
function of the faces $C_{+}$ and $C_{0}$ in the blow-up $[M ;
S_{+}]$, then we have the following result:
\begin{proposition}
  \label{prop:first-expansion-on-blow-up}
  On $C_{+}$, uniformly up to the corner $C_{+}\cap \scri^{+}$ in $[M;
  S_{+}]$, $u$ enjoys an asymptotic expansion with powers given by the
  resonance index set:
  \begin{equation}
    \label{eq:polyhom-at-C}
    \left( \prod_{(\sigma_{j}, m)\in \cE (\varsigma_{0}), \Im
        \sigma_{j} > - \ell}(\varpi D_{\varpi}-\sigma_{j})\right)u \in
    \varpi^{\ell}\Hb^{1,\infty, *, *}([M;S]),
  \end{equation}
  where the $*$'s represent fixed (independent of $\ell$) growth
  orders at $C_{+}$ and $\scri^{+}$.
\end{proposition}

We now turn our attention to the expansion at $\scri^{+}$.  The
following lemma is instrumental in that discussion:
\begin{lemma}
  \label{lem:Psigma-in-terms-of-module}
  With $\module$ denoting the module above,
  \begin{equation*}
    \Pd + 2 D_{v}\left( \rho D_{\rho}+ v D_{v}\right) \in \module^{2}.
  \end{equation*}
\end{lemma}
At this stage, the argument proceeds exactly as in previous
work~\cite[Section~9.2]{BVW18}, commuting powers of the radial vector
field through the operator and then integrating from $v=-\rho /2$ to show that
\begin{equation*}
  \module^{N}\left( \prod_{j=0}^{k}(\rho D_{\rho} + vD_{v} +
    ji)\right)u \in (v + C\rho)^{k+1} \Hb^{1,s-1,\gamma} (M)
\end{equation*}
for all $k$ and $N$.  Lifting to the blown up space $[M; S_{+}]$ then yields
\begin{equation}
  \label{eq:polyhom-at-scri}
  \left( \prod _{j=0}^{k}(\rho D_{\rho} + vD_{v} + ji)\right)u \in (v + C\rho)^{k+1}\Hb^{1,\infty,*,*}([M,S_{+}]),
\end{equation}
for some fixed ($k$-independent) weights $*$.  Applying
Lemma~\ref{lem:double-phg} with the results of
equations~\eqref{eq:polyhom-at-C} and~\eqref{eq:polyhom-at-scri}
yields the result.

\section{Characterization of the exponents}
\label{sec:char-expon}

Having established that the solution is polyhomogeneous at $C_{+}$ and
$\scri_{+}$, we now aim to characterize the exponents $\sigma_{j}$ in
the expansion at $C_{+}$.  As polyhomogeneity is preserved under
coordinate changes, we use a more convenient system of coordinates.

In particular, in a neighborhood of $\overline{C_{+}}$, we want to use
coordinates $(\rhot, \xt, \theta)$, where
\begin{equation}
  \label{eq:coords-in-hypergeom}
  \rhot = \frac{1}{t+r},\quad \xt = \frac{2r}{t+r},\quad  \theta \in \sphere^{2}.
\end{equation}
Similarly, near $\overline{C_{-}}$, it is convenient to use the
coordinate system with
\begin{equation*}
  \rhott = \frac{1}{r-t}, \quad \xtt = \frac{2r}{r-t}, \quad \theta
  \in \sphere^{2}.  
\end{equation*}
As $\dops$ depends on the choice of defining function $\rhot$, let
us now, in an abuse of notation, fix a function $\trho$ near $\mf$ so that
\begin{equation*}
  \trho =
  \begin{cases}
    \rhot = \frac{1}{r+t} & t > \frac{1}{2}r \\
    \rhott = \frac{1}{r-t} & t < - \frac{1}{2}r,
  \end{cases}
\end{equation*}
and so that $\trho$ remains a defining function for the boundary for
$|t|\leq \frac{1}{2}r$.  We may even guarantee that $\trho$ is
homogeneous of degree $-1$ in $(t,r)$ near $\trho = 0$.  Note that this
is necessarily a different boundary defining function than the one
discussed in the previous sections.  

In another abuse of notation, we now redefine the operator $\dops$ of the
previous sections to be the normal operator
\begin{equation*}
  \dops = \widehat{N}(\trho^{-2-i\charge}i\gamma^{0}\dirac_{V} \trho^{1+i\charge}),
\end{equation*}
which is related to our earlier operator by changes of coordinates,
conjugation by a non-vanishing smooth function, and a change
in the definition of the Mellin transform (from $\rho$ to $\trho$).
In particular, the Fredholm properties of $\dops$ and the locations of
the poles of its inverse are unchanged.  The operator $\dops$ is
conformally ``Dirac-type'' in the sense that its square has principal
part which is conformal to the hyperbolic Laplacian, as can be seen
from setting $V = 0$ and recalling from \cite[Section 5]{Vasy13} that
the Mellin-transformed operator
$\widehat{\rho^{-2}\Box}(\sigma)$ is conformally equivalent to
$\Delta_{\mathbb{H}^{3}} - \overline{\sigma}^2 - 1$ on the cap $C_+$.
Whether $\dops$ is given (conformally) by a canonical Dirac operator
on $\mathbb{H}^3$ plus a potential is irrelevant for our purposes
here. 

In a still further abuse of notation, we use the letter $x$ to denote both the
coordinates $\xt$ and $\xtt$ above in the relevant regions of interest.

With our choice of $\trho$ above, we now record the form of the
operator $\dops$ in the caps $C_{\pm}$.
In a neighborhood of $\overline{C_{+}}$, we have
\begin{equation*}
  \dops = xD_{x} + \sigma + \charge - i - \frac{2\charge}{x} - \alpha_{r} \left(
    (2-x)D_{x} - \frac{2i}{x} + \frac{2i}{x}\beta K - \sigma - \charge + i\right),
\end{equation*}
while in a neighborhood of $\overline{C_{-}}$, the signs of  the first
three terms flip:
\begin{equation*}
  \dops = - xD_{x} - \sigma - \charge + i - \frac{2\charge}{x} -
  \alpha_{r}\left( (2-x)D_{x} - \frac{2i}{x} + \frac{2i}{x}\beta K -
    \sigma -\charge + i\right) .
\end{equation*}
Before continuing, we further observe (though we do not record it
explicitly), that $\dops$ is a hyperbolic operator on the interior of $C_{0}$.

We now use separation of variables to show that the operators $\dops$
reduce to hypergeometric operators in $C_{\pm}$.  Recall from
Section~\ref{sec:separation-variables} that for $\kappa \in \integers
\setminus \{0\}$, the $-\kappa$ eigenspace of $K$ is spanned by
\begin{equation*}
  a\begin{pmatrix}
     \Omega _{\kappa, \mu} \\  \Omega_{-\kappa, \mu}
  \end{pmatrix} + b
  \begin{pmatrix}
    \Omega_{\kappa, \mu} \\ -\Omega_{-\kappa, \mu}
  \end{pmatrix}
 \equiv a \bOmega_{-} + b \bOmega _{+},
\end{equation*}
where $\Omega_{\kappa,\mu}$ are spinor spherical harmonics and
$\bOmega_{\pm}$ lies in the $\pm 1$-eigenspace of $\alpha_{r}$.  The
operator $\dops$ respects these eigenspaces and we may therefore
compute the action of $\dops$ on an element of the eigenspace.  In a
neighborhood of $\overline{C_{+}}$, this has the form
\begin{align}
  \label{eq:Ls-near-Cplus-eig}
  \dops a
  \bOmega_{-}  &= 2\left( D_{x} - \frac{i}{x} - \frac{\charge}{x}\right)a
                \bOmega_{-}  + \frac{2i\kappa}{x} a
                \bOmega_{+} \\
  \dops b
  \bOmega_{+}  &=
                -2\left( (1-x)D_{x} - \frac{i}{x} + \frac{\charge}{x} -
                (\sigma + \charge - i) \right)b 
                \bOmega_{+}  - \frac{2i\kappa}{x} b
                \bOmega_{-} \notag
\end{align}
Similarly, in a neighborhood of $\overline{C_{-}}$, it has the form
\begin{align}
  \label{eq:Ls-near-Cminus-eig}
  \dops a
  \bOmega_{-}  &=
                 2 \left( (1-x) D_{x} - \frac{i}{x} - \frac{\charge}{x} - (\sigma +
                 \charge -
                 i)\right) a
                 \bOmega_{-}  + \frac{2i\kappa}{x} a
                 \bOmega_{+}\\
  \dops b
  \bOmega_{+}  &= 
                 -2\left( D_{x} - \frac{i}{x} + \frac{\charge}{x} \right) b
                 \bOmega_{+}   - \frac{2i\kappa}{x}b
                 \bOmega_{-} \notag
\end{align}

In particular, separating into eigenspaces of $K$, if
\begin{equation*}
  \dops (a
  \bOmega_{-}  + b 
  \bOmega_{+} ) = h_{1}
  \bOmega_{-}  + h_{2}
  \bOmega_{+},
\end{equation*}
then $a$ and $b$ must solve a pair of coupled ordinary differential
equations.  

Indeed, after setting $\nu = \sqrt{\kappa^{2} - \charge^{2}}$ and
\begin{equation*}
  F = x^{1-\nu} a, \quad G = x^{1-\nu}b, \quad H_{1} =
  \frac{i}{2}x^{1-\nu}h_{1}, \quad H_{2} = -\frac{i}{2} x^{1-\nu}h_{2},
\end{equation*}
these equations have a particularly nice form.

Near $\overline{C_{+}}$, these equations have the form
\begin{align*}
  \left( \pd[x] + \frac{\nu - i \charge}{x}\right) F +
  \frac{\kappa}{x} G &= H_{1}, \\
  \left( (1-x)\pd[x] + \frac{\nu + i \charge}{x} - (\nu + i
  \sigma + i \charge)\right) G + \frac{\kappa}{x}F&= H_{2}.
\end{align*}
Similarly, near $\overline{C_{-}}$, they have the form
\begin{align*}
  \left( (1-x)\pd[x] + \frac{\nu-i\charge}{x}- (\nu + i\sigma + i \charge)\right)F +
  \frac{\kappa}{x} G &= H_{1}, \\
  \left( \pd[x] + \frac{\nu+i\charge}{x}\right) G + \frac{\kappa}{x}F &=
                                                                     H_{2}
\end{align*}

After substituting one equation into the other, we find that $F$ and
$G$ satisfy decoupled second order equations on $C_{\pm}$.

In a neighborhood of $\overline{C_{+}}$, $F$ and $G$ satisfy the following equations:
\begin{align}
  \label{eq:hypgeom-in-cplus}
  &x(1-x) \pd[x]^{2}F + \left( 1 + 2\nu - x (1 + 2\nu + i\sigma)\right) \pd[x]F - (\nu + i \sigma + i \charge) (\nu - i \charge)F  \\
  & \quad = (1-x)x\pd[x] H_{1} + (1+\nu + i\charge - x(1+\nu +
    i\sigma + i \charge))H_{1} - \kappa H_{2}, \\
  &x(1-x) \pd[x]^{2}G + \left( 1 + 2\nu - x (2 + 2\nu + i\sigma)\right) \pd[x]G - (\nu + i \sigma + i \charge)(1 + \nu - i\charge)G
  \\
  &\quad = x\pd[x] H_{2} + (1 + \nu - i \charge)H_{2} - \kappa H_{1}. \notag
\end{align}
In other words, $F$ and $G$ satisfy inhomogeneous hypergeometric
equations with parameters $(a,b,c)$ and $(a+1, b, c)$ respectively, where
\begin{equation*}
  a = \nu - i\charge, \quad b = \nu + i\sigma + i \charge, \quad c = 1 + 2\nu .
\end{equation*}

Near the other cap $\overline{C_{-}}$, $F$ and $G$ satisfy a slightly
different pair of hypergeometric equations:
\begin{align*}
  &x(1-x)\pd[x]^{2}F + \left(1 + 2\nu - x(2 + 2\nu + i \sigma +
    2i\charge)\right)\pd[x]F - (\nu + i\sigma + i \charge)(1 + \nu + i \charge)F \\
  &\quad = x\pd[x]H_{1} + (1 + \nu + i \charge)H_{1} - \kappa H_{2},
  \\
  &x(x-1)\pd[x]^{2}G + \left( 1 + 2\nu - x (1 + 2 \nu + i\sigma + 2i
    \charge)\right)\pd[x]G - (\nu + i\sigma + i \charge)(\nu + i \charge)G \\
  & \quad = (1-x)x\pd[x]H_{2} + (1 + \nu + i \charge - x(1+ \nu + i
    \sigma+ i \charge))H_{2} - \kappa H_{1}.
\end{align*}
This is again a pair of inhomogeneous hypergeometric equations for $F$
and $G$ but now with parameters $(a+1,b,c)$ and $(a,b,c)$
respectively, with
\begin{equation*}
  a = \nu + i \charge , \quad b = \nu + i \sigma + i \charge, \quad c = 1 + 2\nu.
\end{equation*}

We now exploit the hypergeometric structure of $\dops$ to
characterize the support of $\dops^{-1}f$ when $f$ is supported in $\overline{C_+}$.
\begin{lemma}
  Suppose $\sigma$ is a regular point of $\dops^{-1}$.  If $f$ is supported in
  $\overline{C_{+}}$, then $\dops^{-1}f$ is also supported in $\overline{C_{+}}$.
\end{lemma}

\begin{proof}
  Suppose $u = \dops^{-1}f$.  We proceed by separation of variables
  and decompose both $u$ and $f$ into spinor spherical harmonics;
  without loss of generality, we may assume both $u$ and $f$ lie in the
  span of
  \begin{equation*}
    \bOmega_{-}  =
    \begin{pmatrix}
      \Omega_{\kappa, \mu} \\ \Omega_{-\kappa, \mu}
    \end{pmatrix}
    \text{ and }\bOmega_{+} =
    \begin{pmatrix}
      \Omega_{\kappa, \mu}\\ -\Omega_{-\kappa,\mu}
    \end{pmatrix}
  \end{equation*}
  for $\kappa$ and $\mu$ fixed.

  Consider first the solution in $C_-$ and write $u$ in terms of the
  above basis of spinor spherical harmonics.  By the above
  computation, a rescaling of the components of $u$ in $C_-$ must
  solve homogeneous hypergeometric equations as $f$ vanishes
  identically in $C_-$.

  As $u$ must lie in $L^{2}$ near $x=0$, the $\bOmega_{-}$ component
  of $x^{1-\nu}u$ must be a multiple of the hypergeometric function
  $F(a+1, b, c; x)$, where $a = \nu + i\charge$,
  $b = \nu + i\sigma + i \charge$, and $c = 1 + 2\nu$.  Similarly, the
  $\bOmega_{+}$ component of $x^{1-\nu}u$ must be a multiple of
  $F(a,b,c;x)$.

  By contrast, as $u$ must be regular at $x=1$ (i.e., at $S_{-}$), the
  corresponding components of $x^{1-\nu}u$ must be multiples of
  $F(a+1, b, a+b+2-c; 1-x)$ and $F(a,b,a+b+1-c; 1-x)$ respectively.
  As these pairs of solutions of the corresponding equations are
  linearly independent for a dense open subset of $\sigma$ with $\Im
  \sigma > 0$ and $\dops^{-1}$ is meromorphic, $u$ must vanish
  identically in $\overline{C_{-}}$.

  In fact, $u$ must vanish identically in a neighborhood of $S_{-}$ as
  well because the hypergeometric form of the equation continues to
  hold there in our explicit coordinate systems.  We may thus conclude
  that $u$ vanishes on $C_{0}$ as well by finite speed of propagation,
  as $\dops$ is a hyperbolic operator in $C_{0}$ and $u$ vanishes on a
  Cauchy surface.
\end{proof}

\begin{lemma}
  For $0 < \abs{\charge}< 1/2$, and $f$ compactly supported in the
  interior of $C_{+}$, the poles of $\dops^{-1}f$ occur at
  \begin{equation*}
    -i (1 + \nu + m), \quad \nu = \sqrt{\kappa^{2}-\charge^{2}}, m = 0, 1, 2, \dots.
  \end{equation*}
  If $\charge = 0$, $\dops^{-1}f$ has no poles.
\end{lemma}

\begin{proof}
  We proceed by explicitly characterizing the kernel of $\dops^{-1}$.

  We first consider $\Im \sigma \gg 0$.  As $f$ is compactly supported
  in $C_+$, the computations above show that near $S_+$ but within
  $C_+$, the components of $x^{1-\nu}u$ must be linear combinations of
  hypergeometric functions.  In particular, the
  $\bOmega_{-}$  component of $x^{1-\nu}u$ must be a linear combination of the
  following two hypergeometric functions:
  \begin{align*}
    F(\nu - i \charge, \nu + i \sigma + i \charge, i\sigma ; 1 - x),
    \\
    (1-x)^{1 - i \sigma + i \charge} F(1 + \nu + i \charge, 1 + \nu -
    i \sigma - i \charge, 2  - i \sigma ; 1-x).
  \end{align*}

  Observe that the function given on $(0,\infty)$ by
  \begin{equation*}
    \begin{cases}
      F(\nu - i \charge , \nu + i \sigma + i \charge, i \sigma ; 1-x)
      & 0 < x < 1 \\
      0  & x \geq 1
    \end{cases},
  \end{equation*}
  is not regular enough to lie in the image of $x^{1-\nu}\dops^{-1}$ for $\Im
  \sigma \gg 0$.  In
  particular, this component of $u$ must in fact be a multiple of
  \begin{equation*}
    (1-x)^{1-i\sigma} F(1+\nu + i \charge, 1 + \nu - i
    \sigma - i \charge , 2 - i \sigma; 1-x).
  \end{equation*}

  Near $x=0$, $u$ must lie in $L^{2}$ and so the components of
  $x^{1-\nu}u$ must be multiples of the regular solution there, e.g.,
  the $\bOmega_{-}$ component of $x^{1-\nu}u$ must be a multiple of
  \begin{equation*}
    F(\nu - i \charge, \nu + i \sigma + i \charge, 1 + 2\nu ; x)
  \end{equation*}
  near $x=0$.

  Using the notation of the calculation above, we may therefore write
  down explicitly each component of $u=\dops^{-1}f$; for $f$ in a
  single $(\kappa, \mu)$-eigenspace of $K$,
  \begin{equation*}
    \dops^{-1}f = \dops^{-1}\left( f_{1}
      \bOmega_{-}      + f_{2}
      \bOmega_{+}    \right) = u_{-}
    \bOmega_{-}    + u_{+}
    \bOmega_{+}
  \end{equation*}
  where $u_{+}$ and $u_{-}$ are given explicitly.  

  The coefficient $u_{-}$ is given in terms of the following functions
  (whose names are consistent with the DLMF~\cite[15.10(ii)]{DLMF}):
    \begin{align*}
    w_{1}&= F( \nu - i \charge, \nu + i \sigma + i \charge , 1 + 2\nu ; x) \\
    w_{4} &= (1-x)^{1-i\sigma} F(1 + \nu + i \charge, 1 +
            \nu - i \sigma - i \charge, 2 - i \sigma ; 1-x),
  \end{align*}
  where the $F$ functions denote hypergeometric functions of the given
  parameters and arguments.  One consequence of Kummer's connection
  formulas~\cite[15.10(ii)]{DLMF} is a formula for the Wronskian of
  $w_{1}$ and $w_{4}$:
  \begin{equation*}
    w_{1}(x) w_{4}'(x) - w_{1}'(x) w_{4}(x) = -\frac{\Gamma(1+2\nu)
      \Gamma (2 - i \sigma)}{\Gamma(1 + \nu + i
      \charge)\Gamma(1 + \nu - i \sigma - i \charge)} x^{-1-2\nu} (1-x)^{ - i \sigma}.     
  \end{equation*}

  We use that $x^{1-\nu}u_{-}$ satisfies the
  hypergeometric equation~\eqref{eq:hypgeom-in-cplus} ($F$ in the
  above calculation) and must have the $w_{1}$ behavior near $x=0$ and the
  $w_{4}$ behavior at $x=1$.  A standard variation of parameters
  argument then shows that
  \begin{align*}
    x^{1-\nu}u_{-} =-\frac{\Gamma (1 + \nu + i \charge)\Gamma ( 1 + \nu
      - i \sigma - i \charge)}{\Gamma(1 + 2\nu)\Gamma ( 2 - i \sigma)}
    \left(w_{4}(x) \int_{0}^{x}y^{1+2\nu}(1-y)^{i\sigma}w_{1}(y)f_{+}(y) \,dy \right. \\
    \left. - w_{1}(x)
      \int_{1}^{x}y^{1+2\nu}(1-y)^{i\sigma} w_{4}(y)f_{+}(y)\,dy \right),
  \end{align*}
  where, for ease of display, we have introduced notation for the
  inhomogeneous term in equation~\eqref{eq:hypgeom-in-cplus}:
  \begin{equation*}
    f_{+} = \frac{i}{2} \pd[x] (x^{1-\nu}f_{1}) + \frac{i}{2}\frac{1 +
      \nu + i \charge}{x}(x^{1-\nu}f_{1}) + \frac{i}{2}\frac{- i \sigma}{1-x} (x^{1-\nu}f_{1}) + \frac{i}{2} \frac{\kappa f_{2}}{x(1-x)}.
  \end{equation*}
  A similar expression\footnote{This expression could also be derived
    from formulas for contiguous hypergeometric functions.} is available for $u_{+}$, which satisfies a
  different hypergeometric equation:
  \begin{align*}
    x^{1-\nu}u_{+} = -\frac{\Gamma (\nu + i \charge)\Gamma (1 + \nu -
                     i \sigma- i\charge)}{\Gamma (1+2\nu)\Gamma(1 -
                     i \sigma)}\left(
                     \tw_{4}\int_{0}^{x}y^{1+2\nu}(1-y)^{1 + i \sigma}\tw_{1}(y)f_{-}(y)\,dy \right. \\
     \left( - \tw_{1}(x) \int_{1}^{x}y^{1+2\nu} (1-y)^{1+i\sigma}\tw_{4}(y) f_{-}(y)\right),
  \end{align*}
  where $f_{-}$ is the same rescaling of the right side of the second
  equation in~\eqref{eq:hypgeom-in-cplus} and $\tw_{1}, \tw_{4}$ are
  the analogous hypergeometric functions here:
  \begin{align*}
    f_{-} &= -\frac{i}{2 (1-x)}\pd[x](x^{1-\nu}f_{2}) -
            \frac{i}{2}\frac{1+\nu-i\charge}{x(1-x)}f_{2}-
            \frac{i}{2}\kappa f_{1}, \\
    \tw_{1} &= F(\nu - i \charge + 1 , \nu + i \sigma + i \charge, 1 + 2\nu ; x),
    \\
    \tw_{4} &= (1-x)^{- i \sigma} F(\nu + i \charge, 1 + \nu
              - i \sigma - i \charge, 1 - i \sigma ; 1-x).
  \end{align*}

  As these expressions are meromorphic in $\sigma$ and hold on a dense
  open subset of $\Im \sigma \gg 0$, they also characterize
  $\dops^{-1}f$ at all regular $\sigma$.

  In particular, the poles of $\dops^{-1}$ occur precisely when one of
  the leading coefficients has a pole, which occurs only when $1 + \nu
  - i\sigma - i \charge$ is a pole of the Gamma function,\footnote{If $\charge =
    0$, then the poles all cancel; this is a reflection of Huygens'
    principle.} i.e., only when
  \begin{equation*}
    \sigma = - \charge-i (1 + \nu + m), \quad m =0, 1, 2, \dots.
  \end{equation*}
\end{proof}

\appendix
\section{Estimates in the bulk near the singularity, and the proof of
  Theorem \ref{thm:bulk-prop-near-np}} \label{sec:bulk prop sing section }

We now generalize the results of \cite{BW20} to our setting.  The
proof of Theorem \ref{thm:bulk-prop-near-np} is summarized at the end
of Section \ref{sec:hyperb-prop}.

\subsection{Commutators with $\bl$-pseudodifferential operators}
\label{sec:commutators-with-bl}

The following lemma allows us to essentially ignore the distinction
between the operators $\dop$ (and $\Pd$) and their conjugated
counterparts.
\begin{lemma}
  \label{lem:easy-conjugation}
  For any $\ell \in\reals$,
  \begin{equation*}
    \rho^{-\ell}\Pd \rho^{\ell} - P \in \frac 1x \Psib^0 + \Psib^1
  \end{equation*}
  and
  \begin{equation*}
    \rho^{-\ell}\dop \rho^{\ell} - \dop \in \Psib^{0}.
  \end{equation*}
\end{lemma}

\begin{lemma}[{c.f.~\cite[Lemma 25]{BW20}}]
  \label{lemma:easy-bulk-commutator}
  If $C \in \Psib^{m}$ is invariant with scalar, real-valued principal
  symbol, then
  \begin{equation*}
    [P, C] \in \left\{ \frac{1}{x^2} \Delta_\theta , D_x^2, \frac 1x D_x,
      \frac{1}{x^2} \right\} \Psib^{m-1} + \left\{D_x, \frac 1x
    \right\} \Psib^{m} +  \Psib^{m+1},
  \end{equation*}
  with the microsupport of the right side contained in $\WFb'C$.

  Similarly,
  \begin{equation*}
    [\dop, C] \in \left\{ \frac{1}{x} \beta K , D_x, \frac 1x \right\}
    \Psib^{m-1} + \Psib^{m}, 
  \end{equation*}
  also with the microsupport of the right side contained in $\WFb'C$.
\end{lemma}

\begin{lemma}[{c.f.~\cite[Lemma 26]{BW20}}]
  \label{lemma:careful-bulk-commutator-general-new}
  If $C \in \Psib^{m}$ is invariant with scalar, real-valued principal
  symbol $\sigma_{\bl, m}(C) = c$, then near the north pole $\{ \rho =
  0, x=0, t > 0\}$,
  \begin{equation}
    \label{eq:list-of-all-terms-bulk-commutator}
    \frac{1}{i}\left[ \dop, C\right] = A_{0} \left( \alpha_{r}\left( i
        \pd[x] + \frac{i}{x} -
        \frac{i}{x}\beta K
      \right) - \frac{\charge}{x} \right) + B_{0} + \alpha_{r}B_{1} +
    \bB_{2} \frac 1r +     \bB_{3} D_r + \bB_{4}
  \end{equation}
  where
  \begin{itemize}
  \item $\ds A_{0} \in \Psib^{m-1}$, with $\ds \sigmab(A_{0}) = -
    \pd[\xi]c$,
  \item $\ds B_{0} \in \Psib^{m}$, with $\ds \sigmab(B_{0}) =
    -\rho\pd[\rho]c-x\pd[x]c$,
  \item $\ds B_{1} \in \Psib^{m}$, with $\ds\sigmab(B_{1}) =
    \pd[x]c$, and
  \item $\ds \bB_{2} \in \Psib^{m}$, with $\ds \supp \sigmab(\bB_{2})
    \subset \supp \pd[\eta]c$.
      \item $\ds \bB_{3} \in \Psib^{m-1}$, with $\ds \supp \sigmab(\bB_{2})
    \subset \supp \pd[\eta]c$.
    \item $\ds \bB_{4} \in \left\{ \frac 1x, D_x \right\} \Psib^{m -2} + \Psib^{m - 1}$.
  \end{itemize}
\end{lemma}

As in \cite{BW20}, bold letters are used to denote nonscalar
operators.

\begin{proof}
  The proof follows \cite[Lemma 26]{BW20} closely.
  Briefly, using the expression for $\dop$ from Section
  \ref{sec:relat-wave-equat},  one can take $B_0 = [\frac{1}{i}\rho
  \pd[\rho] + \frac{1}{i}x\pd[x], C]$ and $[1/x, C] = A_0 (i/x)$.  The lemma
  follows from repeated application of Lemma \ref{lem:comms-with-b-ops}.
 \end{proof}

\subsection{Elliptic regularity}
\label{sec:elliptic-regularity}

The main result of this section is the following proposition
establishing the first part of Theorem~\ref{thm:bulk-prop-near-np}.
\begin{proposition}
  \label{prop:elliptic-bulk-prop}
  If $\psi \in \Hb^{1,-\infty,\gamma}$ is the forward solution of $i \dirac_{\charge/r} \psi
  = g$ with $g\in \CI_{c}$, then in a neighborhood of $\{\rho=0,
  x=0\}$,
  \begin{equation*}
    \WFb^{1,m,\gamma}\psi \subset \Sigmadot
  \end{equation*}
  for all $m$.
\end{proposition}
As $g$ is compactly supported $\psi$ is a solution of $\dirac \psi =
0$ near $\{ \rho = 0, x=0\}$ and so $\WFb^{0,*,*}g$ plays no role
there, though an analogous statement is true with $\WFb^{0,m+1,\gamma+1}g$ in the
correct place.

Proposition~\ref{prop:elliptic-bulk-prop} essentially follows by an
integration by parts argument with several useful consequences.  The
argument is nearly identical to the one given in previous
work~\cite{BW20} and so we omit many of the details.

As the result is local near the pole, we replace $\psi$ by $\chi
\psi$, where $\chi$ is a smooth function localizing to this region.  
By Corollary~\ref{prop:local-H1}, we know that $\chi \psi \in
\Hb^{1,0,1/2-\epsilon}$.  As we may assume $\gamma < 1/2$, we then
have $\rho^{-\gamma}\chi \psi \in \Hb^{1,0,0}$; it therefore suffices
to prove the analogous theorem about $\WFb^{1,m,0}$ for $u \in
\Hb^{1,0,0}$; the difference between the equation satisfied by $\psi$
and that of $u$ is of lower order and can be absorbed into the error
terms in the estimate.

The proof of
Proposition~\ref{prop:elliptic-bulk-prop} is essentially identical to
the proof away from the boundary~\cite[Lemma 28]{BW20} and
will be mostly omitted for brevity.  We include only the main tool in
the proof: the following lemma,
which is proved by considering the real part of the
  pairing $\ang{\Pd A_{\lambda} u,
  A_{\lambda} u}$.
\begin{lemma}
  \label{lem:integrate-by-parts-bulk}
  Suppose that $K \subset U \subset \Sbstar M$ with $K$ compact and
  $U$ open, and suppose that $A_{\lambda}$ are a bounded family of
  invariant elements in $\Psib^{m}$ with $\WFb' A_{\lambda}
  \subset K$ in the sense of uniform wavefront set of families, and
  $A_{\lambda} \in \Psib^{m-1}$ for all $\lambda \in  (0,1)$.
  There exist $G \in \Psib^{m-1/2}$, $\tG \in \Psib^{m}$, both
  microsupported in $U$, and $C_{0}$ so that for all $\epsilon > 0$,
  $\lambda \in (0,1)$, $u \in \Hb^{1,0,0}$ with $\WFb^{1,m-1/2,0}u
  \cap U = \emptyset$, and $\WFb^{0,m,0}(\dop u) \cap U = \emptyset$,
  \begin{align*}
    &\abs{\norm{D_{x}A_{\lambda}u}^{2} +
    \norm{\frac{1}{x}\grad_{\theta}A_{\lambda}u}^{2} - \norm{\left(\rho \pd[\rho] + x
      \pd[x] + 1 + i \charge- \frac{i\charge}{x}\right)A_{\lambda}u}^{2}} \\
    &\quad\quad \leq C_{0} \left( \epsilon
      \norm{A_{\lambda}u}_{\Hb^{1,0,0}}^{2} +
      \norm{u}_{\Hb^{1,0,0}}^{2} + \norm{Gu}_{\Hb^{1,0,0}}^{2} + \frac{1}{\epsilon}\left(
      \norm{\dop u}^{2} + \norm{\tG\dop u}^{2} \right)\right).
  \end{align*}
\end{lemma}

\begin{proof}
  We start by fixing $G, \tG$ of the appropriate order microsupported
  in $U$ and so that the principal symbols of both operators are
  identically $1$ on $K$.

  The pairing
  \begin{equation*}
    \Re \ang{ \Pd A_{\lambda}u, A_{\lambda}u}
  \end{equation*}
  is finite for all $\lambda > 0$ by our wavefront set hypothesis,
  which implies that $PA_{\lambda}u \in \Hb^{-1,0,0}$ and
  $A_{\lambda}u \in \Hb^{1,0,0}$.  We now write
  \begin{equation}
    \label{eq:int-by-parts-pair}
    \abs{\Re \ang{\Pd A_{\lambda}u, A_{\lambda} u}} \leq
    \abs{\ang{[\Pd ,
        A_{\lambda}]u,A_{\lambda}u}} + \abs{\ang{A_{\lambda}\Pd u,A_{\lambda}u}}.
  \end{equation}
  The second term is estimated by
  \begin{equation*}
    \abs{\ang{A_{\lambda}\Pd u, A_{\lambda}u}} \leq
    \norm{A_{\lambda}\Pd
      u}_{\Hb^{-1,0,0}}\norm{A_{\lambda}u}_{\Hb^{1,0,0}} \leq \epsilon
    \norm{A_{\lambda}u}_{\Hb^{1,0,0}}^{2} +
    \epsilon^{-1}\norm{A_{\lambda}\Pd u}_{\Hb^{-1,0,0}}^{2}.
  \end{equation*}
  We now write $\Pd = \tdop_{1}\dop$.  By
  Lemma~\ref{lem:comms-with-b-ops}, $[A_{\lambda}, \tdop_{1}] \in
  \frac{1}{x}\Psib^{s}$ with uniform estimates in this space and so,
  by the elliptic regularity in Lemma~\ref{lem:elliptic-reg-bulk} we
  have the bound
  \begin{equation*}
    \epsilon \norm{A_{\lambda}u}_{\Hb^{1,0,0}}^{2} +
    C\epsilon^{-1}\left(\norm{\dop u}^{2} + \norm{\tG \dop u}^{2}\right).
  \end{equation*}

  We now turn our attention to the commutator term in
  equation~\eqref{eq:int-by-parts-pair}.  By
  Lemma~\ref{lemma:easy-bulk-commutator}, we know
  \begin{equation*}
    [P, A_{\lambda}] \in \left\{ \frac{1}{x^2} \Lap_\theta , D_x^2,
      \frac{1}{x} D_x, \frac{1}{x^2}\right\} \Psib^{m-1} + \left\{
      \frac{1}{x}\beta K, D_x, \frac{1}{x}\right\}\Psib^{m} + \Psib^{m+1},
  \end{equation*}
  and so again by Lemma~\ref{lem:elliptic-reg-bulk}, we have
  \begin{equation*}
    \abs{\ang{[P,A_{\lambda}]u, A_{\lambda}u}} \lesssim
    \norm{Gu}_{\Hb^{1,0,0}}^{2} + \norm{u}_{\Hb^{1,0,0}}^{2}.
  \end{equation*}
\end{proof}

\subsection{Hyperbolic propagation}
\label{sec:hyperb-prop}

As in the previous subsection, we work near the poles as the statement
away from $\rho=0$ is in previous work~\cite[Theorem
22]{BW20}.  

Lemma~\ref{lem:integrate-by-parts-bulk} and the Hardy inequality show
that if $u \in \Hb^{1,0,0}$ and $q_{0}\notin \WFb^{0,m+1}\dop u$, then
\begin{equation*}
  q_{0}\in \WFb^{1,m}u \text{ if and only if }q_{0}\in \WFb^{0,m+1}u.
\end{equation*}

The proof of the hyperbolic part of the estimate near the pole is
similar to setting for finite time and exploits the near-homogeneity
(in $x$) of $\dop$.   We denote by $U$ a neighborhood of $q_{0}$ in
$\Sigmadot$ with
\begin{equation*}
  U \cap \{ \xi/\sigma_{0} > 0\} \cap \WFb^{0,m+1/2}u , \  U \cap
  \WFb^{0,m+1/2}(\dop u) = \emptyset.
\end{equation*}
For our inductive hypothesis, we assume that $q_{0} \notin
\WFb^{0,m}(u)$ and aim to show that $q_{0} \notin \WFb^{0,m+1/2}u$.
We assume for the text below that $\sigma_{0} = 1$; only minor
modifications are needed for $\sigma_{0} = -1$.

Set $\omega = x^{2} + \rho^{2}$ and let
\begin{equation*}
  \phi = - \hat{\xi/\sigma_{0}} + \frac{1}{\beta^{2}\delta}\omega .
\end{equation*}
Fix cutoff functions $\chi_{0}$, $\chi_{1}$, and $\chi_{2}$ so that
\begin{itemize}
\item $\ds\chi_0$ is supported in $[0, \infty)$ with $\chi_{0}(s) =
  \exp (-1/s)$ for $s > 0$,
\item $\ds\chi_{1}$ is supported in $[0,\infty)$ with $\chi_{1}(s)=1$
  for $s\geq 1$ and $\chi_{1}'\geq 0$, and
\item $\ds\chi_{2}$ is supported in $[-2c_{1},2c_{1}]$ and is equal to
  $1$ on $[-c_{1},c_{1}]$.
\end{itemize}
Here $c_{1}$ is chosen so that $\hat{\xi}^{2} + \abs{\hat{\eta}}^{2} <
c_{1}< 2$ in $\Sigmadot \cap U$.
We now set
\begin{equation}
  \label{eq:def-of-a-bulk}
  a =
  \abs{\sigma}^{m+1/2}\chi_{0}(2-\phi/\delta)\chi_{1}(2-\hat{\xi}/\delta)\chi_{2}(\hat{\xi}^{2}
    + \abs{\hat{\eta}}^{2})\mathbf{1}_{\sgn \sigma = \sgn \sigma_{0}},
\end{equation}
and let $A$ be its quantization to an invariant element of
$\Psib^{m+1/2}$.  Note that
\begin{equation*}
  \supp a \subset \left\{ \abs{\hat{\xi}} < 2\delta, \omega <
    4\beta^{2}\delta^{2}\right\},
\end{equation*}
and so the support of $a$ in $\Tbstar M$ can be made arbitrarily close
to $q_{0}$.

The following lemma is proved by a careful application of Lemma~\ref{lemma:careful-bulk-commutator-general-new}.
\begin{lemma}
  \label{lem:careful-bulk-comm-detail-with-a}
  For $A$ defined as above,
  \begin{equation*}
    \frac{1}{i}\left[\dop, A^{*}A \right] = \tR\dop +
    \sgn(\sigma_{0})Q^{*}Q + \bR  + B_{0} + \alpha_{r}B_{1} + E' + E''
  \end{equation*}
  where
  \begin{itemize}
  \item $\ds Q\in \Psib^{m+1/2}$ is invariant and self-adjoint with
    \begin{equation*}
      \sigmab(Q) = \sqrt{2}\abs{\sigma}^{m}\abs{\sigma +
        \xi}^{1/2}\delta^{-1/2}(\chi_{0}'\chi_{0})^{1/2}\chi_{1}\chi_{2}\mathbf{1}_{\sgn
      \sigma = \sgn \sigma_{0}},
    \end{equation*}
  \item $\ds\tR \in \Psib^{2m}$,
  \item $\ds \bR \in \left\{ \frac{1}{x}\beta K , D_x,
      \frac{1}{x}\right\} \Psib^{2m-1} + \Psib^{2m}$,
  \item $\ds B_{0},B_{1}\in \Psib^{2m+1}$ with
    $\abs{\sigmab(B_{\bullet})}$ equal to an order $0$ symbol times
    $C\beta^{-1}\sigmab(Q)^{2}$,
  \item $\ds E'\in \Psib^{2m+1}$ with $\WFb'E' \subset \{ \delta \leq
    \hat{\xi}\leq 2\delta, \omega \leq 4 \beta^{2}\delta^{2}\}$, and
  \item $\ds E''\in \frac{1}{x} \Psib^{2m+1} + D_x\Psib^{2m} + \Psib^{2m+1}$ with
    $\WFb'E''\cap \Sigmadot = \emptyset$.
  \end{itemize}
\end{lemma}

\begin{proof}
  We carefully apply
  Lemma~\ref{lemma:careful-bulk-commutator-general-new} and use its
  notation.  The term $A_{0}$ arising there has principal symbol
  $-\pd[\xi](a^{2})$ and arises from $\dop$ being nearly homogeneous
  in $r$ of degree $-1$.  We rewrite the $A_{0}$ term
  in~\eqref{eq:list-of-all-terms-bulk-commutator} as $A_{0}(\dop + i
  \rho\pd[\rho] + i x\pd[x])$, modulo $A_{0}$ times smooth lower-order
  terms (which are absorbed into $\bR$).  We then split the symbol of
  $A_{0}$ into three terms: those terms where the $\xi$ derivative
  falls on $\chi_{0}$ can be written in the form
  $\tQ^{2}(i\rho\pd[\rho] + ix\pd[x])$, which we write as the product
  of $\sgn(\tau_{0})$ times squares $Q^{2}$ modulo a lower-order term
  to be absorbed into $\bR$.  Those terms where the derivative falls
  on $\chi_{1}$ are absorbed into $E'$ and those where it falls on
  $\chi_{2}$ form part of $E''$.  Thus, modulo further commutators (to
  be absorbed into $\bR$), the first term on the right side
  of~\eqref{eq:list-of-all-terms-bulk-commutator} is given by $\tR
  \dop + \sgn(\sigma_{0})Q^{*}Q$.

  The $B_{1}$ term satisfies the stated symbol bound because $x$
  derivatives on $a^{2}$ may fall only on the $\chi_{0}$ term, giving
  \begin{equation*}
    2\abs{\sigma}^{2m+1}(\chi_{0}'\chi_{0})\chi_{1}^{2}\chi_{2}^{2}(-2x)(\beta^{-2}\delta^{-2}).
  \end{equation*}
  As $0 \leq x \leq 2\beta\delta$ on the support of $a$, this term can
  be estimated by a multiple of
  \begin{equation*}
  \beta^{-1}\delta^{-1}\abs{\sigma}^{2m+1}\chi_{0}'\chi_{0}\chi_{1}^{2}\chi_{2}^{2},
\end{equation*}
  which is a constant multiple of $\frac{\abs{\tau}}{\abs{\tau +
      \xi}}\sigmab(Q)^{2}$.  This prefactor is a symbol of order zero
  (provided the neighborhood $U$ is small enough).  Likewise, the
  $B_{0}$ term in the
  Lemma~\ref{lemma:careful-bulk-commutator-general-new} becomes the $B_0$
  term here and is estimated similarly, as the $\rho$ derivative may
  also only hit the $\chi_{0}$ term.

  Finally, the remaining $\bB_{2}$ term in
  Lemma~\ref{lemma:careful-bulk-commutator-general-new} is proportional to
  $\pd[\eta](a^{2})$ and therefore the derivative must fall on
  $\chi_{2}$ and so these terms are absorbed into $E''$.
\end{proof}

The rest of the proof of Theorem~\ref{thm:bulk-prop-near-np} is a
positive commutator estimate.  We pair $\frac{1}{i}[\dop, A^{*}A]u$
with $u$ and regularize as in the elliptic setting.  On the one hand,
as $\dop^{*} = \dop - i$, we may bound
\begin{equation*}
  \abs{\ang{[\dop, A^{*}A]u, u}} \leq  2\norm{Au}\norm{A\dop u} + \norm{Au}^{2} \leq
  2 \norm{Au}^{2} + \norm{A\dop u}^{2},
\end{equation*}
while on the other hand we apply
Lemma~\ref{lem:careful-bulk-comm-detail-with-a}.

The main term is $\sgn(\sigma_{0})\ang{Q^{*}Qu,u} =
\sgn(\sigma_{0})\norm{Qu}^{2}$, which has a definite sign.  We then
bound
\begin{align*}
  \norm{Qu}^{2} \leq 2\norm{Au}^{2} + \norm{A\dop
  u}^{2} + \abs{\ang{\tR\dop u,u}} + \abs{\ang{\bR u,u}} +
  \abs{\ang{B_{0}u,u}} + \abs{\ang{\alpha_{r}B_{1}u,u}} +
  \abs{\ang{E'u,u}} + \abs{\ang{E''u,u}}.
\end{align*}
The $\tR$ term is bounded by $\norm{G_{m}\dop u}\norm{G_{m}u}$ for
some $G_{m}\in \Psib^{m}$, while the $\bR$ term can be estimated by
$\norm{G_{m-1}u}_{\Hb^{1,0,0}}\norm{G_{m}u}$ for $G_{m-1}\in
\Psib^{m-1}$ and $G_{m}\in \Psib^{m}$.  This leaves the terms
involving $B_{0}$, $B_{1}$, $E'$, and $E''$. 

The terms involving $B_{0}$ and $B_{1}$ are estimated by the symbol
calculus, i.e.,
\begin{equation*}
  \abs{\ang{B_{j}u,u}} \leq C \beta^{-1}\norm{Qu}^{2} + C\norm{Gu}^{2}
  + C\norm{u}_{\Hb^{1,0,0}}^{2}
\end{equation*}
for some $G\in \Psib^{m}$ elliptic on the support of $B_{j}$.

The term involving $E'$ is bounded by $\norm{G_{m+1/2}u}^{2}$, where
$G_{m+1/2}\in \Psib^{2+1/2}$ has $\WFb'G_{m+1/2}\subset \{\delta \leq
\hat{\xi}\leq 2\delta, \omega \leq 4\beta^{2}\delta^{2}\}$.  The
hypothesis that $U\cap \{\xi > 0\} \cap \WFb^{1,m-1/2}u = \emptyset$
implies that this term is finite.

Finally, we turn to the term involving $E''$.  The microsupport of
$E''$ is contained in the elliptic set of $\dop$, so we may use
elliptic regularity to bound this term by
\begin{equation*}
  C\left( \norm{G_{m-1}u}_{\Hb^{1,0,0}}^{2} + \norm{G_{m}\dop u}^{2}
    + \norm{u}_{\Hb^{1,0,0}}^{2}\right),
\end{equation*}
where $G_{r}\in \Psib^{r}$ are microsupported in the elliptic region
within $U$.

As $\sigmab(A)$ is bounded by a small multiple of $\sigmab(Q)$, we
then know that $\norm{Qu}^{2}$ is finite; since $Q$ is elliptic at
$q_{0}$, we know that $q_{0} \notin \WFb^{0,m+1/2}u$ (and hence not in
$\WFb^{1,m-1/2}u$), finishing the proof.


\end{document}